\documentclass[notitlepage,11pt,reqno]{amsart}
\usepackage[foot]{amsaddr}
\usepackage{amssymb,nicefrac,bm,upgreek,mathtools,verbatim}
\usepackage[final]{hyperref}
\usepackage[mathscr]{eucal}
\usepackage{dsfont}
\usepackage[normalem]{ulem}
\usepackage{amsopn}

\usepackage[margin=1in]{geometry}
\allowdisplaybreaks
\newcommand{\stkout}[1]{\ifmmode\text{\sout{\ensuremath{#1}}}\else\sout{#1}\fi}

\newtheorem{lemma}{Lemma}[section]
\newtheorem{theorem}{Theorem}[section]

\newtheorem{corollary}{Corollary}[section]

\theoremstyle{definition}
\newtheorem{definition}{Definition}[section]

\newtheorem{example}{Example}[section]

\theoremstyle{remark}
\newtheorem{remark}{Remark}[section]
\numberwithin{theorem}{section}
\numberwithin{equation}{section}

\hypersetup{
  colorlinks=true,
  citecolor=mblue,
  linkcolor=mblue,
  urlcolor = blue,
  anchorcolor = blue,
  frenchlinks=false,
  pdfborder={0 0 0},
  naturalnames=false,
  hypertexnames=false,
  breaklinks}
\usepackage[capitalize,nameinlink]{cleveref}
\usepackage[abbrev,msc-links,nobysame,citation-order]{amsrefs} 

\crefname{section}{Section}{Sections}
\crefname{subsection}{Section}{Sections}
\crefname{condition}{Condition}{Conditions}
\crefname{hypothesis}{Hypothesis}{Conditions}
\crefname{assumption}{Assumption}{Assumptions}
\crefname{lemma}{Lemma}{Lemmas}
\crefname{fact}{Fact}{Facts}

\Crefname{figure}{Figure}{Figures}

\crefformat{equation}{\textup{#2(#1)#3}}
\crefrangeformat{equation}{\textup{#3(#1)#4--#5(#2)#6}}
\crefmultiformat{equation}{\textup{#2(#1)#3}}{ and \textup{#2(#1)#3}}
{, \textup{#2(#1)#3}}{, and \textup{#2(#1)#3}}
\crefrangemultiformat{equation}{\textup{#3(#1)#4--#5(#2)#6}}%
{ and \textup{#3(#1)#4--#5(#2)#6}}{, \textup{#3(#1)#4--#5(#2)#6}}%
{, and \textup{#3(#1)#4--#5(#2)#6}}

\Crefformat{equation}{#2Equation~\textup{(#1)}#3}
\Crefrangeformat{equation}{Equations~\textup{#3(#1)#4--#5(#2)#6}}
\Crefmultiformat{equation}{Equations~\textup{#2(#1)#3}}{ and \textup{#2(#1)#3}}
{, \textup{#2(#1)#3}}{, and \textup{#2(#1)#3}}
\Crefrangemultiformat{equation}{Equations~\textup{#3(#1)#4--#5(#2)#6}}%
{ and \textup{#3(#1)#4--#5(#2)#6}}{, \textup{#3(#1)#4--#5(#2)#6}}%
{, and \textup{#3(#1)#4--#5(#2)#6}}

\crefdefaultlabelformat{#2\textup{#1}#3}
\newcommand{\vertiii}[1]{{\left\vert\kern-0.25ex\left\vert\kern-0.25ex\left\vert #1 
    \right\vert\kern-0.25ex\right\vert\kern-0.25ex\right\vert}}

\newcommand{\Uadm}{\mathfrak U}
\newcommand{\Act}{\mathbb{U}}
\newcommand{\Usm}{\mathfrak U_{\mathsf{sm}}}

\newcommand{\Um}{\mathfrak U_{\mathsf{m}}}

\newcommand{\pV}{\mathrm{V}} 
\newcommand{\pv}{\mathrm{v}} 

\newcommand{\cB}{{\sB}}  
\newcommand{\sB}{{\mathscr{B}}}  
\newcommand{\cC}{{\mathcal{C}}}   
\newcommand{\sE}{{\mathscr{E}}} 
\newcommand{\sF}{{\mathfrak{F}}}   
\newcommand{\cJ}{{\mathcal{J}}}  
  
\newcommand{\sL}{{\mathscr{L}}}  %
\newcommand{\Lp}{{L}}            

\newcommand{\cP}{{\mathcal{P}}}  

\newcommand{\Lyap}{{\mathcal{V}}}  

\newcommand{\RR}{\mathds{R}}
\newcommand{\NN}{\mathds{N}}

\newcommand{\Rd}{{\mathds{R}^{d}}}
\DeclareMathOperator{\Exp}{\mathbb{E}}

\newcommand{\D}{\mathrm{d}}

\newcommand{\Ind}{\mathds{1}}   
\newcommand{\cD}{\mathcal{D}} 
 
\newcommand{\Sob}{{\mathscr W}}    
\newcommand{\Sobl}{{\mathscr W}_{\text{loc}}} 

\newcommand{\df}{:=}
\newcommand{\transp}{^{\mathsf{T}}}
\DeclareMathOperator*{\osc}{osc}

\DeclareMathOperator*{\trace}{Tr}

\newcommand{\sorder}{{\mathfrak{o}}}
\newcommand{\grad}{\nabla}
\newcommand{\uuptau}{{\Breve\uptau}}

\newcommand{\abs}[1]{\lvert#1\rvert}
\newcommand{\norm}[1]{\lVert#1\rVert}

\usepackage{color}
\definecolor{dmagenta}{rgb}{.4,.1,.5}
\definecolor{dblue}{rgb}{.0,.0,.5}
\definecolor{mblue}{rgb}{.0,.0,.7}
\definecolor{ddblue}{rgb}{.0,.0,.4}
\definecolor{dred}{rgb}{.7,.0,.0}
\definecolor{dgreen}{rgb}{.0,.5,.0}
\definecolor{Eeom}{rgb}{.0,.0,.5}

\begin{document}
\title[Robustness of Controlled  regime-switching diffusion Processes]
{Robustness of Stochastic Optimal Control for Controlled  regime-switching diffusions with Incorrect Models}

\author[Somnath Pradhan]{Somnath Pradhan$^\dag$}
\address{$^\dag$Department of Mathematics, Indian Institute of Science Education and Research Bhopal,
Bhopal, MP - 462066, India}
\email{somnath@iiserb.ac.in}

\author[Dinesh Rathia]{Dinesh Rathia$^{\ddag}$}
\address{$^\ddag$Department of Mathematics, Indian Institute of Science Education and Research Bhopal,
Bhopal, MP - 462066, India}
\email{dinesh23@iiserb.ac.in}





\begin{abstract}
This paper investigates the robustness of stochastic optimal control for controlled regime-switching diffusions. We consider systems driven by both continuous fluctuations and discrete regime changes, allowing for model misspecification in both the diffusion and switching components. Within a unified framework, we study four classical cost formulations---finite-horizon, infinite-horizon discounted and ergodic costs, and the exit-time cost, and establish continuity of value functions and robustness of optimal controls. Specifically, we show that as a sequence of approximating regime-switching models converges to the true model, the associated value functions and optimal policies converge as well, ensuring vanishing performance loss. The analysis relies on the regularity of the solution to the associated weakly coupled HJB systems, and their stochastic representation (obtained via It$\hat{\rm o}$–Krylov formula). The results extend the robustness framework developed for diffusion processes to a significantly broader class of hybrid systems with interacting continuous and discrete dynamics.
\end{abstract}

\maketitle

\section{Introduction} A stochastic optimal control problem addresses decision-making under uncertainty, 
where the goal is to minimize or maximize a specified performance criterion. 
Typically, the system dynamics are described by a stochastic differential equation (SDE) 
driven by Brownian motion, and the controller seeks to minimize an expected cost functional 
over a given time horizon. 
The decision-maker must determine a control strategy referred to as a stochastic control  
that optimizes the objective in this random environment. A fundamental analytical approach to such problems is the \emph{dynamic programming principle} 
and the associated Hamilton--Jacobi--Bellman (HJB) equation, 
whose solution characterizes the optimal control policy \cites{FlemingSoner2006,Bor-book,KushnerDupuis2001,ABG-book,HP09-book,yong2012stochastic}.

In practical applications, one typically works with an idealized system model, often constructed from available incomplete or noisy data. A control strategy is then designed based on this approximate model; however, when implemented on the true system, it may not perform as expected and give rise to a performance loss due to the mismatch between the assumed model and the true model. This naturally gives rise to the question of \emph{robustness}: does the performance loss vanish as the model used for design becomes increasingly accurate? Equivalently, as the approximate model converges to the true system, 
do the corresponding value functions and optimal controls also converge?

Formally, let the true model be denoted by $(X,c)$, 
where $X$ describes the system dynamics, 
and $c$ is the running cost function. 
Let $(X_n, c_n)$ represent a sequence of approximating models converging to $(X,c)$ as $n\to\infty$ 
in an appropriate sense. 
For each admissible control policy $U$, denote the total expected cost by 
$\mathcal{J}(c,U)$ under the true model and by $\mathcal{J}_n(c_n,U)$ under the approximate model. 
If $v^*$ and $v_n^*$ denote the corresponding optimal controls, 
the robustness problem is to show that the performance loss
\(
|\mathcal{J}(c,v^*) - \mathcal{J}(c,v_n^*)|\)
vanishes as $n\to\infty$, i.e., \(
|\mathcal{J}(c,v^*) - \mathcal{J}(c,v_n^*)| \longrightarrow 0.\)

This problem is of significant theoretical and practical importance and has been 
studied extensively in the discrete-time setting within the framework of 
Markov decision processes (MDPs); see, for example, \cites{KY-20,KRY-20,BJP02,KV16,SX15,NG05} 
and the references therein. In contrast, the continuous-time setting has received relatively limited 
attention; see, for instance, \cites{HS08,GL99,LJE15}. Motivated by these developments, 
Pradhan and Yüksel~\cite{PradhanYuksel2023} extended the robustness analysis to continuous-time 
controlled diffusion processes, establishing detailed results under several cost criteria including finite horizon, infinite horizon discounted and ergodic costs, and exit-time formulations. To the best 
of our knowledge, a corresponding robustness theory for hybrid models such as regime-switching diffusions 
has not yet been developed, and the present work aims to fill this gap.

A rich literature exists on controlled diffusion processes and their optimality properties 
\cites{FlemingSoner2006,Bor-book,KushnerDupuis2001,HP09-book,ABG-book}, 
as well as on robust and risk-sensitive extensions 
\cites{GL99,LJE15,arapostathis2022risk}. 
However, purely diffusion-based models often fail to capture systems that exhibit 
both continuous fluctuations and discrete regime changes. 
Such systems are more naturally represented by 
\emph{regime-switching diffusions}, in which the continuous dynamics are influenced 
by a finite-state Markov chain. 
These hybrid models have been extensively investigated 
\cites{YinZhu2010,MaoYuan2006}, 
owing to their ability to describe realistic systems arising in finance, 
engineering, manufacturing, communications, and biological processes. 
Classical results in this area include \cites{Ghosh97,AGM93}, 
while more recent studies focus on optimal control problems for partially observed regime-switching diffusions under discounted and ergodic performance criteria, extending dynamic programming methods to systems with hidden Markovian switching \cites{Escobedo}. 
Despite this progress, robustness analysis for such hybrid models remains an open challenge, 
as the associated HJB equations form a weakly coupled system of nonlinear PDEs 
that must account simultaneously for both diffusion and switching dynamics.


In this work, we investigate the continuity and robustness properties of optimal controls 
for controlled regime-switching diffusions under four important cost evaluation criteria: 
the \emph{finite-horizon}, \emph{ infinite-horizon discounted and ergodic (long-run average) costs},  and \emph{exit-time} costs 
(see, e.g., \cites{FH,AGM93,Ghosh97}). 
Our goal is to extend the diffusion-based robustness framework 
of~\cite{PradhanYuksel2023} to this hybrid setting and to establish 
continuity of value functions and robustness of optimal controls.

The study of regime-switching diffusions is motivated by the need to model systems 
that experience both continuous fluctuations and abrupt regime changes. 
Such hybrid dynamics arise in diverse applications including finance, 
engineering, and networked systems. Unlike existing works, we allow for 
model misspecification in both the diffusion and switching components, 
which introduces new analytical challenges. In particular, the associated 
Hamilton–Jacobi–Bellman (HJB) equations form a weakly coupled system of 
nonlinear partial differential equations (PDEs) rather than a single equation. 
This coupling necessitates refined continuity arguments adapted to the hybrid setting, 
regularity results for PDE systems, and new techniques to establish robustness 
when both continuous dynamics and discrete regime transitions interact.


\medskip
\noindent
\textbf{Contributions.} 
The main contributions of this paper are summarized as follows:
\begin{enumerate}
    \item \textbf{Continuity of value functions:} 
    We establish that the value functions associated with a sequence of 
    approximating regime-switching diffusion models converge to those of the true model 
    under mild structural and regularity conditions on the system coefficients.
    
    \item \textbf{Robustness of optimal controls:} 
    We prove that optimal controls derived from approximate regime-switching 
    diffusion models remain asymptotically optimal when applied to the true model. 
    This extends the diffusion-based robustness results of 
    \cite{PradhanYuksel2023} to a broader class of hybrid systems 
    involving both continuous and discrete dynamics.
\end{enumerate}
The continuity results obtained in Theorems \ref{TH2.3}, \ref{TH3.3}, \ref{thm:finite-horizon-conv} and \ref{thm:exit-continuity} also ensure the stability of the solution of associated semi-linear parabolic/elliptic PDEs (see e.g \cites{WLS01,SI72}). 
Furthermore, as corollaries, we show that:
\begin{itemize}
\item $\epsilon$-optimal controls for the approximating models remain $3\epsilon$-optimal in the true model, and
\item the robustness results remain valid when the noise is approximated by an It$\mathrm{\hat o}$ process.
\end{itemize}


\medskip
\noindent
\textbf{Paper organization.}
Section~2 introduces the controlled regime-switching diffusion model 
and the associated cost criteria. 
Section~3 develops the continuity and robustness theory for the discounted cost problem. 
Section~4 treats the ergodic (long-run average) case, 
while Section~5 studies the finite-horizon formulation, Section~6 addresses the exit-time cost problem and presents two important concluding corollaries, and Section~7 presents an example.
Finally, Section~8 concludes the paper with a discussion of applications 
and potential directions for future research.

\section{Description of the problem}\label{PD} Let $\Act$ be a compact metric space of control actions and $\pV=\mathscr{P}(\Act)$ be the space of probability measures on  $\Act$ endowed with the topology of weak convergence. 
Consider the controlled  regime-switching diffusion process $(X_t, S_t)$ taking values in $\Rd \times \mathbb{S}$, where $\mathbb{S} = \{1, \dots, N\}$ denotes the finite set of regimes, defined on a complete probability space \((\Omega,{\sF},\mathbb{P})\). The following stochastic differential equations describe the dynamics of $(X_t, S_t)$
\begin{equation}\label{E1.1}
\begin{aligned}
dX_t &= b(X_t, S_t, U_t)\,dt + \upsigma(X_t, S_t)\,dW_t,\\
dS_t &= \int_{\RR} h(X_t, S_{t-}, U_t, z)\, \mathcal{P}(dt, dz)
\end{aligned}
\end{equation}

Where

\begin{itemize}
\item Initial distribution of $S$ is $S_0$. 
    
    \item Initial distribution of $X$ is $X_0$.
    \item The functions, $b=[b_1,....,b_d]^T:\Rd \times \mathbb{S} \times \Act \to  \Rd$ is the drift vector, $ \upsigma = [\upsigma_{ij}(\cdot,\cdot)]_{1\leq i,j\leq d}:\Rd \times \mathbb{S} \to  \RR^{d\times d}$ is the diffusion matrix.

    \item $W$ is a $d$-dimensional standard Wiener process.
    
    \item $\mathcal{P}(dt, dz)$ is a Poisson random measure on $\RR_+ \times \RR$ with intensity $dt \times m(dz)$, where $m$ is the Lebesgue measure on $\RR$.
    
    \item $\mathcal{P}(\cdot, \cdot)$, $W(\cdot)$, $X_0$, $S_0$ are independent.
    
    \item The function $h : \Rd \times \mathbb{S} \times \Act \times \RR \to \RR$ is defined by
    \[
        h(x, i, \zeta, z) := 
        \begin{cases}
            j - i & \text{if } z \in \triangle_{i,j}(x, \zeta), \\
            0 & \text{otherwise},
        \end{cases}
    \]
    where for $i, j \in \mathbb{S}$ and fixed $x, \zeta$, the sets $\triangle_{i,j}(x, \zeta)$ are left-closed, right-open disjoint intervals of $\RR$ having length $m_{ij}(x, \zeta)$.
    
    \item $\mathbb{M} := (m_{ij})_{i,j \in \mathbb{S}}$ is the transition matrix of the controlled Markov chain $S_t$, where $m_{ij}:\Rd \times \mathbb{S}\to \RR$ are the switching rates such that $m_{ij} \geq 0$ if $i \neq j$ and, $\sum_{j=1}^{N} m_{ij} = 0$ for all $i \in \mathbb{S}$, 
    
    \item The control process $\{U_t\}$ takes values in a compact metric space $\Act$, is progressively measurable with respect to $\sF_t := \text{completion of } \sigma\{X_s, S_s; s \leq t\}$ relative to $(\sF, \mathbb{P})$, and is non-anticipative: for each $t \geq 0$, the $\sigma$-field $\sigma\{U_s\,;\, s \leq t\}$ is independent of 
    \[
        \sigma\{W_s - W_t,\, \cP(A, B) : A \in \mathcal{B}([s, \infty)), B \in \mathcal{B} (\RR), s \geq t\}.
    \]
    The process $U$ is called an \textit{admissible control}, and the set of all admissible controls is denoted by $\Uadm$ (see, \cite{ABG-book}*{Chapter 5, p. 197}).
\item 
 We extend the drift term $b : \Rd \times \mathbb{S} \times \pV \to  \Rd$ as follows:
\begin{equation*}
b (x,i,\mathrm{v}) = \int_{\Act} b(x,i,\zeta)\mathrm{v}(\D \zeta), 
\end{equation*}
for $\mathrm{v}\in\pV$.
\end{itemize}

To guarantee the existence and uniqueness of strong solutions to \cref{E1.1}, we impose a set of structural assumptions on the drift vector 
$b$, the diffusion matrix 
$\upsigma$, and the transition rate matrix 
$\mathbb{M}$. 
\subsection{\bf Assumptions:}
\begin{itemize}
\item[\hypertarget{A1}{{(A1)}}]
\emph{Local Lipschitz continuity:\/}
The functions $b(x,i,\zeta),\upsigma^{ij}(x,k),\,m_{ij}(x,\zeta)$,\,\,
are continuous and locally Lipschitz continuous in $x$ (uniformly with respect $\zeta$) with a Lipschitz constant $C_{R}>0$
depending on $R>0$, i.e.,
\begin{equation*}
\abs{b(x,i,\zeta) - b(y,i, \zeta)}^2 + \norm{\upsigma(x,i) - \upsigma(y,i)}^2\,+\,|m_{ij}(x,\zeta)-m_{ij}(y,\zeta)|^2 \,\le\, C_{R}\,\abs{x-y}^2
\end{equation*}
for all $x,y\in \sB_R\,,\,i,j \in \mathbb{S}$ and $\zeta\in\Act$, where $\norm{\upsigma}\df\sqrt{\trace(\upsigma\upsigma\transp)}$\,. 

\medskip
\item[\hypertarget{A2}{{(A2)}}]
\emph{Affine growth condition:\/} The drift term
$b$ and the diffusion coefficient $\upsigma$ satisfy a global growth condition of the form
\begin{equation*}
\sup_{\zeta\in\Act}\, \langle b(x,i, \zeta),x\rangle^{+} + \norm{\upsigma(x,i)}^{2} \,\le\,C_0 \bigl(1 + \abs{x}^{2}\bigr) 
\end{equation*}
for all $x\in\Rd,\,i \in \mathbb{S}$ and for some constant $C_0>0$.

\medskip
\item[\hypertarget{A3}{{(A3)}}]
\emph{Nondegeneracy:\/}
For each $R>0$, it holds that
\begin{equation*}
\sum_{i,j=1}^{d} a^{ij}(x,k)z_{i}z_{j}
\,\ge\,C^{-1}_{R} \abs{z}^{2} \qquad\forall\, x\in \sB_{R}\,, k\in \mathbb{S},
\end{equation*}
and for all $z=(z_{1},\dotsc,z_{d})\transp\in\Rd$,
where $a\df \frac{1}{2}\upsigma \upsigma\transp$.
\end{itemize}
 Under these assumptions \hyperlink{A1}{{(A1)}}--\hyperlink{A3}{{(A3)}}, the system \cref{E1.1}
admits a unique, strong solution  for every admissible control (see, for example, \cite{ABG-book}*{p. 197} and \cite{Yin10}*{Theorem 3.10}), with 
\[
X \in \cC(\RR_+; \Rd), \quad S \in\cD(\RR_+; \mathbb{S}),
\]
where \( \cD(\RR_+; \mathbb{S}) \) is the space of all right-continuous functions from \( \RR_+ \) to \( \mathbb{S} \) having left limits.

The ergodic behavior of the joint process $Y_t := (X_t, S_t)$ depends strongly on the coupling
coefficients $\{m_{ij}\}$. For this, we define the matrix
\[
\widetilde{\mathbb{M}}(x, \zeta) := (\widetilde{m}_{ij}(x, \zeta)) : \Rd \times  \Act \to \RR^{N \times N},
\]
where
\[
\widetilde{m}_{ij}(x, \zeta) := 
\begin{cases}
m_{ij}(x, \zeta), & \text{if } i \ne j, \\
0, & \text{otherwise}.
\end{cases}
\]

In addition to the usual structural assumptions \hyperlink{A1}{{(A1)}}--\hyperlink{A3}{{(A3)}}, we impose the following condition:

\vspace{1em}
\noindent
\begin{itemize}
\item[\hypertarget{A4}{{(A4)}}]
\emph{Irreducibility:\/}
The matrix \( \breve{\mathbb{M}}(x) := (\breve{m}_{ij}(x)) \), where 
\[
\breve{m}_{ij}(x) := \min_{\zeta \in  \Act} \widetilde{m}_{ij}(x, \zeta),
\]
is irreducible in \( \Rd \), that is, for every nonempty disjoint sets  \( \,\mathbb{S}_1, \mathbb{S}_2 \subset \mathbb{S} \) satisfying \( \mathbb{S}_1 \cup\, \mathbb{S}_2 = \mathbb{S} \), there exist \( i_0 \in \mathbb{S}_1 \) and \( j_0 \in \mathbb{S}_2 \) such that
\[
\big|\{ x \in \Rd : \breve{m}_{i_0 j_0}(x) > 0 \} \big| > 0,
\]
where \( | \cdot | \) denotes the Lebesgue measure. 
\end{itemize}
We assume $\norm{m_{ij}}_{\infty}\le M$ throughout this article.

In this article, we consider the problem of minimizing discounted, finite horizon, exit-time, and ergodic cost criteria.

Let $c\colon\Rd\times \mathbb{S} \times \Act \to \RR_+$ be the \emph{running cost} function. We assume that 
\begin{itemize}
\item[\hypertarget{A5}{{(A5)}}]
The \emph{running cost} $c$ is bounded (i.e., there exist $M>0$ such that $\|c\|_{\infty} \leq M$), continuous and locally Lipschitz continuous in $x$ uniformly with respect to $\zeta\in\Act$.
\end{itemize}
 We extend $c\colon\Rd\times \mathbb{S}\times\pV \to\RR_+$ as follows: for $\pv \in \pV$
\begin{equation*}
c(x,i,\pv) := \int_{\Act}c(x,i,\zeta)\pv(\D\zeta)\,.
\end{equation*}

\subsection{Cost criteria:}\label{Cost} The following costs will be considered in this article.

\vspace{2mm}
{\bf Discounted cost criterion.} 
For any admissible control $U \in\Uadm$, the associated \emph{$\alpha$-discounted cost} is defined by
\begin{equation}\label{EDiscost}
\cJ_{\alpha}^{U}(x,i, c) \,\df\, \Exp_{x,i}^{U} \left[\int_0^{\infty} e^{-\alpha t} c(X_s,S_s, U_s) \D s\right],\quad (x,i)\in\Rd \times \mathbb{S}\,,
\end{equation} where $\alpha > 0$ is the discount factor, $(X_{(\cdot)},S_{(\cdot)})$ is the solution of the controlled system  \cref{E1.1} under $U \in\Uadm$, and $\Exp_{x,i}^{U}$ denotes the expectation with respect to the law of the process $(X_{(\cdot)},S_{(\cdot)})$ with the initial condition $(x,i)$. The control objective is to minimize the cost in~\eqref{EDiscost} over all admissible controls. A control $U^{*}\in\Uadm$ is said to be \emph{optimal} if, for every $(x,i)\in\Rd\times\mathbb{S}$,
\begin{equation}\label{OPDcost}
\cJ_{\alpha}^{U^*}(x,i, c) = \inf_{U\in \Uadm}\cJ_{\alpha}^{U}(x,i, c) \,\,\, (\,=:\, \,\, V_{\alpha}(x,i))\,,
\end{equation} where $V_{\alpha}(x,i)$ is called the $\alpha$-discounted optimal value function.

\vspace{2mm}
{\bf Ergodic cost criterion.} 
For a control $U \in \Uadm$, the corresponding \emph{ergodic cost functional} is defined as
\begin{equation*}\label{ECost1}
\sE_{x,i}(c,U) = \limsup_{T\to \infty}\frac{1}{T}\Exp_{x,i}^{U}\left[\int_0^{T} c(X_s,S_s, U_s) \D{s}\right]\,,\qquad (x,i)\in\Rd\times\mathbb{S}
\end{equation*} and the optimal value is defined as
\begin{equation*}\label{ECost1Opt}
\sE^*(c) \,\df\, \inf_{(x,i)\in\Rd\times \mathbb{S}}\,\inf_{U\in \Uadm}\sE_{x,i}(c, U)\,.
\end{equation*}
Then a control $U^*\in \Uadm$ is said to be optimal if we have 
\begin{equation*}\label{ECost1Opt1}
\sE_{x,i}(c, U^*) = \sE^*(c)\,.
\end{equation*}

{\bf Finite horizon cost.}  For any $U\in \Uadm$, the associated \emph{finite horizon cost} is given by
\begin{equation*}\label{FiniteCost1}
\cJ_{T}^U(x,i,c) = \Exp_{x,i}^{U}\left[\int_0^{T} c(X_s,S_s,U_s) \D{s} + c_{_{T}}(X_T,S_T)\right]\,,
\end{equation*} where $c_{_{T}}(\cdot,\cdot)$ is the terminal cost. The optimal value is defined as
\begin{equation*}\label{FiniteCost1Opt}
\cJ_{T}^*(x,i,c) \,\df\, \inf_{U\in \Uadm}\cJ_{T}^U(x,i,c)\,.
\end{equation*}
Thus, a policy $U^*\in \Uadm$ is said to be (finite horizon) optimal if we have 
\begin{equation*}\label{FiniteCost1Opt1}
\cJ_{T}^{U^*}(x,i,c) = \cJ_{T}^*(x,i,c)\quad \text{for all}\,\,\, (x,i)\in \Rd\times\mathbb{S},.
\end{equation*}
We assume the terminal function $c_{_{T}}$ satisfies $c_{_{T}} \in \Sob^{2,p}(\Rd\times \mathbb{S}) \cap L^\infty(\Rd\times \mathbb{S})$ for some $p \ge 2$ throughout this paper.
\vspace{2mm}

{\bf Cost up to an exit time.} For each $U\in\Uadm$, the associated exit time cost is defined as
 \[
    \hat{\cJ}^U_e(x,i) := \Exp_{x,i}^U\!\left[
        \int_0^{\uptau(\mathcal{O})} e^{-\int_0^t \beta(X_s,S_s,U_s)\,ds}\, c(X_t,S_t,U_t)\,dt
        + e^{-\int_0^{\uptau(\mathcal{O})} \beta(X_s,S_s,U_s)\,ds}\, h(X_{\uptau(\mathcal{O})},S_{\uptau(\mathcal{O})})
    \right]
   \]
where $\mathcal{O} \subset \Rd$ is a bounded domain, $\beta(\cdot,\cdot,\cdot): \bar{\mathcal{O}}\times\mathbb{S} \times \Act \to [0,\infty)$ is the discount function, and 
$h:\bar{\mathcal{O}}\times\mathbb{S}\to \RR_+$ is the terminal cost function. The optimal value is defined as
    \begin{equation*}
         \hat{\cJ}^{*}_e(x,i)=\inf_{U\in \Uadm}\hat{\cJ}^{U}_e(x,i)
    \end{equation*}
    and a control $U^*\in \Uadm$ is said to be optimal if we have 
\begin{equation*}
     \hat{\cJ}^{U^*}_e(x,i)=\hat{\cJ}^{*}_e(x,i)=\inf_{U\in \Uadm}\hat{\cJ}^{U}_e(x,i)
\end{equation*}

\begin{definition}(\textit{Markov control}:)
    An admissible control is called a \emph{Markov control} if it is of the form \(U_t\,=\, v(t, X_t, S_t) \), for some Borel measurable function $v : \RR_+ \times \Rd \times \mathbb{S} \to  \Act$.\\ We  denote $\Um$ as the space of all Markov controls.
\end{definition}
\begin{definition}
(\emph{Stationary and Stable Stationary Markov Controls:})
 If the function $v$ in the above definition is independent of $t$, then $U$, or by an abuse of notation $v$ itself, is called a \emph{stationary Markov control}.
 We denote the set of all such controls by \( \Usm \).\\
A stationary Markov control \(v \in \Uadm_{\mathrm{sm}}\) is said to be \emph{stable} if the 
corresponding controlled regime-switching diffusion process \((X_t, S_t)\) is positive recurrent.
The set of all stable stationary Markov controls is denoted by \( \Uadm_{\mathrm{ssm}}\).
\end{definition}

The hypotheses in \hyperlink{A1}{{(A1)}}--\hyperlink{A3}{{(A3)}} also imply the existence of unique strong solutions under Markov controls, which is a strong Feller (therefore strong Markov) process (see \cite{AGM93}*{Theorem 2.1} and \cite{ABG-book}*{Theorem 5.2.9}).
 From \cite{AGM93}*{Section 3}, we have that the set $\Usm$ is metrizable with compact metric with the following topology: a sequence $v_n\to v$ in $\Usm$ if and only if
\begin{equation*}
\lim_{n\to\infty}\int_{\Rd}f(x,i)\int_{\Act}g(x,i,\cdot)v_{n}(x,i)(\D \zeta)\D x = \int_{\Rd}f(x,i)\int_{\Act}g(x,i,\cdot)v(x,i)(\D \zeta)\D x
\end{equation*}
for all $f\in \Lp^1(\Rd\times \mathbb{S})\cap \Lp^2(\Rd\times \mathbb{S})$, $g\in \cC_b(\Rd\times \mathbb{S} \times  \Act)$ and $i\in \mathbb{S}$ (for more details, see \cite{AGM93}*{Lemma~3.2})\,.
Similarly, in view of \cite{Bor89}, from \cite{SPSY}*{Definition 2.2} we say a sequence $v_n \to v$ in $\mathcal{\Um}$ if and only if
\begin{align*}
\lim_{n \to \infty} 
&\int_0^\infty \int_{\Rd} f(t,x,i) 
\Bigg( \int_{\Act} g(t,x,i,\zeta)\, v_n(t,x,i)(d\zeta) \Bigg) dx\,dt\nonumber\\
&=
\int_0^\infty \int_{\Rd} f(t,x,i) 
\Bigg( \int_{\Act} g(t,x,i,\zeta)\, v(t,x,i)(d\zeta) \Bigg) dx\,dt,
\end{align*}
for all 
$i\in\mathbb{S,}\,\, f \in L^1([0,\infty)\times\Rd \times\mathbb{S}) \cap L^2([0,\infty)\times\Rd \times\mathbb{S} ),\,\, g \in \cC_b([0,\infty)\times\Rd\times\mathbb{S} \times \Act).$


We define a family of operators $\sL_{\zeta}$ mapping
$\cC^2(\Rd\times \mathbb{S})$ to $\cC(\Rd\times \mathbb{S})$ by
\begin{equation*}\label{E-cI}
\sL_{\zeta} f(x,i) \,\df\, \trace\bigl(a(x,i)\grad^2 f(x,i)\bigr) + \,b(x,i,\zeta)\cdot \grad f(x,i)\,+\, \sum_{j \in \mathbb{S}} m_{ij}(x,\zeta) f(x,j)\,, 
\end{equation*}
for $\zeta\in\Act$, $f\in \cC^2(\Rd\times \mathbb{S})$\,.
For $\pv \in\pV$ we extend $\sL_{\zeta}$ as follows:
\begin{equation*}\label{EExI}
\sL_\pv f(x,i) \,\df\, \int_{\Act} \sL_{\zeta} f(x,i)\pv(\D \zeta)\,.
\end{equation*} For $v \in\Usm$, we define
\begin{equation*}\label{Efixstra}
\sL_{v} f(x,i) \,\df\, \trace(a(x,i)\grad^2 f(x,i)) + b(x,i,v(x,i))\cdot\grad f(x,i)\,+\, \sum_{j \in \mathbb{S}} m_{ij}(x,v(x,i)) f(x,j)\, .
\end{equation*}
\subsection{Approximating Control Regime-Switching Diffusion Process:}
Let, $\upsigma_n\,=\,\bigl[\upsigma_n^{ij}\bigr]\colon\Rd\times \mathbb{S} \to\RR^{d\times d}$, $b_n\colon\Rd\times \mathbb{S} \times \Act\to\Rd$, $c_n\colon\Rd\times \mathbb{S} \times \Act\to\Rd$ and $m_{jk}^n\,\colon\Rd \times \Act\to\Rd$ be a sequence of functions satisfying the following assumptions 
\begin{itemize}
\item[\hypertarget{A6}{{(A6)}}]
\begin{itemize}
\item[(i)] as $n\to\infty$ 
\begin{equation*}\label{AproxiE1}
\upsigma_n(x,i)\to \upsigma(x,i)\quad \text{a.e.}\,\, (x,i)\in\Rd\times \mathbb{S},
\end{equation*}
\item[(ii)]\emph{Continuous convergence in controls}: for any sequence $\zeta_n\to \zeta$
\begin{align*}\label{AproxiE2}
&c_n(x,i,\zeta_n)\to c(x,i,\zeta),\quad b_n(x,i,\zeta_n)\to b(x,i,\zeta)\quad \text{a.e.}\,\, (x,i)\in\Rd\times \mathbb{S}\,\nonumber \\ \quad\text{and} &\quad m_{jk}^n(x,\zeta_n)\to m_{jk}(x,\zeta).\quad \text{a.e.}\,\, x\in\Rd\,\text{and}\,\,j,k\in \mathbb{S}. 
\end{align*}
\item[(iii)] for each $n\in\NN$,\, $b_n$, $m_{jk}^n$ and $\upsigma_n$ satisfy Assumptions \hyperlink{A1}{{(A1)}} - \hyperlink{A4}{{(A4)}} and $c_n$, $m_{ij}^n$'s are uniformly bounded ( in particular, $\norm{m_{ij}^n}_{\infty} ,\norm{c_n}_{\infty} \leq M$ where $M>0$ is a constant as in \hyperlink{A5}{{(A5)}}), jointly continuous in $x$ and $\zeta $, and locally Lipschitz continuous in its first argument uniformly with respect to $\zeta\in\Act$. 
\end{itemize}
\end{itemize}
Let for each $n\in\NN$, $(X_t^n,S_t^n)$ be the solution of the following SDE
\begin{equation}\label{ASE1.1}
\begin{aligned}
dX_t^n &= b_n(X_t^n,S_t^n, U_t)\,dt + \upsigma_n(X_t^n,S_t^n)\,dW_t, \\
dS_t^n &= \int_{\RR} h_n(X_t^n, S_{t-}^n, U_t, z)\, \mathcal{P}(dt, dz) 
\end{aligned}
\end{equation}
for t $\geq$ 0, with $(X_0^n,S_0^n)=(x,i)\in \Rd \times \mathbb{S}$.\\
Define a family of operators $\sL_{\zeta}^n$ mapping
$\cC^2(\Rd\times \mathbb{S})$ to $\cC(\Rd\times \mathbb{S})$ by
\begin{equation*}\label{E-cIn}
\sL_{\zeta}^n f(x,i) \,\df\, \trace\bigl(a_n(x,i)\grad^2 f(x,i)\bigr) + \,b_n(x,i,\zeta)\cdot \grad f(x,i)\,+\, \sum_{j \in \mathbb{S}} m^{n}_{ij}(x,\zeta) V_{\alpha}^{n}(x,j)\,, 
\end{equation*}
for $\zeta\in\Act$, $f\in \cC^2(\Rd\times \mathbb{S})$.

For the approximated model, for each $n\in\NN$ and $U\in\Uadm$ define the associated costs as follows

\begin{enumerate}
    \item The $\alpha$-discounted cost:
\begin{equation*}\label{EApproDiscost}
\cJ_{\alpha, n}^{U}(x,i,c_n) \,\df\, \Exp_{x,i}^{U} \left[\int_0^{\infty} e^{-\alpha t} c_n(X_s^n,S_s^n, U_s) \D s\right],\quad (x,i)\in\Rd\times \mathbb{S}\,.
\end{equation*} and the optimal value is defined as 
\begin{equation*}\label{EApproOptDisc}
V_{\alpha}^n(x,i) \,\df\, \inf_{U\in\Uadm}\cJ_{\alpha, n}^{U}(x,i,c_n)
\end{equation*}
    \item The ergodic cost:   
\begin{equation*}\label{ECostAprox1}
\sE_{x,i}^n(c_n, U) = \limsup_{T\to \infty}\frac{1}{T}\Exp_{x,i}^{U}\left[\int_0^{T} c_n(X_s^n,S_s^n, U_s) \D{s}\right]\,,
\end{equation*} and the optimal value is defined as
\begin{equation*}\label{ECost1OptAprox}
\sE^{n*}(c_n) \,\df\, \inf_{(x,i)\in\Rd\times \mathbb{S}}\inf_{U\in \Uadm}\sE_{x,i}^n(c_n, U)\,,
\end{equation*}

    \item The finite horizon cost:

\begin{equation*}\label{APFiniteCost1}
\cJ_{T,n}^U(x,i,c_n) = \Exp_{x,i}^{U}\left[\int_0^{T} c_n(X_s^n,S_s^n, U_s) \D{s} + c_{_{T}}(X_T^n,S_T^n)\right]\,,
\end{equation*}
and the optimal value is defined as
\begin{equation*}
\cJ_{T,n}^*(x,i,c_n) \,\df\, \inf_{U\in \Uadm}\cJ_{T,n}^U(x,i,c_n)\,.
\end{equation*}
    
    \item Cost up to an exit time. 
 \begin{align*}
    \hat{\cJ}^U_{e,n}(x,i):=\Exp_{x,i}^U\!\bigg[
\int_0^{\uptau(\mathcal{O})}&e^{-\int_0^t \beta(X_s^n,S_s^n, U_s)\,ds} c_n(X_s^n,S_s^n, U_s)\,dt \\&+ e^{-\int_0^{\uptau(\mathcal{O})} \beta(X_s^n,S_s^n, U_s)\,ds} h(X_{\uptau(\mathcal{O})}^n,S_{\uptau(\mathcal{O})}^n)
    \bigg],
   \end{align*}
and the optimal value is defined as
    \begin{equation*}
     \hat{\cJ}^{*}_{e,n}(x,i)=\inf_{U\in \Uadm}\hat{\cJ}^{U}_{e,n}(x,i)
    \end{equation*}
\end{enumerate} 

where state process $(X_t^n,S_t^n)$ is given by the solution of the SDE \cref{ASE1.1}\,.
\subsection{Continuity and Robustness problems}\label{CRPB}
In this work, our primary objective is to address the following two fundamental questions:
\begin{itemize}
    \item \textbf{Continuity.}  
    Suppose the approximate model \cref{ASE1.1} converges to the true model \cref{E1.1}.  
    Does this imply convergence of the corresponding value functions, namely:
    \begin{itemize}
        \item[•] \emph{Discounted cost:} $V_{\alpha}^n(x,i) \;\to\; V_{\alpha}(x,i)$ ? 
        \item[•] \emph{Ergodic cost:} $\mathcal{\sE}^{n*}(c_n) \;\to\; \mathcal{\sE}^*(c)$ ? 
        \item[•] \emph{Finite-horizon cost:} $\mathcal{J}_{T,n}^*(x,i,c_n) \;\to\; \mathcal{J}_{T}^*(x,i,c)$ ? 
        \item[•] \emph{Exit-time cost:} $\hat{\mathcal{J}}_{e,n}^*(x,i) \;\to\; \hat{\mathcal{J}}_{e}^*(x,i)$ ?
    \end{itemize}

    \item \textbf{Robustness.}  
    Suppose $v_{n}^*$ is an optimal policy designed for the approximate model \cref{ASE1.1} in the discounted/ergodic/finite-horizon/exit-time setting.  
    Does this imply asymptotic optimality for the true model \cref{E1.1}, in the sense that:
    \begin{itemize}
        \item[•] \emph{Discounted cost:} $\mathcal{J}_{\alpha}^{v_n^*}(x,i,c) \;\to\; V_{\alpha}(x,i)$ ? 
        \item[•] \emph{Ergodic cost:} $\mathcal{\sE}_{x,i}(c, v_n^*) \;\to\; \mathcal{\sE}^*(c)$ ? 
        \item[•] \emph{Finite-horizon cost:} $\mathcal{J}_{T}^{v_n^*}(x,i,c) \;\to\; \mathcal{J}_{T}^*(x,i,c)$ ? 
        \item[•] \emph{Exit-time cost:} $\hat{\mathcal{J}}_{e}^{v_n^*}(x,i) \;\to\; \hat{\mathcal{J}}_{e}^*(x,i)$ ?
    \end{itemize}
\end{itemize}
Now we introduce the notations that will be used throughout the rest of the article.
\subsection*{Notation:}
\begin{itemize}
\item For any set $A\subset\Rd$, by $\uptau(A)$ we denote \emph{first exit time} of the process $(X_{t},S_t)$ from the set $A\subset\Rd$, defined by
\begin{equation*}
\uptau(A) \,\df\, \inf\,\{t>0\,\colon (X_{t},S_t)\not\in A\times \mathbb{S}\}\,.
\end{equation*}
\item $\sB_{r}$ denotes the open ball of radius $r$ in $\Rd$, centered at the origin, and $\sB_{r}^c$ denotes the complement of $\sB_{r}$ in $\Rd$\,.
\item $\uptau_{r}$, $\uuptau_{r}$ denote the first exist time from $\sB_{r}$, $\sB_{r}^c$ respectively, i.e., $\uptau_{r}\df \uptau(\sB_{r})$, and $\uuptau_{r}\df \uptau(\sB^{c}_{r})$.
\item By $\trace A$ we denote the trace of a square matrix $A$.
\item For any domain $\cD\subset\Rd$, the space $\cC^{k}(\cD)$ ($\cC^{\infty}(\cD)$), $k\ge 0$, denotes the class of all real-valued functions on $\cD$ whose partial derivatives up to and including order $k$ (of any order) exist and are continuous.
\item $\cC_{\mathrm{c}}^k(\cD)$ denotes the subset of $\cC^{k}(\cD)$, $0\le k\le \infty$, consisting of functions that have compact support. This denotes the space of test functions.
\item $\cC_{b}(\Rd)$ denotes the class of bounded continuous functions on $\Rd$\,.
\item $\cC^{k}_{0}(\cD)$ denotes the subspace of $\cC^{k}(\cD)$, $0\le k < \infty$, consisting of functions that vanish in $\cD^c$.
\item $\cC^{k,r}(\cD)$ denotes the class of functions whose partial derivatives up to order $k$ are H\"older continuous of order $r$.
\item $\Lp^{p}(\cD)$, $p\in[1,\infty)$ denotes the Banach space
of (equivalence classes of) measurable functions $f$ satisfying
$\int_{\cD} \abs{f(x)}^{p}\,\D{x}<\infty$.
\item $\Sob^{k,p}(\cD)$, $k\ge0$, $p\ge1$ denotes the standard Sobolev space of functions on $\cD$ whose weak derivatives up to order $k$ are in $\Lp^{p}(\cD)$, equipped with its natural norm (see, \cite{Adams})\,.
\item  If $\mathcal{X}(Q)$ is a space of real-valued functions on $Q$, $\mathcal{X}_{\mathrm{loc}}(Q)$ consists of all functions $f$ such that $f\varphi\in\mathcal{X}(Q)$ for every $\varphi\in\cC_{\mathrm{c}}^{\infty}(Q)$. In a similar fashion, we define $\Sobl^{k, p}(\cD)$. 

\item Define $\mathfrak{C}$ as the collection of functions $f:\Rd\times \mathbb{S}\times \Act \to \RR$ such that, for each $ \zeta\in \Act$ and $i\in \mathbb{S}$, the mapping 
$f(\cdot,i,\zeta)$ is continuous and locally Lipschitz. 

\item We also adopt the notation \( \mathcal{X}(Q \times\mathbb{S}) \) to indicate the product space \( (\mathcal{X}(Q))^N \), where \( N \) is the cardinality of \( \mathbb{S} \). The corresponding norm on \( \mathcal{X}(Q \times\mathbb{S}) \) is defined by
\[
\|f\|_{\mathcal{X}(Q \times\mathbb{S})} := \sum_{k \in\mathbb{S}} \|f_k\|_{\mathcal{X}(Q)}
\]

Let \( f \in \cC(\Rd \times\mathbb{S}) \), then, by \( f \gg 0 \), we mean that \( f_k > 0 \) for all \( k \in\mathbb{S} \).

\item If $h \in \cC(\Rd \times \mathbb{S} \times  \Act)$ with $h>0$, $\sorder(h)$ denotes the set of functions $f \in \cC(\Rd \times \mathbb{S} \times  \Act)$ having the property 
\begin{equation*}
\limsup_{|x| \to \infty} \sup_{\zeta\in  \Act} \sup_{i \in \mathbb{S}} \frac{|f(x, i, \zeta)|}{h(x, i, \zeta)} = 0
\end{equation*}
\item $\eta_v$ denotes the invariant probability measure of controlled  regime-switching diffusion process $(X,S)$ under $v \in \Uadm_{\mathsf{ssm}}$ (see, \cite{ABG-book}*{p. 206}).
\item $\pi_v$ denotes the ergodic occupation measure for $v\in \Uadm_{\mathsf{ssm}}$ (see, \cite{ABG-book}*{Chapter 3},\cite{AGM93}*{p. 10}).

\end{itemize}
\section{Analysis of Discounted Cost Criterion}
In this section, we establish the continuity properties of the value functions and robustness of optimal controls for the discounted cost criterion introduced in \eqref{EDiscost}--\eqref{OPDcost}. Our objective is to show that, as the 
approximating sequence of regime-switching diffusion models converges to that of 
the true model, the corresponding optimal values and controls also converge to the optimal value and
control of the true model.

From \cite{AGM93}*{Theorems~6.1 and 6.2, together with Corollary~6.1}, we obtain the following characterization of the optimal $\alpha$-discounted cost function $V_{\alpha}$.

\begin{theorem}\label{TH2.1}
Suppose Assumptions \hyperlink{A1}{{(A1)}}–\hyperlink{A5}{{(A5)}} hold. 
Then the optimal discounted cost function $V_{\alpha}$, defined in \cref{OPDcost}, 
is the unique solution in 
$\mathcal{C}^2(\Rd \times \mathbb{S}) \cap \mathcal{C}_b(\Rd \times \mathbb{S})$ 
that satisfies the Hamilton–Jacobi–Bellman (HJB) equation
\begin{equation}\label{2.1}
\min_{\zeta \in \Act}
\left[
\sL_{\zeta} V_{\alpha}(x,i) + c(x,i,\zeta)
\right]
= \alpha V_{\alpha}(x,i),
\quad (x,i) \in \Rd \times \mathbb{S}.
\end{equation}
Moreover, $v^*\in \Usm$ is $\alpha$-discounted optimal control if and only if it is a measurable minimizing selector of \cref{2.1}; that is,
\begin{align*}\label{OtpHJBSelc}
&b(x,i,v^*(x,i))\cdot \grad V_{\alpha}(x,i) + c(x,i,v^*(x,i))\,+\, \sum_{j \in \mathbb{S}} m_{ij}(x,v^*(x,i)) V_{\alpha}(x,j)\nonumber\\ &= \min_{\zeta\in \Act}\Big[ b(x,i, \zeta)\cdot \grad V_{\alpha}(x,i) + c(x,i,\zeta)\,+\, \sum_{j \in \mathbb{S}} m_{ij}(x,\zeta) V_{\alpha}(x,j)\Big]\quad \text{a.e.}\,\,\, (x,i)\in\Rd\times \mathbb{S}\,.
\end{align*}
\end{theorem}

As in \cref{TH2.1}, for each approximating model, we have the following complete characterization of the optimal value and optimal policies in $\Usm$.
\begin{theorem}\label{TD1.2}
Suppose \hyperlink{A6}{{(A6)}}(iii) holds. Then for each $n\in \NN$, there exists a unique solution $V_{\alpha}^n\in\cC^2(\Rd\times \mathbb{S})\cap\cC_b(\Rd\times \mathbb{S})$ of
\begin{equation}\label{APOptDHJB1}
\min_{\zeta \in\Act}\left[\sL_{\zeta}^nV_{\alpha}^n(x,i) + c_n(x,i,\zeta)\right] = \alpha V_{\alpha}^n(x,i) \,,\quad \text{for all\ }\,\, (x,i)\in\Rd\times \mathbb{S}\,.
\end{equation}
Moreover, we have the following:
\begin{itemize}
\item[(i)] $V_{\alpha}^n$ is the optimal discounted cost, i.e.,
\begin{equation*}
V_{\alpha}^n(x,i) = \inf_{U\in \Uadm}\Exp_{x,i}^{U} \left[\int_0^{\infty} e^{-\alpha s} c_n(X_s^n,S_s^n, U_s) \D s\right]\quad (x,i)\in\Rd\times \mathbb{S},
\end{equation*}
\item[(ii)]$v_n^*\in \Usm$ is $\alpha$-discounted optimal control if and only if it is a measurable minimizing selector of \cref{APOptDHJB1}, i.e.,
\begin{align}\label{OtpHJBSelc1}
&b_n(x,i,v_n^*(x,i))\cdot \grad V_{\alpha}^n(x,i) + c_n(x,i,v_n^*(x,i))\,+\, \sum_{j \in \mathbb{S}} m^{n}_{ij}(x,v_n^*(x,i)) V_{\alpha}^{n}(x,j)\nonumber\\
&= \min_{\zeta\in \Act}\left[ b_n(x,i,\zeta)\cdot \grad V_{\alpha}^n(x,i) + c_n(x,i,\zeta)\,+\, \sum_{j \in \mathbb{S}} m^{n}_{ij}(x,\zeta) V_{\alpha}^{n}(x,j)\right]\quad \text{a.e.}\,\,\, (x,i)\in\Rd\times \mathbb{S}\,
\end{align}
\end{itemize}
\end{theorem}
In the next theorem, we establish the continuity result that $V_{\alpha}^n(x,i) \to V_{\alpha}(x,i)$ as $n \to \infty$ for all $(x,i) \in \Rd \times \mathbb{S}$. This convergence plays a key role in demonstrating the robustness of discounted optimal controls.
\begin{theorem}\label{TH2.3}
Suppose Assumptions \hyperlink{A1}{{(A1)}}-\hyperlink{A6}{{(A6)}} hold. Then 
\begin{equation*}\label{EC1.1}
\lim_{n\to\infty} V_{\alpha}^n(x,i) = V_{\alpha}(x,i) \quad\forall\,\, (x,i)\in\Rd\times \mathbb{S}\,.
\end{equation*}
\end{theorem}

\begin{proof}
Let $v_n^*\in \Usm$ be a minimizing selector of \cref{APOptDHJB1}. From \cref{APOptDHJB1} and \cref{OtpHJBSelc1} it follows that
\begin{align*}
&\trace\bigl(a_n(x,i)\grad^2 V_{\alpha}^n(x,i)\bigr) + b_n(x,i,v_n^*(x,i))\cdot \grad V_{\alpha}^n(x,i) + c_n(x,i,v_n^*(x,i))\\&\qquad+ \sum_{j \in \mathbb{S}} m^{n}_{ij}(x,v^*_n(x,i)) V_{\alpha}^{n}(x,j)
 = \alpha V_{\alpha}^n(x,i)\,.
\end{align*}

Rewriting the above equation, we have 
\begin{align*}
&\trace\bigl(a_n(x,i)\grad^2 V_{\alpha}^n(x,i)\bigr) + b_n(x,i,v_n^*(x,i))\cdot \grad V_{\alpha}^n(x,i)+ (m^{n}_{ii}(x,v^*_n(x,i))-\alpha) V_{\alpha}^n(x,i)\\& =  f(x,i)\,,\quad \text{a.e.}\,\, (x,i)\in\Rd\times \mathbb{S}\,,
\end{align*} where 
\begin{equation*}
f(x,i) = -[ c_n(x,i,v^*_n(x,i)) + \sum_{j \neq i} m^{n}_{ij}(x,v^*_n(x,i)) V_{\alpha}^{n}(x,j)  ]\,.
\end{equation*}
Then using the standard elliptic PDE estimate as in \cite{GilTru}*{Theorem~9.11},  for any $p\geq d+1$ and $R >0$, we deduce that
\begin{equation}\label{ETC1.3A}
\norm{V_{\alpha}^n(x,i)}_{\Sob^{2,p}(\sB_R)}
\,\le\, \kappa_1\bigl(\norm{V_{\alpha}^n(x,i)}_{\Lp^p(\sB_{2R})} + \norm{f(x,i)}_{L^p(\sB_{2R})}\bigr)\,,
\end{equation}
where $\kappa_1$ is a positive constant which is independent of $n$\,.\\
Since 
\begin{align*}
&\norm{c_n(x,i,\zeta)}_{\infty} \,\df\, \sup_{(x,i, \zeta)\in\Rd\times \mathbb{S}\times\Act} c_n(x,i, \zeta) \leq M, \quad  V_{\alpha}^n(x,i) \leq \frac{\norm{c_n(x,i, \zeta)}_{\infty}}{\alpha}\,,\\
&\qquad\qquad\text{and}\quad \norm{m_{ij}^n(x,i, \zeta)}_{\infty} \,\df\, \sup_{(x,\zeta)\in\Rd\times\Act} m_{ij}^n(x,\zeta) \leq M,
\end{align*}
we have,
\begin{align}\label{fb}
\norm{f(x,i)}_{L^p(\sB_{2R})}
&\leq \norm{c_n(x,i, \zeta)}_{L^p(\sB_{2R})} +\norm{\sum_{j\neq i}m^{n}_{ij}(x,v^*_n(x,i))V_{\alpha}^n(x,j)}_{L^p(\sB_{2R})}\nonumber\\
&\leq \norm{c_n(x,i, \zeta)}_{\infty}|\sB_{2R}|^{\frac{1}{p}}\,+\,\norm{\sum_{j\neq i}m^{n}_{ij}(x,v^*_n(x,i))}_{\infty}\norm{V_{\alpha}^n(x,j)}_{\infty}|\sB_{2R}|^{\frac{1}{p}}\nonumber\\
&\leq M|\sB_{2R}|^{\frac{1}{p}}\,+\,|\mathbb{S}|M.\frac{M}{\alpha}|\sB_{2R}|^{\frac{1}{p}}\nonumber\\
&\leq  M|\sB_{2R}|^{\frac{1}{p}}\bigl( 1+ \frac{|\mathbb{S}|M}{\alpha}\bigr)
\end{align}
from \cref{ETC1.3A} and \cref{fb}, we get
\begin{equation}\label{ETC1.3B}
\norm{V_{\alpha}^n(x,i)}_{\Sob^{2,p}(\sB_R\times \mathbb{S})}=\sum_{i \in \mathbb{S}}\norm{V_{\alpha}^n(x,i)}_{\Sob^{2,p}(\sB_R)}
\,\le\,\kappa_1 |\mathbb{S}| M |\sB_{2R}|^{\frac{1}{p}}\bigl(1+ \frac{M|\mathbb{S}|+1}{\alpha}\bigr)\,.
\end{equation}

We know that for $1< p < \infty$, the space $\Sob^{2,p}(\sB_R)$ is reflexive and separable, and reflexivity and separability are preserved under finite Cartesian product, it follows that \( \Sob^{2,p}(\cB_R \times \mathbb{S}) \) is also reflexive and separable for \( 1 < p < \infty \); hence, as a corollary of the Banach Alaoglu theorem, we have that every bounded sequence in $\Sob^{2,p}(\sB_R\times \mathbb{S})$ has a weakly convergent subsequence (see, \cite{HB-book}*{Theorem~3.18}). Also, we know that for $p\geq d+1$ the space $\Sob^{2,p}(\sB_R)$ is compactly embedded in $\cC^{1, \beta}(\bar{\sB}_R)$\,, where $\beta < 1 - \frac{d}{p}$ (see \cite{ABG-book}*{Theorem~A.2.15 (2b)}). Since $\mathbb{S}$ is finite, the space $\Sob^{2,p}(\cB_R \times \mathbb{S})$ can be thought of as a finite product of such Sobolev spaces, one for each $i \in \mathbb{S}$. Therefore, the embedding $
\Sob^{2,p}(\cB_R \times \mathbb{S}) \hookrightarrow \cC^{1,\beta}(\cB_R \times \mathbb{S})$
is compact, which implies that every weakly convergent sequence in $\Sob^{2,p}(\sB_R\times \mathbb{S})$ will converge strongly in $\cC^{1, \beta}(\bar{\sB}_R\times \mathbb{S})$\,. Thus, in view of estimate \cref{ETC1.3B}, by standard diagonalization argument and Banach Alaoglu theorem, we can extract a subsequence $\{V_{\alpha}^{n_k}\}$ such that for some $V_{\alpha}^*\in \Sobl^{2,p}(\Rd\times \mathbb{S})$
\begin{equation}\label{ETC1.3BC}
\begin{cases}
V_{\alpha}^{n_k}\to & V_{\alpha}^*\quad \text{in}\quad \Sobl^{2,p}(\Rd\times \mathbb{S})\quad\text{(weakly)}\\
V_{\alpha}^{n_k}\to & V_{\alpha}^*\quad \text{in}\quad \cC^{1, \beta}_{loc}(\Rd\times \mathbb{S}) \quad\text{(strongly)}\,.
\end{cases}       
\end{equation} 
Next, we will show that $V^*_{\alpha} = V_{\alpha}$. Now, for any compact set $K\subset \Rd$, it is easy to see that

\begin{align*}
&\max_{x\in K} \Big| \min_{\zeta\in \Act} \big\{ b_{n_k}(x,i,\zeta)\cdot \nabla V_{\alpha}^{n_k}(x,i) + c_{n_k}(x,i,\zeta) + \sum_{j \in \mathbb{S}} m^{n_{k}}_{ij}(x,\zeta) V_{\alpha}^{n_k}(x,j) \big\}\nonumber\\
 &\quad \quad- \min_{\zeta\in \Act} \big\{ b(x,i,\zeta)\cdot \nabla V_{\alpha}^*(x,i) + c(x,i,\zeta) 
+ \sum_{j \in \mathbb{S}} m_{ij}(x,\zeta)V^{*}_{\alpha}(x,j) \big\} \Big|\nonumber \\
&\leq \max_{x\in K} \max_{\zeta\in \Act} | 
b_{n_k}(x,i,\zeta)\cdot \nabla V_{\alpha}^{n_k}(x,i) - b(x,i,\zeta)\cdot \nabla V_{\alpha}^*(x,i) |\nonumber\\
&+ \max_{x\in K} \max_{\zeta\in \Act} | c_{n_k}(x,i,\zeta) - c(x,i,\zeta) |+ \max_{x\in K} \max_{\zeta\in \Act} \Big| 
\sum_{j \in \mathbb{S}} m^{n_{k}}_{ij}(x,\zeta) V_{\alpha}^{n_k}(x,j)
- \sum_{j \in \mathbb{S}} m_{ij} (x,\zeta)V^{*}_{\alpha}(x,j) \Big|
\end{align*}

Since $c_{n}(x,i, \cdot)\to c(x,i, \cdot)$, $b_{n}(x,i, \cdot)\to b(x,i, \cdot)$, $m^{n}_{ij}(x,\cdot)\to m_{ij}(x,\cdot)$ continuously on compact set $\Act$ and $V_{\alpha}^{n_k}\to V_{\alpha}^*$ in $\cC^{1, \beta}_{loc}(\Rd\times \mathbb{S})\,,$ for any compact set $K\subset \Rd$, as $k\to \infty$ we deduce that 
\begin{align}\label{2.11}
&\max_{x\in K} \Big| \min_{\zeta\in \Act} \big\{ b_{n_k}(x,i,\zeta)\cdot \nabla V_{\alpha}^{n_k}(x,i) + c_{n_k}(x,i,\zeta) 
+ \sum_{j \in \mathbb{S}} m^{n_{k}}_{ij}(x,\zeta) V_{\alpha}^{n_k}(x,j) \big\} \nonumber \\
&\quad - \min_{\zeta\in \Act} \big\{ b(x,i,\zeta)\cdot \nabla V_{\alpha}^*(x,i) + c(x,i,\zeta) 
+ \sum_{j \in \mathbb{S}} m_{ij}(x,\zeta) V^{*}_{\alpha}(x,i) \big\} \Big| \to 0 
\end{align}
Thus, multiplying by a test function $\phi\in \cC_{c}^{\infty}(\Rd\times \mathbb{S})$, from \cref{APOptDHJB1}, we obtain
\begin{align*}
&\int_{\Rd} 
\trace\bigl(a_{n_k}(x,i)\nabla^2 V_{\alpha}^{n_k}(x,i)\bigr)\, \phi(x,i)\,\mathrm{d}x \\& + \int_{\Rd} \min_{\zeta\in \Act} 
\Big\{ b_{n_k}(x,i,\zeta)\cdot \nabla V_{\alpha}^{n_k}(x,i) 
+ c_{n_k}(x,i,\zeta) 
+ \sum_{j \in \mathbb{S}} m^{n_{k}}_{ij}(x,\zeta) V_{\alpha}^{n_k}(x,j) \Big\} 
\phi(x,i)\,\mathrm{d}x \\
&= \alpha \int_{\Rd} V_{\alpha}^{n_k}(x,i)\, \phi(x,i)\,\mathrm{d}x\,.
\end{align*}
In view of \cref{ETC1.3BC} and \cref{2.11}, letting $k\to\infty$ it follows that
\begin{align}\label{2.12}
&\int_{\Rd} 
\trace\bigl(a(x,i)\nabla^2 V_{\alpha}^*(x,i)\bigr)\, \phi(x,i)\, \mathrm{d}x \nonumber\\ &\qquad+ \int_{\Rd} 
\min_{\zeta\in \Act} \Big\{ 
b(x,i,\zeta)\cdot \nabla V_{\alpha}^*(x,i) 
+ c(x,i,\zeta) 
+ \sum_{j \in \mathbb{S}} m_{ij}(x,\zeta)\,V_{\alpha}^*(x,j) \Big\} 
\phi(x,i)\, \mathrm{d}x \nonumber\\
&\quad \quad= \alpha \int_{\Rd} V_{\alpha}^*(x,i)\, \phi(x,i)\, \mathrm{d}x\,.
\end{align} 
Since $\phi\in \cC_{c}^{\infty}(\Rd\times \mathbb{S})$ is arbitrary and $V_{\alpha}^*\in \Sobl^{2,p}(\Rd\times \mathbb{S})$ from \cref{{2.12}} we deduce that
\begin{equation}\label{ETC1.3E}
\min_{\zeta \in\Act}\left[\sL_{\zeta}V_{\alpha}^*(x,i) + c(x,i,\zeta)\right] = \alpha V_{\alpha}^*(x,i) \,,\quad \text{a.e.\ }\,\, (x,i)\in\Rd\times \mathbb{S}\,.
\end{equation}
Consider a minimizing selector $\tilde{v}^* \in \Usm$ of \cref{ETC1.3E}, and let $(\tilde{X}, \tilde{S})$ be the solution of the SDE \cref{E1.1} corresponding to $\tilde{v}^*$. Then, by It$\hat{\rm o}$–Krylov formula (\cite{ABG-book}*{Lemma~5.1.4}), we deduce the following.
\begin{align*}
&\Exp_{x,i}^{\tilde{v}^*}\left[ e^{-\alpha T} V_{\alpha}^{*}(\tilde{X}_{T},\tilde{S}_{T})\right] - V_{\alpha}^{*}(x,i)\\
&\,=\,\Exp_{x,i}^{\tilde{v}^*}\Bigg[\int_0^{T} e^{-\alpha s}\{\trace\bigl(a(\tilde{X}_s,\tilde{S}_s)\grad^2 V_{\alpha}^{*}(\tilde{X}_s,\tilde{S}_s)\bigr) + b(\tilde{X}_s,\tilde{S}_s, \tilde{v}^*(\tilde{X}_s,\tilde{S}_s))\cdot \grad V_{\alpha}^{*}(\tilde{X}_s,\tilde{S}_s)\\
& \quad+ \sum_{j \in \mathbb{S}} m_{\tilde{S}_sj}(x,\tilde{v}^*(\tilde{X}_s,\tilde{S}_s))\,V_{\alpha}^*(\tilde{X}_s,j) - \alpha V_{\alpha}^{*}(\tilde{X}_s,\tilde{S}_s)\} \D{s}\Bigg]
\end{align*}

Hence, using \cref{ETC1.3E} and rewriting the above equation, we obtain,
\begin{align}\label{ETC1.3FC}
e^{-\alpha T}\Exp_{x,i}^{\tilde{v}^*}\left[  V_{\alpha}^{*}(\tilde{X}_T,\tilde{S}_{T})\right] - V_{\alpha}^{*}(x,i) \,=\,- \Exp_{x,i}^{\tilde{v}^*}\left[\int_0^{T} e^{-\alpha s}c(\tilde{X}_s,\tilde{S}_s, \tilde{v}^*(\tilde{X}_s,\tilde{S}_s))\D{s}\right] \,.
\end{align}
Since $V_{\alpha}^{*}$ is bounded, 
it follows that 
\(
e^{-\alpha T}\Exp_{x,i}^{\tilde{v}^*}\left[  V_{\alpha}^{*}(\tilde{X}_T,\tilde{S}_{T})\right] \to 0
\)  as $T\to\infty$.
Now, by monotone convergence theorem and letting $T \to \infty$ in \cref{ETC1.3FC} we obtain,
\begin{align}\label{ETC1.3FD}
 V_{\alpha}^{*}(x,i) \,=\, \Exp_{x,i}^{\tilde{v}^*}\left[\int_0^{\infty} e^{-\alpha s}c(\tilde{X}_s,\tilde{S}_s, \tilde{v}^*(\tilde{X}_s,\tilde{S}_s)) \D{s}\right] \,
\end{align}
By a similar argument, for any $U \in \Uadm$, 
using \cref{ETC1.3E} and applying the It\^o--Krylov formula 
(\cite{ABG-book}*{Lemma~5.1.4}), we obtain

\begin{align*}
 V_{\alpha}^{*}(x,i) \,
 \leq\, \Exp_{x,i}^{U}\left[\int_0^{\infty} e^{-\alpha s}c(\tilde{X}_s,\tilde{S}_s, U_s) \D{s}\right] \,.
\end{align*}
By taking the infimum over all $U\in \Uadm$, we then have
\begin{align}\label{ETC1.3FE}
 V_{\alpha}^{*}(x,i) \,
 \leq\,\inf_{U\in\Uadm} \Exp_{x,i}^{U}\left[\int_0^{\infty} e^{-\alpha s}c(\tilde{X}_s,\tilde{S}_s, U_s) \D{s}\right] \,.
\end{align}
 Thus, from \cref{ETC1.3FD} and \cref{ETC1.3FE}, we deduce that
\begin{align}\label{ETC1.3FF}
 V_{\alpha}^{*}(x,i) \,=
 \, \inf_{U\in\Uadm}\Exp_{x,i}^{U}\left[\int_0^{\infty} e^{-\alpha s}c(\tilde{X}_s,\tilde{S}_s, U_s) \D{s}\right] \,.
\end{align}
Since both $V_{\alpha}$ and $V_{\alpha}^*$ are continuous functions on $\Rd$, it follows  from \cref{OPDcost} and \cref{ETC1.3FF} that $V_{\alpha}(x,i) = V_{\alpha}^*(x,i)$ for all $(x,i)\in\Rd\times \mathbb{S}$. This completes the proof.
\end{proof}
Let $(\hat{X}^{n}, \hat{S}^{n})$ denote the solution of the SDE \cref{E1.1} corresponding 
to the optimal control $v_n^*$. Then we have
\begin{equation*}\label{EDiscostRobust}
\cJ_{\alpha}^{v_n^*}(x,i,c) \,=\, \Exp_{x,i}^{v_n^*} \left[\int_0^{\infty} e^{-\alpha t} c(\hat{X}^{n}_s,\hat{S}^n_s, v_n^*(\hat{X}^{n}_s,\hat{S}^n_s)) \D s\right],\quad (x,i)\in\Rd\times \mathbb{S}\,,
\end{equation*}
Next, using the above convergence result, we show that the optimal controls designed for approximate models yield asymptotically optimal performance for the true model. Formally, if $v_n^*\in\Usm$ denotes an optimal stationary Markov control corresponding to $V_\alpha^n$, then
$J_{\alpha}^{v_n^*}(x,i,c) \;\longrightarrow\; V_\alpha(x,i)
\quad \text{as } n\to\infty,$
for every $(x,i)\in\Rd\times\mathbb{S}$.  
Hence, the discounted optimal control problem for controlled  regime-switching diffusions is robust under model approximations satisfying \hyperlink{A6}{{(A6)}}.
Following the approach in \cite{PradhanYuksel2023}, 
we employ the continuity result obtained above as an intermediate step in the proof.
\begin{theorem}\label{TH2.4}
Suppose Assumptions \hyperlink{A1}{{(A1)}}-\hyperlink{A6}{{(A6)}} hold. Then 
\begin{equation*}\label{ETC1.4A}
\lim_{n\to\infty} \cJ_{\alpha}^{v_n^*}(x,i,c) = V_{\alpha}(x,i) \quad\forall\,\, (x,i)\in\Rd\times \mathbb{S},\,.
\end{equation*}
\end{theorem}
\begin{proof}
By the triangle inequality, we have
\[
|\cJ_{\alpha}^{v_{n}^*}(x,i,c) - V_{\alpha}(x,i)| \leq |\cJ_{\alpha}^{v_{n}^*}(x,i,c) - V_{\alpha}^{n}(x,i)| + |V_{\alpha}^{n}(x,i) - V_{\alpha}(x,i)|\,. \] Applying~\cref{TH2.3}, we have $|V_{\alpha}^{n}(x,i) - V_{\alpha}(x,i)|\to 0$ as $n\to \infty$. To complete the proof, it remains to show that $|\cJ_{\alpha}^{v_{n}^*}(x,i,c) - V_{\alpha}^{n}(x,i)|\to 0$ as $n\to \infty$.

For each $n\in \NN$ and $v_n^*\in \Usm$ (an optimal control for the approximating model),
 consider the corresponding Poisson equation, 
\begin{align}\label{ETC1.4B}
&\trace\bigl(a(x,i)\grad^2 \varphi_{n,R}(x,i)\bigr) + b(x,i,v_n^*(x,i))\cdot \grad \varphi_{n,R}(x,i) + c(x,i,v_n^*(x,i))\nonumber\\&\qquad+\sum_{j \in \mathbb{S}} m_{ij}(x,v_{n}^*(x,i))\,\varphi_{n,R}(x,j)= \alpha \varphi_{n,R}(x,i)\,.
\end{align}
By applying \cite{AS}*{Theorem~2.1}, we obtain a unique solution to \cref{ETC1.4B} in every ball \(\cB_R\subset \Rd\).  
Using Sobolev estimates, standard diagonalization and the Banach--Alaoglu theorem, one can then extract a subsequence that converges weakly in \(\Sob^{2,p}\) and strongly in \(\cC^{1,\beta}_{\mathrm{loc}}\), yielding a global solution  $\Psi_{\alpha}^{n}\in \Sobl^{2,p}(\Rd\times \mathbb{S})\cap \cC_{b}(\Rd\times \mathbb{S})$ satisfying \cref{ETC1.4B}.
By applying It$\hat{\rm o}$-Krylov formula (\cite{ABG-book}*{Lemma 5.1.4}), we obtain
\begin{align*} 
&\mathbb{E}_{x,i}^{v_n^*}\left[ e^{-\alpha T} \Psi_{\alpha}^{n}(\hat{X}^{n}_T,\hat{S}^n_T)\right] 
- \Psi_{\alpha}^{n}(x,i)\\
&=\,\mathbb{E}_{x,i}^{v_n^*} \Bigg[ \int_0^{T} e^{-\alpha s} \Big\{
\trace\bigl(a(\hat{X}^{n}_s,\hat{S}^n_s)\nabla^2 \Psi_{\alpha}^{n}(\hat{X}^{n}_s,\hat{S}^n_s)+ \, b\big(\hat{X}^{n}_s,\hat{S}^n_s, v_n^*(\hat{X}^{n}_s,\hat{S}^n_s)\big)\cdot \nabla \Psi_{\alpha}^{n}(\hat{X}^{n}_s,\hat{S}^n_s)\\
&\hspace{10mm}+\, \sum_{j \in \mathbb{S}} m_{\hat{S}_sj}(x,v_{n}^*(\hat{X}^{n}_s,\hat{S}^n_s))\,\Psi_{\alpha}^{n}(\hat{X}^{n}_s,j) - \alpha \Psi_{\alpha}^{n}(\hat{X}^{n}_s,\hat{S}^n_s) \Big\} \, \mathrm{d}s \Bigg]
\end{align*}
Now using \cref{ETC1.4B} with $\Psi_{\alpha}^{n}$ and arguing as in \cref{ETC1.3FC}-\cref{ETC1.3FD}, we obtain
\begin{align*}\label{ETC1.4D}
 \Psi_{\alpha}^{n}(x,i) \,=\, \Exp_{x,i}^{v_n^*}\left[\int_0^{\infty} e^{-\alpha s}c(\hat{X}^{n}_s,\hat{S}^n_s, v_n^*(\hat{X}^{n}_s,\hat{S}^n_s)) \D{s}\right] \,=\, \cJ_{\alpha}^{v_n^*}(x,i,c)\,.
\end{align*}
This implies that $\Psi_{\alpha}^{n} \leq \frac{\norm{c}_{\infty}}{\alpha}$. Hence, as in Theorem \ref{TH2.3} (see, \cref{ETC1.3A}), by standard Sobolev estimate, for any $R>0$ we obtain
$\norm{\Psi_{\alpha}^{n}}_{\Sob^{2,p}(\sB_R\times \mathbb{S})} \leq \kappa_2$\,, for some positive constant $\kappa_2$ independent of $n$. Therefore, by the Banach--Alaoglu theorem and standard diagonalization argument 
(as in \cref{ETC1.3BC}), there exists $\hat{V}_{\alpha}^*\in \Sobl^{2,p}(\Rd\times \mathbb{S})\cap \cC_{b}(\Rd\times \mathbb{S})$ such that, along some subsequence, $\{\Psi_{\alpha}^{n_k}\}$ 
\begin{equation}\label{ETC1.4E}
\begin{cases}
\Psi_{\alpha}^{n_k}\to & \hat{V}_{\alpha}^*\quad \text{in}\quad \Sobl^{2,p}(\Rd\times \mathbb{S})\quad\text{(weakly)}\\
\Psi_{\alpha}^{n_k}\to & \hat{V}_{\alpha}^*\quad \text{in}\quad \cC^{1, \beta}_{loc}(\Rd\times \mathbb{S})\quad\text{(strongly)}\,.
\end{cases}       
\end{equation}
Since the space of stationary Markov strategies $\Usm$ is compact, there exists a subsequence (which we denote by the same index for convenience) such that $v_{n_k}^*\to \hat{v}^*$ in $\Usm$\,. It is straightforward to verify that
\begin{align*}
&b(x,i,v_{n_k}^*(x,i))\cdot \grad \Psi_{\alpha}^{n_k}(x,i) - b(x,i,\hat{v}^*(x,i))\cdot \grad \hat{V}_{\alpha}^*(x,i) \\
&= b(x,i,v_{n_k}^*(x,i))\cdot \grad \left(\Psi_{\alpha}^{n_k} - \hat{V}_{\alpha}^*\right)(x,i) + \left(b(x,i,v_{n_k}^*(x,i)) - b(x,i,\hat{v}^*(x,i))\right)\cdot \grad \hat{V}_{\alpha}^*(x,i)\,.
\end{align*}
Also,
\begin{align*}
    &\sum_{j \in \mathbb{S}} m_{ij}(x,v_{n_{k}}^*(x,i))\,\Psi_{\alpha}^{n_k}(x,j)\,-\, \sum_{j \in \mathbb{S}} m_{ij}(x,\hat{v}^*(x,i))\,\hat{V}^*_{\alpha}(x,j)\\&=\,\sum_{j \in \mathbb{S}} m_{ij}(x,v_{n_{k}}^*(x,i))\,(\Psi_{\alpha}^{n_k}(x,j)\,-\,\hat{V}_{\alpha}^{*}(x,j))+ \sum_{j \in \mathbb{S}} (m_{ij}(x,v_{n_{k}}^*(x,i))-m_{ij}(x,\hat{v}^*(x,i)))\,\hat{V}_{\alpha}^{*}(x,j). 
\end{align*}

Since $\Psi_{\alpha}^{n_k}\to \hat{V}_{\alpha}^*$ in $\cC^{1, \beta}_{loc}(\Rd\times \mathbb{S}),$ it follows that on every compact set $\left(b(x,i,v_{n_k}^*(x,i))\right)\cdot \grad \left(\Psi_{\alpha}^{n_k} - \hat{V}_{\alpha}^*\right)(x,i)\to 0$ and $m_{ij}(x,v_{n_{k}}^*(x,i))\,(\Psi_{\alpha}^{n_k}(x,j)\,-\,\hat{V}_{\alpha}^{*}(x,j))\to 0$\, strongly (since $m_{ij}$'s are bounded). Moreover, by the topology of $\Usm$, we have \begin{align*}
&\left(b(x,i,v_{n_k}^*(x,i)) - b(x,i,\hat{v}^*(x,i))\right)\cdot \grad \hat{V}_{\alpha}^*(x,i)\to 0\,\quad\, \quad\text{weakly}\\
&\sum_{j \in \mathbb{S}} m_{ij}(x,v_{n_{k}}^*(x,i))\,\Psi_{\alpha}^{n_k}(x,j)\,-\, \sum_{j \in \mathbb{S}} m_{ij}(x,\hat{v}^*(x,i))\,V_{\alpha}(x,j)\to 0\, \quad\text{weakly}
\end{align*}
Thus, in view of the topology of $\Usm$, and the convergence $\Psi_{\alpha}^{n_k}\to \hat{V}_{\alpha}^*$ in $\cC^{1, \beta}_{loc}(\Rd\times \mathbb{S})\,,$ as $k\to \infty$ we obtain 

\begin{align}\label{ETC1.4EA}
b(x,i,v_{n_k}^*(x,i))\cdot \grad \Psi_{\alpha}^{n_k}(x,i) + c(x,i,v_{n_k}^*(x,i))\,+\,\sum_{j \in \mathbb{S}} m_{ij}(x,v_{n_{k}}^*(x,i))\,\Psi_{\alpha}^{n_k}(x,j)\,\nonumber\\ \to b(x,i,\hat{v}^*(x,i))\cdot \grad \hat{V}_{\alpha}^*(x,i) + c(x,i,\hat{v}^*(x,i))\,+\,\sum_{j \in \mathbb{S}} m_{ij}(x,\hat{v}^*(x,i))\,\hat{V}_{\alpha}^{*}(x,j)\quad\text{weakly}\,.
\end{align}

Now, multiplying by a test function $\phi\in \cC_{c}^{\infty}(\Rd\times \mathbb{S})$, from \cref{ETC1.4B}, we obtain
\begin{align*}
&\int_{\Rd}\trace\bigl(a(x,i)\grad^2 \Psi_{\alpha}^{n_k}(x,i)\bigr)\phi(x,i)\D x + \int_{\Rd}\{b(x,i,v_{n_k}^*(x,i))\cdot \grad \Psi_{\alpha}^{n_k}(x,i) +  c(x,i,v_{n_k}^*(x,i))\\
&+\,\sum_{j \in \mathbb{S}} m_{ij}(x,v_{n_{k}}^*(x,i))\,\Psi_{\alpha}^{n_k}(x,j)\}\phi(x,i)\D x 
 = \alpha\int_{\Rd} \Psi_{\alpha}^{n_k}(x,i)\phi(x,i)\D x\,.
\end{align*}
Hence, by \cref{ETC1.4E}, \cref{ETC1.4EA}, and letting $k\to\infty$, we obtain
\begin{align}\label{ETC1.4EB}
&\int_{\Rd}\trace\bigl(a(x,i)\grad^2 \hat{V}_{\alpha}^*(x,i)\bigr)\phi(x,i)\D x + \int_{\Rd} \{b(x,i,\hat{v}^*(x,i))\cdot \grad \hat{V}_{\alpha}^*(x,i) + c(x,i,\hat{v}^*(x,i))\nonumber\\
&+\,\sum_{j \in \mathbb{S}} m_{ij}(x,\hat{v}^*(x,i))\,\hat{V}_{\alpha}^{*}(x,j)\}\phi(x,i)\D x\,=\, \alpha\int_{\Rd} \hat{V}_{\alpha}^*(x,i)\phi(x,i)\D x\,.
\end{align}
Since $\phi\in \cC_{c}^{\infty}(\Rd\times \mathbb{S})$ is arbitrary and $\hat{V}_{\alpha}^*\in \Sobl^{2,p}(\Rd\times \mathbb{S})$, it follows from \cref{ETC1.4EB} that 
the function $\hat{V}_{\alpha}^*\in \Sobl^{2,p}(\Rd\times \mathbb{S})\cap \cC_{b}(\Rd\times \mathbb{S})$ satisfies
\begin{align}\label{ETC1.4F}
&\trace\bigl(a(x,i)\grad^2 \hat{V}_{\alpha}^{*}(x,i)\bigr) + b(x,i,\hat{v}^*(x,i))\cdot \grad \hat{V}_{\alpha}^{*}(x,i) + c(x,i,\hat{v}^*(x,i))+\,\sum_{j \in \mathbb{S}} m_{ij}(x,\hat{v}^*(x,i))\,\hat{V}_{\alpha}^{*}(x,j)\nonumber\\ &= \alpha \hat{V}_{\alpha}^{*}(x,i)\,
\end{align}
As before, applying It$\hat{\rm o}$-Krylov formula (\cite{ABG-book}*{Lemma 5.1.4}) and using \cref{{ETC1.4F}}, we obtain
\begin{align}\label{ETC1.4G}
 \hat{V}_{\alpha}^{*}(x,i) \,=\, \Exp_{x,i}^{\hat{v}^*}\left[\int_0^{\infty} e^{-\alpha s}c(\hat{X}_s,\hat{S}_s, \hat{v}^*(\hat{X}_s,\hat{S}_s)) \D{s}\right] \,,
\end{align} where $(\hat{X},\hat{S})$ is the solution of SDE \cref{E1.1} corresponding to $\hat{v}^*$\,.
 
From \cref{TD1.2}, we know that $v_n^*\in \Usm$ is a minimizing selector of the HJB equation \cref{APOptDHJB1} of the approximated model, thus it follows that
\begin{align}\label{ETC1.4GA1}
\alpha V_{\alpha}^{n_k}(x,i) &\,=\, \min_{\zeta \in\Act}\left[\sL_{\zeta}^{n_k}V_{\alpha}^{n_k}(x,i) + c_{n_k}(x,i,\zeta)\right]\nonumber\\
& \,=\, \trace\bigl(a_{n_k}(x,i)\grad^2 V_{\alpha}^{n_k}(x,i)\bigr) + b_{n_k}(x,i,v_{n_k}^*(x,i))\cdot \grad V_{\alpha}^{n_k}(x,i)\nonumber\\
& + c_{n_k}(x,i,v_{n_k}^*(x,i))+\sum_{j \in \mathbb{S}} m_{ij}^{n_{k}}(x,v_{n_k}^*(x,j))V_{\alpha}^{n_{k}}(x,j),\,\,\,\, \text{a.e.\ }\, (x,i)\in\Rd\times \mathbb{S}
\end{align}
By standard Sobolev estimate (as in \cref{TH2.3}), it follows that for each $R>0$, $\norm{V_{\alpha}^{n_k}}_{\Sob^{2,p}(\sB_R\times \mathbb{S})} \leq \kappa_3$\,, for some constant $\kappa_3>0$ independent of $k$. Thus, one can extract a further sub-sequence denoted again by $\{V_{\alpha}^{n_k}\}$, without loss of generality) such that for some $\Tilde{V}_{\alpha}^*\in \Sobl^{2,p}(\Rd\times \mathbb{S})\cap \cC_{b}(\Rd\times \mathbb{S})$ (as in~\cref{ETC1.3BC}), such that the following convergences hold:
\begin{equation*}\label{ETC1.4H}
\begin{cases}
V_{\alpha}^{n_k}\to & \Tilde{V}_{\alpha}^*\quad \text{in}\quad \Sobl^{2,p}(\Rd\times \mathbb{S})\quad\text{(weakly)}\\
V_{\alpha}^{n_k}\to & \Tilde{V}_{\alpha}^*\quad \text{in}\quad \cC^{1, \beta}_{loc}(\Rd\times \mathbb{S})\quad\text{(strongly)}\,.
\end{cases}       
\end{equation*}
Following arguments analogous to those in \cref{TH2.3}, 
we multiply~\cref{ETC1.4GA1} by a suitable test function and let $k\to\infty$. 
Then, we deduce that $\Tilde{V}_{\alpha}^*\in \Sobl^{2,p}(\Rd\times \mathbb{S})\cap \cC_{b}(\Rd\times \mathbb{S})$ satisfies
\begin{align}\label{ETC1.4I}
\alpha \Tilde{V}_{\alpha}^{*}(x,i) &\,=\, \min_{\zeta \in\Act}\left[\sL_{\zeta}\Tilde{V}_{\alpha}^{*}(x,i) + c(x,i,\zeta)\right]\nonumber\\
& \,=\, \trace\bigl(a(x,i)\grad^2 \Tilde{V}_{\alpha}^{*}(x,i)\bigr) + b(x,i,\hat{v}^*(x,i))\cdot \grad \Tilde{V}_{\alpha}^{*}(x,i)\nonumber\\
&\quad+ c(x,i,\hat{v}^*(x,i))+\,\sum_{j \in \mathbb{S}} m_{ij}(x,\hat{v}^*(x,i))\,\Tilde{V}_{\alpha}^{*}(x,j)\,.
\end{align}
From the continuity results established in~\cref{TH2.3}, 
it follows that $\Tilde{V}_{\alpha}^{*}(x,i) = V_{\alpha}(x,i)$ for all $(x,i)\in\Rd \times \mathbb{S}$. Moreover, applying the It$\hat{\rm o}$-Krylov formula~(\cite{ABG-book}*{Lemma~5.1.4}) 
and using~\cref{ETC1.4I}, we obtain 
\begin{align}\label{ETC1.4J}
 \Tilde{V}_{\alpha}^{*}(x,i) \,=\, \Exp_{x,i}^{\hat{v}^*}\left[\int_0^{\infty} e^{-\alpha s}c(\hat{X}_s,\hat{S}_s, \hat{v}^*(\hat{X}_s,\hat{S}_s)) \D{s}\right] \,.
\end{align}
Since both $\hat{V}_{\alpha}^{*}$ and $\Tilde{V}_{\alpha}^{*}$ are continuous,
it follows from~\cref{ETC1.4G} and~\cref{ETC1.4J} that $\cJ_{\alpha}^{v_{n_k}^*}(x,i,c)$ (which is equal to $\Psi_{\alpha}^{n_k}(x,i)$) and $V_{\alpha}^{n_{k}}(x,i)$ converge to the same limit. This completes the proof.
\end{proof}
\begin{remark}
It is worth noting that the above analysis also establishes the continuity of the value function with respect to the control policy (under the considered topology).  
Indeed, the uniqueness of the solution to the corresponding PDE (coupled HJB equation) implies this continuity property.  
While in the diffusion setting a related observation was made earlier by Borkar~\cite{Bor-book},  
the present result is derived directly through a detailed optimality analysis.  
This continuity property has significant implications for the development of numerical schemes and for approximation results.
\end{remark}


\section{Analysis of Ergodic Cost}
We now consider the ergodic (long-run average) cost problem for the controlled 
regime-switching diffusion model introduced in \cref{Cost}. Our objective is to establish the robustness of the corresponding optimal controls. 
The associated optimal control problem for this cost criterion has been studied 
extensively in the literature; see, for example, \cite{Ghosh97}.

For this cost evolution criterion we will study the robustness problem under the assumption of Lyapunov stability.

\subsection{Analysis under Lyapunov stability}\label{Lyapunov stability} We assume the following Foster-Lyapunov condition on the dynamics.
\begin{itemize}
\item[\hypertarget{A7}{{(A7)}}]
\begin{itemize}
\item[(i)]
There exists a positive constant $\widehat{C}_0$, and a pair of inf-compact  functions $(\Lyap, h)\in \cC^{2}(\Rd\times \mathbb{S})\times\cC(\Rd\times \mathbb{S} \times \Act)$ (i.e., the sub-level sets $\{\Lyap(\cdot,i)\leq k\} \,,\{h(\cdot,i)\leq k\}$ are compact or empty sets in $\Rd$\,, $\Rd \times \Act$ respectively for each $k\in\RR,\,i\in \mathbb{S}$) such that
\begin{equation}\label{Lyap1}
\sL_{\zeta}\Lyap(x,i) \leq \widehat{C}_{0} - h(x,i, \zeta)\quad \forall\,\,\, (x,i,\zeta)\in \Rd\times \Act\,,
\end{equation} where  $h$ is locally Lipschitz continuous in its first argument uniformly with respect to rest. 
\item[(ii)] For each $n\in\NN$\,, we have
\begin{equation}\label{Lyap2}
\sL_{\zeta}^{n}\Lyap(x,i) \leq \widehat{C}_{0} - h(x,i, \zeta)\quad \forall\,\,\, (x,i, \zeta)\in \Rd\times \mathbb{S}\times \Act\,,
\end{equation} where the functions $\Lyap, h$ are as in \cref{Lyap1}\,. 
\end{itemize} 
\end{itemize} 
By combining the results of \cref{vd&hjb} and \cref{3.7.12}, we obtain a complete characterization of the optimal control in the ergodic case.
\begin{theorem}\label{TErgoOpt1}
Suppose that assumptions \hyperlink{A1}{{(A1)}}-\hyperlink{A5}{{(A5)}} and \hyperlink{A7}{{(A7)}}(i) hold. Then there exists a unique solution pair $(V^*, \rho)\in \Sobl^{2,p}(\Rd\times \mathbb{S})\cap \sorder(\Lyap)\times \RR$, with $V^*(0,1) = 0$\,, satisfying the ergodic HJB equation
\begin{equation}\label{EErgoOpt1A}
\rho = \min_{\zeta \in\Act}\left[\sL_{\zeta}V^*(x,i) + c(x,i,\zeta)\right]
\end{equation}
Moreover, we have
\begin{itemize}
\item[(i)]$\rho = \sE^*(c)$
\item[(ii)] a stationary Markov control $v^*\in \Usm$ is an optimal control (i.e., $\sE_{x,i}(c, v^*) = \sE^*(c)$) if and only if it satisfies
\begin{align}\label{EErgoOpt1B}
\min_{\zeta \in\Act}\left[\sL_{\zeta}V^*(x,i) + c(x,i,\zeta)\right] \,
&=\, \trace\bigl(a(x,i)\grad^2 V^*(x,i)\bigr) + b(x,i,v^*(x,i))\cdot \grad V^*(x,i)\,\nonumber\\
&+\, \sum_{j \in \mathbb{S}} m_{ij}(x,v^*(x,i)) V^*(x,j) + c(x,i, v^*(x,i))\,, 
\end{align}
\text{a.e.}\, $(x,i)\in\Rd\times \mathbb{S}\,.$
\item[(iii)] for any $v^*\in \Usm$ satisfying \cref{EErgoOpt1B}, we have
\begin{align*}\label{EErgoOpt1C}
V^*(x,i) \,=\, \lim_{r\downarrow 0}\Exp_{x,i}^{v^*}\left[\int_{0}^{\uuptau_{r}} \left( c(X_t,S_t, v^*(X_t,S_t)) - \sE^*(c)\right)\D t\right] \quad\forall\,\,\, x\in \Rd,\,i \in \mathbb{S}.
\end{align*}
\end{itemize} 
\end{theorem} 
Again, from \cref{vd&hjb} and \cref{3.7.12}, it follows that for each approximating model indexed by $n \in \NN$, the optimal control can be fully characterized as follows.
\begin{theorem}\label{TErgoOptApprox1}
Suppose that Assumptions \hyperlink{A6}{(A6)}(iii) and \hyperlink{A7}{(A7)}(ii) hold. Then the ergodic HJB equation
\begin{equation}\label{TErgoOptApprox1A}
\rho_n = \min_{\zeta \in\Act}\left[\sL_{\zeta}^nV^{n*}(x,i) + c_n(x,i,\zeta)\right]
\end{equation} admits unique solution $(V^{n*}, \rho_n)\in \Sobl^{2,p}(\Rd\times \mathbb{S})\cap \sorder(\Lyap)\times \RR$ satisfying $V^{n*}(0,1) = 0$\,.
Moreover, we have
\begin{itemize}
\item[(i)]$\rho_n = \sE^{n*}(c_n)$
\item[(ii)] a stationary Markov control $v_n^*\in \Usm$ is an optimal control (i.e., $\sE_{x,i}^n(c_n, v_n^{*}) = \sE^{n*}(c_n)$) if and only if it satisfies
\begin{align}\label{TErgoOptApprox1B}
\min_{\zeta \in\Act}\left[\sL_{\zeta}^n V^{n*}(x,i) + c_n(x,i,\zeta)\right] \,&=\, \trace\bigl(a_n(x,i)\grad^2 V^{n*}(x,i)\bigr) + b_n(x,i,v_n^*(x,i))\cdot \grad V^{n*}(x,i)\nonumber\\
&+\, \sum_{j \in \mathbb{S}} m_{ij}^n(x,v_n^*(x,i)) V^{n*}(x,j) + c(x,i, v_n^*(x,i))
\end{align}
for almost every $(x,i)\in\Rd\times \mathbb{S}.$
\item[(iii)] for any $v_n^*\in \Usm$ satisfying \cref{TErgoOptApprox1B}, we have
\begin{equation*}\label{TErgoOptApprox1C}
V^{n*}(x,i) \,=\, \lim_{r\downarrow 0}\Exp_{x,i}^{v_n^*}\left[\int_{0}^{\uuptau_{r}} \left( c_n(X_t^n,S_t^n, v_n^*(X_t^n,S_t^n)) - \sE^{n*}(c_n)\right)\D t\right] \quad\forall\, (x,i)\in \Rd\times \mathbb{S}.
\end{equation*}
\end{itemize} 
\end{theorem} 
Next, we establish that under the stability assumption \hyperlink{A7}{{(A7)}}, the optimal value $\sE^{n*}(c_n)$ of the approximating model converges to the optimal value $\sE^{*}(c)$ of the true model as $n \to \infty$.
\begin{theorem}\label{TH3.3}
Suppose that Assumptions \hyperlink{A1}{{(A1)}}-\hyperlink{A7}{{(A7)}} hold. Then, it follows that 
\begin{equation*}\label{ETErgoOptCont1}
\lim_{n\to\infty} \sE^{n*}(c_n) = \sE^{*}(c)\,.
\end{equation*}
\end{theorem}
\begin{proof}Since $\norm{c_n}_{\infty} \le M$, it follows that $\sE^{n*}(c_n) \le M$.
Moreover, by \cref{5.5.2}, there exist positive constants $\widehat{C}_1(R)$ (depending only on the radius $R>0$) and $\widehat{C}_2$ (independent of $R$), such that for all $\alpha > 0$, we have

\begin{equation}\label{ETErgoOptCont1A}
\|V_\alpha^n(\cdot,\cdot) - V_\alpha^n(0,1)\|_{\Sob^{2,p}(\sB_{R}\times \mathbb{S})} \leq \widehat{C}_1 \quad \text{and}\,\,\, \sup_{\sB_R\times \mathbb{S}}\alpha V_{\alpha}^{n} \leq \widehat{C}_2\,.
\end{equation}
By standard vanishing discount argument (see \cref{vd&hjb}) as $\alpha\to 0$ we obtain $V_\alpha^n(\cdot,\cdot) - V_\alpha^n(0,1) \to V^{n*}$ and $\alpha V_\alpha^n(0,1) \to \rho_n$\,. Consequently, the estimates \cref{ETErgoOptCont1A} yield $\|V^{n*}\|_{\Sob^{2,p}(\sB_{R}\times \mathbb{S})} \leq \widehat{C}_1$\,. Since the constant $\widehat{C}_1$ does not depend on $n$, the sequence ${V^{n*}}$ is uniformly bounded in $\Sob^{2,p}(\sB_R \times \mathbb{S})$. Therefore, by standard diagonalization argument and the Banach–Alaoglu theorem, we can extract a subsequence ${V^{n_k*}}$ such that for some $\widehat{V}^*\in \Sobl^{2,p}(\Rd\times \mathbb{S})$ (as in \cref{ETC1.3BC})
\begin{equation}\label{ETErgoOptCont1B}
\begin{cases}
V^{n_k*}\to & \widehat{V}^*\quad \text{in}\quad \Sobl^{2,p}(\Rd\times \mathbb{S})\quad\text{(weakly)},\\
V^{n_k*}\to & \widehat{V}^*\quad \text{in}\quad \cC^{1, \beta}_{loc}(\Rd\times \mathbb{S}) \quad\text{(strongly)}\,.
\end{cases} 
\end{equation}
Moreover, since $\rho_n \le M$, we can, without loss of generality, assume that along a further subsequence (denoted by the same index ${n_k}$), we have $\rho_{n_k} \to \widehat{\rho}$\, as\, $k \to \infty$.
Multiplying both sides of the equation \cref{TErgoOptApprox1A} by a test function $\phi$ and integrating over $\Rd$, we obtain

\begin{align*}
&\int_{\Rd}\trace\bigl(a_{n_k}(x,i)\grad^2 V^{n_k*}(x,i)\bigr)\phi(x,i)\D x \\
&+ \int_{\Rd} \min_{\zeta\in \Act} \{b_{n_k}(x,i,\zeta)\cdot \grad V^{n_k*}(x,i)+\, \sum_{j \in \mathbb{S}} m_{ij}^{n_k}(x,\zeta) V^{n_k*}(x,i) + c_{n_k}(x,i,\zeta)\}\phi(x,i)\D x \\
&= \int_{\Rd} \rho_{n_k}\phi(x,i)\D x\,.
\end{align*}

Following the argument used in \cref{TH2.3}, and employing the convergence results from \cref{ETErgoOptCont1B}, letting $k \to \infty$, we deduce that $\widehat{V}^*\in \Sobl^{2,p}(\Rd\times \mathbb{S})$ satisfies
\begin{equation}\label{TErgoOptCont1C}
\widehat{\rho} = \min_{\zeta \in\Act}\left[\sL_{\zeta}\widehat{V}^*(x,i) + c(x,i,\zeta)\right]\,.
\end{equation}

Next we show that $\widehat{V}^*\in \sorder{(\Lyap)}$. Since $\sup_{n}\|c_n\| \leq M$, it follows that $1 + \tilde{c} \in \sorder{(h)}$, where $\tilde{c} \,\df\, \sup_n c_n$\,. Moreover, as $h$ is inf-compact, there exists $r > 0$ such that $ \inf_{\zeta\in\Act} h(x,i,\zeta)-\displaystyle{\widehat{C}_0  \geq \epsilon}$ for all $(x,i)\in \sB_r^c\times \mathbb{S}$, for some $\epsilon>0$. Let $(X_t^n, S_t^n)$ denote the solution to \cref{ASE1.1} under a stationary Markov control $v \in \Usm$. Then, by \cref{Lyap2} and applying the It$\hat{\rm o}$-Krylov formula (\cite{ABG-book}*{Lemma~5.1.4}), we deduce that for any $v \in \Usm$ and  $(x,i)\in \sB_r^c\cap\sB_R\times \mathbb{S}$, the following holds:
\begin{align*}
\Exp_{x,i}^{v}\left[\Lyap(X_{\uuptau_{r}^n \wedge \uptau_{R}^n}^n, S_{\uuptau_{r}^n \wedge \uptau_{R}^n}^n)\right] - \Lyap(x,i)
&= \Exp_{x,i}^{v}\left[\int_{0}^{\uuptau_{r}^n \wedge \uptau_{R}^n} \sL_{v}^n \Lyap(X_s^n,S_s^n) \D s\right]\\ 
& \leq \Exp_{x,i}^{v}\left[\int_0^{\uuptau_{r}^n \wedge \uptau_{R}^n} (\widehat{C}_0 - h(X_s^n,S_s^n, v(X_s^n,S_s^n)))\D s \right]\\
&\leq -\epsilon \Exp_{x,i}^{v}\left[\uuptau_{r}^n \wedge \uptau_{R}^n\right]\,,  
\end{align*} where $\uuptau_{r}^n \df \inf\{t\geq 0: (X_t^n,S_t^n)\in \sB_r\times \mathbb{S}\}$ and $\uptau_{R}^n \df \inf \{t\geq 0:(X_t^n,S_t^n)\in \sB_R^c\times \mathbb{S}\}$\,. Letting $R\to \infty$ and applying Fatou’s lemma, we obtain
\begin{equation}\label{3.13}
\Exp_{x,i}^{v}\left[\uuptau_{r}^n\right] \leq \frac{1}{\epsilon} \Lyap(x,i)\quad \forall\,\,\, (x,i)\in \sB_r^c\times \mathbb{S}\,\,\,\text{and}\,\,\, n\in\NN\,. 
\end{equation}

Again, by It$\hat{\rm o}$-Krylov formula (\cite{ABG-book}*{Lemma 5.1.4}), for any $v\in \Usm$ and $(x,i)\in \sB_r^c\cap\sB_R\times \mathbb{S}$ we have
\begin{align*}
\Exp_{x,i}^{v}\left[\Lyap(X_{\uuptau_{r}^n \wedge \uptau_{R}^n}^n, S_{\uuptau_{r}^n \wedge \uptau_{R}^n}^n)\right] - \Lyap(x,i)
&= \Exp_{x,i}^{v}\left[\int_{0}^{\uuptau_{r}^n \wedge \uptau_{R}^n} \sL_{v}^n \Lyap(X_s^n,S_s^n) \D s\right]\\ 
&\leq  \Exp_{x,i}^{v}\left[\int_0^{\uuptau_{r}^n \wedge \uptau_{R}^n} (\widehat{C}_0 - h(X_s^n,S_s^n, v(X_s^n,S_s^n)))\D s \right]\,,  
\end{align*}
Therefore
\[ \Exp_{x,i}^{v}\left[\int_{0}^{\uuptau_{r}^n \wedge \uptau_{R}^n} h(X_s^n,S_s^n,v(X_s^n,S_s^n)) \D s\right] \leq \Lyap(x,i)+\widehat{C}_{0}\Exp_{x,i}^{v}\left[\uuptau_{r}^n \wedge \uptau_{R}^n\right]\]
Applying Fatou’s lemma and letting $R \to \infty$, and using \cref{3.13}, we obtain
\begin{equation*}
\sup_{n\in\NN}\sup_{v\in\Usm}\Exp_{x,i}^{v}\left[\int_0^{\uuptau_{r}^n} h(X_s^n,S_s^n, v(X_s^n,S_s^n)) \D s\right] \leq \widehat{M}_1 \Lyap(x,i)\,,  
\end{equation*} for some positive constant $\widehat{M}_1$\,. Proceeding as in the proof of \cite{ABG-book}*{Lemma~3.7.2(i)}, we further obtain
\begin{equation}\label{TErgoOptCont1E}
\sup_{n\in \NN}\sup_{v\in\Usm} \Exp_{x,i}^{v}\left[\int_{0}^{\uuptau_{r}^n}(1 + \tilde{c}(X_s^n,S_s^n, v(X_s^n,S_s^n)))\D s\right]\in \sorder{(\Lyap)}\,.
\end{equation}

Next, following the argument of \cref{vd&hjb} (see also Eq.~(3.7.47) in \cite{ABG-book}), we have
  \begin{equation}\label{TErgoOptCont1F}
V^{n*}(x,i) \,\leq\, \sup_{v\in\Usm}\Exp_{x,i}^{v}\left[\int_{0}^{\uuptau_{r}^n} \left( c_n(X_t^n,S_t^n, v(X_t^n,S_t^n)) - \sE^{n*}(c_n)\right)\D t + V^{n*}(X_{\uuptau_{r}^n},S_{\uuptau_{r}^n})\right]\,
\end{equation}
Recall that for $p \geq d+1$ the Sobolev space $\Sob^{2,p}(\sB_{R}\times \mathbb{S})$ is compactly embedded in $\cC^{1, \beta}(\bar{\sB}_R\times \mathbb{S})$ where $0< \beta < 1 - \frac{d}{p}$\,. Since $\|V^{n*}\|_{\Sob^{2,p}(\sB_{R}\times \mathbb{S})} \leq \widehat{C}_1$ for some positive constant $\widehat{C}_1$ depending only on $R$, it follows that $\sup_{n\in\NN}\sup_{\sB_r\times\mathbb{S}}|V^{n*}| \leq \widehat{M}_2$, for some $\widehat{M}_2 (>0)$. Moreover, since $\sE^{n*}(c_n) \leq \|c_n\|_{\infty} \leq M$,  combining \cref{TErgoOptCont1F} with \cref{3.7.42} yields
\begin{equation}\label{TErgoOptCont1H}
|V^{n*}(x,i)| \,\leq\, M\sup_{n\in\NN}\sup_{v\in\Usm}\Exp_{x,i}^{v}\left[\int_{0}^{\uuptau_{r}^n} \left( \tilde{c}(X_t^n,S_t^n, v(X_t^n,S_t^n)) + 1\right)\D t + \sup_{n\in\NN}\sup_{\sB_r\times\mathbb{S}}|V^{n*}|\right]\,.
\end{equation}
Therefore, combining \cref{ETErgoOptCont1B}, \cref{TErgoOptCont1E}, and \cref{TErgoOptCont1H}, we conclude that $\widehat{V}^*\in \sorder{(\Lyap)}$\,.
Since $(\widehat{V}^*, \widehat{\rho})\in \Sobl^{2,p}(\Rd\times \mathbb{S})\cap \sorder(\Lyap)\times \RR$ satisfies $\widehat V^*(0,1) = 0$ and the ergodic HJB equation \cref{EErgoOpt1A}, the uniqueness result of \cref{TErgoOpt1} implies that
 $(\widehat{V}^*, \widehat{\rho}) \equiv (V^*, \rho)$. This completes the proof of the theorem.      
\end{proof}
\begin{remark} Note that
    rewriting \cref{TErgoOptCont1C}, we get 
\(
\trace\bigl(a(x,i)\grad^2 \widehat{V}^*(x,i)\bigr) =  \hat{f}(x,i)\,,\quad \text{a.e.}\,\, (x,i)\in\Rd\times \mathbb{S}\,,
\) where 
\(
\hat{f}(x,i) = - \inf_{\zeta\in \Act}\left[b(x,i,\zeta)\cdot \grad \widehat{V}^*(x,i)+\sum_{j \in \mathbb{S}} m_{ij}(x,\zeta) \widehat{V}^*(x,j) + c(x,i,\zeta) - \widehat{\rho} \right]\,.
\)
In view of \cref{ETErgoOptCont1B} and assumptions \hyperlink{A1}{{(A1)}}–\hyperlink{A2}{{(A2)}}, it follows that $\hat{f} \in \cC^{0,\beta}_{\mathrm{loc}}(\Rd\times\mathbb{S})$ for some $0 < \beta < 1 - \tfrac{d}{p}$. Therefore, by the elliptic regularity results of \cite{CL89}*{Theorem~3} and \cite{GilTru}*{Theorem~9.19}, we conclude that $\widehat{V}^* \in \cC^{2}(\Rd \times \mathbb{S})$. But for our analysis $\widehat{V}^* \in \Sobl^{2,p}(\Rd \times \mathbb{S})$ is enough.
\end{remark}

We now present an existence and uniqueness result for the Poisson equation associated with a fixed stationary Markov control. This theorem provides the analytical foundation for the ergodic control problem and will later play a key role in establishing robustness.
\begin{theorem}\label{TErgoExisPoiss1}
Suppose that assumptions \hyperlink{A1}{{(A1)}}-\hyperlink{A5}{{(A5)}} and \hyperlink{((A7)}{(A7)}(i) hold. Then, for each $v\in \Usm$ there exists a unique  pair $(V^v, \rho^{v})\in \Sobl^{2,p}(\Rd\times \mathbb{S})\cap \sorder(\Lyap)\times \RR$ for any $p >1$ satisfying
\begin{equation}\label{TErgoExisPoiss1A}
\rho^{v} = \sL_{v}V^v(x,i) + c(x,i, v(x,i))\quad\text{with}\quad V^v(0,1) = 0\,.
\end{equation}
Furthermore, the following hold:
\begin{itemize}
\item[(i)]$\rho^{v} = \sE_{x,i}(c, v)$
\item[(ii)] for all $(x,i)\in\Rd \times \mathbb{S}$, we have
\begin{equation}\label{TErgoExisPoiss1B}
V^v(x,i) \,=\, \lim_{r\downarrow 0}\Exp_{x,i}^{v}\left[\int_{0}^{\uuptau_{r}} \left( c(X_t,S_t, v(X_t,S_t)) - \sE_{x,i}(c, v)\right)\D t\right]\,.
\end{equation}
\end{itemize} 
\end{theorem} 
\begin{proof}
The existence of a solution pair $(V^v, \rho^{v})\in \Sobl^{2,p}(\Rd\times \mathbb{S})\cap \sorder(\Lyap)\times \RR$ for any $p >1$ satisfying (i) and (ii) follows from \cref{vd&hjb} (2)\,.We now prove the uniqueness of the solution to \cref{TErgoExisPoiss1A}. Let $(\bar{V}^v, \bar{\rho}^{v})\in \Sobl^{2,p}(\Rd\times \mathbb{S})\cap \sorder(\Lyap)\times \RR$ for any $p >1$ be another solution pair of \cref{TErgoExisPoiss1A} with $\bar{V}^v(0,1) = 0$. Applying the It$\hat{\rm o}$-Krylov formula (\cite{ABG-book}*{Lemma 5.1.4}), for $R>0$ we obtain
\begin{align*}\label{TErgoExisPoiss1C}
\Exp_{x,i}^{v}\left[\bar{V}^v(X_{T\wedge\uptau_{R}},S_{T\wedge\uptau_{R}})\right] - \bar{V}^v(x,i) &= \Exp_{x,i}^{v}\left[\int_{0}^{T\wedge\uptau_{R}} \sL_{v} \bar{V}^v(X_s,S_s) \D s\right]\nonumber\\
& = \Exp_{x,i}^{v}\left[\int_{0}^{T\wedge\uptau_{R}} \left(\bar{\rho}^{v} - c(X_s,S_s, v(X_s,S_s))\right)\D s \right]\,.  
\end{align*}
Note that 
\begin{equation*}
\int_{0}^{T\wedge\uptau_{R}} \left(\bar{\rho}^{v} - c(X_s,S_s, v(X_s,S_s))\right)\D s = \int_{0}^{T\wedge\uptau_{R}} \bar{\rho}^{v} - \int_{0}^{T\wedge\uptau_{R}}c(X_s,S_s, v(X_s,S_s))\D s
\end{equation*} 
Letting $R \to \infty$ and applying the monotone convergence theorem, we get
\begin{equation*}
\lim_{R\to\infty}\Exp_{x,i}^{v}\left[\int_{0}^{T\wedge\uptau_{R}} \left(\bar{\rho}^{v} - c(X_s,S_s, v(X_s,S_s))\right)\D s \right] = \Exp_{x,i}^{v}\left[\int_{0}^{T} \left(\bar{\rho}^{v} - c(X_s,S_s, v(X_s,S_s))\right)\D s \right]\,.
\end{equation*} 
Since $\bar{V}^v \in \sorder{(\Lyap)}$, in view of \cref{3.7.2} (2), letting $R\to\infty$, yields
\begin{align}\label{TErgoExisPoiss1D}
\Exp_{x,i}^{v}\left[\bar{V}^v(X_{T},S_{T})\right] - \bar{V}^v(x,i) = \Exp_{x,i}^{v}\left[\int_{0}^{T} \left(\bar{\rho}^{v} - c(X_s,S_s, v(X_s,S_s))\right)\D s \right]\,.  
\end{align} Moreover, from \cref{3.7.2} (2) we have
\begin{equation*}
\lim_{T\to\infty}\frac{\bar{V}^v(X_{T},S_{T})}{T} = 0\,. 
\end{equation*}
Dividing both sides of \cref{TErgoExisPoiss1D} by $T$ and letting $T \to \infty$, we obtain
\begin{align*}
\bar{\rho}^{v} = \limsup_{T\to \infty}\frac{1}{T}\Exp_{x,i}^{v}\left[\int_{0}^{T} \left(c(X_s,S_s, v(X_s,S_s))\right)\D s \right]\,.
\end{align*}This implies that $\bar{\rho}^{v} = \rho^{v}$\,. Next, using \cref{TErgoExisPoiss1A} and again applying the It$\hat{\rm o}$–Krylov formula (\cite{ABG-book}*{Lemma~5.1.4}), we obtain 
\begin{align}\label{TErgoExisPoiss1E}
\bar{V}^v(x,i)\,=\, \Exp_{x,i}^{v}\left[\int_0^{\uuptau_{r}\wedge \uptau_{R}} \left(c(X_t,S_t, v(X_t,S_t)) - \bar{\rho}^{v}\right) \D{t} + \bar{V}^{v}\left(X_{\uuptau_{r}\wedge \uptau_{R}},S_{\uuptau_{r}\wedge \uptau_{R}}\right)\right]\,.
\end{align} Furthermore, by \cref{Lyap1} and the It$\hat{\rm o}$–Krylov formula (\cite{ABG-book}*{Lemma~5.1.4}), we have 
\begin{equation*}
\Exp_{x,i}^{v}\left[\Lyap\left(X_{\uptau_{R}},S_{\uptau_{R}}\right)\Ind_{\{\uuptau_{r}\geq \uptau_{R}\}}\right]\leq \widehat{C}_0 \Exp_{x,i}^{v}\left[\uuptau_{r}\right] + \Lyap(x,i)\quad \forall \,\,\, r <|x|<R\,.
\end{equation*} Since $\bar{V}^{v} \in \sorder(\Lyap)$, the above estimate implies
\begin{equation*}
\liminf_{R\to\infty}\Exp_{x,i}^{v}\left[\bar{V}^{v}\left(X_{\uptau_{R}},S_{\uptau_{R}}\right)\Ind_{\{\uuptau_{r}\geq \uptau_{R}\}}\right] = 0\,.
\end{equation*}
Letting $R \to \infty$ and applying Fatou’s lemma in \cref{TErgoExisPoiss1E}, we obtain
\begin{align*}
\bar{V}^v(x,i)&\,\geq\, \Exp_{x,i}^{v}\left[\int_0^{\uuptau_{r}} \left(c(X_t,S_t, v(X_t,S_t)) - \bar{\rho}^{v}\right) \D{t} +\bar{V}^{v}\left(X_{\uuptau_{r}},S_{\uuptau_{r}}\right)\right]\nonumber\\
&\,\geq\, \Exp_{x,i}^{v}\left[\int_0^{\uuptau_{r}} \left(c(X_t,S_t, v(X_t,S_t)) - \bar{\rho}^{v}\right) \D{t}\right] +\inf_{\sB_r\times\mathbb{S}}\bar{V}^{v}\,.
\end{align*}Since $\bar{V}^{v}(0,1) =0$, letting $r\to 0$ gives
\begin{align}\label{TErgoExisPoiss1F}
\bar{V}^v(x,i)\,\geq\, \limsup_{r\downarrow 0}\Exp_{x,i}^{v}\left[\int_0^{\uuptau_{r}} \left(c(X_t,S_t, v(X_t,S_t)) - \bar{\rho}^{v}\right) \D{t} \right]\,.
\end{align}
Since $\bar{\rho}^{v} = \rho^{v}$, combining \cref{TErgoExisPoiss1B} and \cref{TErgoExisPoiss1F} gives $V^v - \bar{V}^v \leq 0$ in $\Rd\times \mathbb{S}$. Finally, since both $(V^{v}, \rho^{v})$ and $(\bar{V}^{v}, \bar{\rho}^{v})$ satisfy \cref{TErgoExisPoiss1A}, we have
\begin{align*}\label{3.23}
  0 &= \sL_{v}\left(V^v - \bar{V}^v\right)(x,i) \nonumber\\
  &= \trace\bigl(a(x,i)\grad^2 \left(V^v - \bar{V}^v\right)(x,i) + b(x,i,v(x,i))\cdot \grad \left(V^v - \bar{V}^v\right)(x,i)\nonumber\\
&\quad\qquad+m_{ii}(x,v(x,i))\left(V^v - \bar{V}^v\right)(x,i)+\, \sum_{j \neq i} m_{ij}(x,v(x,i)) \left(V^v - \bar{V}^v\right)(x,j)\nonumber\\
&\leq \trace\bigl(a(x,i)\grad^2 \left(V^v - \bar{V}^v\right)(x,i)+ b(x,i,v(x,i))\cdot \grad \left(V^v - \bar{V}^v\right)(x,i)\nonumber\\
&\qquad \qquad+m_{ii}(x,v(x,i))\left(V^v - \bar{V}^v\right)(x,i)\,,\quad \text{a.e.}\, (x,i)\in\Rd\times \mathbb{S}.
\end{align*}
Therefore, by the strong maximum principle (\cite{GilTru}*{Theorem~9.6}), it follows that $V^{v} = \bar{V}^{v}$, which proves the uniqueness of the solution.
\end{proof}
\noindent
\begin{remark}
The above theorem guarantees that for every stationary Markov control $v \in \Usm$, the associated Poisson equation admits a unique solution pair $(V^{v}, \rho^{v})$ within the class $\Sobl^{2,p}(\Rd \times \mathbb{S}) \cap \sorder(\Lyap) \times \RR$.
This establishes the well-posedness of the ergodic equation under each markov control and ensures that the average cost $\rho^{v}$ and potential function $V^{v}$ are uniquely determined by the underlying dynamics.
This characterization serves as the analytical basis for the ensuing robustness analysis, which in particular establishes the
stability of $(V^{v}, \rho^{v})$ under perturbations of the system coefficients.
\end{remark}
We now turn to the main robustness result. Recall that for each approximating model (indexed by $n$) we denote by $v_n^*$ an optimal stationary Markov ergodic control (see \cref{TErgoOptApprox1}), and by
$\sE_{x,i}(c, v_{n}^*)$ the long-run average cost under control
$v$ starting from $(x,i)$ . Our aim is to prove that $\sE_{x,i}(c, v_{n}^*)\to \sE^*(c)$ as $n\to \infty$
where 
$\sE^*(c)$ denotes the optimal ergodic value for the true model.
\begin{theorem}\label{ErgodLyapRobu1}
Suppose that Assumptions \hyperlink{A1}{{(A1)}}-\hyperlink{A7}{{(A7)}} hold. Then,
\begin{equation*}\label{ErgodLyapRobu1A}
\lim_{n\to\infty} \inf_{(x,i)\in\Rd\times \mathbb{S}}\sE_{x,i}(c, v_{n}^*) = \sE^{*}(c)\,.
\end{equation*}
\end{theorem}
\begin{proof} We follow a proof strategy analogous to that used in \cref{TH2.4} for the discounted case. From \cref{TErgoExisPoiss1}, for each $n \in \NN$, there exists a unique pair $(V^{v_{n}^*}, \rho^{v_{n}^*})\in \Sobl^{2,p}(\Rd\times \mathbb{S})\cap\sorder{(\Lyap)}\times \RR$, $1< p < \infty$, with $V^{v_{n}^*}(0,1) = 0$, satisfying
\begin{equation}\label{ErgodLyapRobu1B}
\rho^{v_{n}^*} = \left[\sL_{v_{n}^*}V^{v_{n}^*}(x,i) + c(x,i,{v_{n}^*}(x,i))\right]
\end{equation}
From \cref{5.5.2}, there exists a constant $\hat{\kappa}_1 > 0$, independent of $n \in \NN$, such that \(\norm{V^{v_{n}^*}}_{\Sob^{2,p}(\sB_R\times \mathbb{S})}\leq \hat{\kappa}_1\). By the Banach–Alaoglu theorem and standard diagonalization argument (as in \cref{ETC1.3BC}), we deduce the existence of $\tilde{V}\in \Sobl^{2,p}(\Rd\times \mathbb{S})$ such that, along a subsequence
\begin{equation*}\label{ErgodLyapRobu1C}
\begin{cases}
V^{v_{n_k}^*}\to & \tilde{V}\quad \text{in}\quad \Sobl^{2,p}(\Rd\times \mathbb{S})\quad\text{(weakly)}\\
V^{v_{n_k}^*}\to & \tilde{V}\quad \text{in}\quad \cC^{1, \beta}_{loc}(\Rd\times \mathbb{S})\quad\text{(strongly)}\,.
\end{cases}       
\end{equation*}
for some $0 < \beta < 1 - \tfrac{d}{p}$.
Since $\rho^{v_{n}^*} \leq M$, there exists a further subsequence (denoted by the same index) such that
$\rho^{v_{n_k}^*}\to \tilde{\rho}$ as $k\to \infty$\,. As $\Usm$ is compact, we may also assume $v_{n_k}^* \to \tilde{v}^*$ as $k \to \infty$. 
Multiplying \cref{ErgodLyapRobu1B} by a test function and passing to the limit $k \to \infty$, it follows that $(\tilde{V}, \tilde{\rho})\in \Sobl^{2,p}(\Rd\times \mathbb{S})\times \RR$, \, $1< p < \infty$ satisfies 
\begin{equation*}\label{ErgodLyapRobu1D}
\tilde{\rho} = \left[\sL_{\tilde{v}^*}\tilde{V}(x,i) + c(x,i,{\tilde{v}^*}(x,i))\right]
\end{equation*}
Since $V^{v_{n_k}^*}(0,1) = 0$ for all $k \in \NN$, we have $\tilde{V}(0,1) = 0$.
Using the estimate $\norm{V^{v_{n}^*}}_{\Sob^{2,p}(\sB_R\times \mathbb{S})}\leq \hat{\kappa}_1$ and arguing as in the proof of \cref{TH3.3}, we obtain
\begin{equation*}\label{ErgodLyapRobu1E}
|\tilde{V}(x,i)| \,\leq\, M\sup_{v\in\Usm}\Exp_{x,i}^{v}\left[\int_{0}^{\uuptau_{r}} \left( c(X_t,S_t, v(X_t,S_t)) + 1\right)\D t + \sup_{n\in\NN}\sup_{\sB_r\times \mathbb{S}}|V^{v_n^*}|\right] \in \sorder{(\Lyap)}\,.
\end{equation*}
Hence, by the uniqueness result of \cref{TErgoExisPoiss1}, we conclude that $(\tilde{V}, \tilde{\rho})\equiv (V^{\tilde{v}^*}, \rho^{\tilde{v}^*})$.

By the triangle inequality,
\begin{equation*}
|\sE_{x,i}(c,v_{n_k}^*) - \sE^*(c)| \leq |\sE_{x,i}(c,v_{n_k}^*) - \sE^*(c_{n_k},v_{n_k}^*)| + |\sE^*(c_{n_k},v_{n_k}^*) - \sE^*(c)|\,. 
\end{equation*}
From \cref{TH3.3}, we know that $|\sE^*(c_{n_k},v_{n_k}^*) - \sE^*(c)| \to 0$ as $k \to \infty$.
Thus, it remains to show that $|\sE_{x,i}(c,v_{n_k}^*) - \sE^*(c_{n_k},v_{n_k}^*)|\to 0$. For any minimizing selector $v_{n_k}^*\in \Usm$ of \cref{TErgoOptApprox1A}, we have
\begin{equation}\label{ErgodLyapRobu1F}
\rho_{n_k} = \left[\sL_{v_{n_k}^*}^{n_k}V^{n_k}(x,i) + c_{n_k}(x,i,v_{n_k}^*(x,i))\right]\,.
\end{equation} 
By \cref{ETErgoOptCont1A}, there exists a constant $\hat{\kappa} > 0$, independent of $k \in \NN$, such that
\begin{equation*}\label{ErgodLyapRobu1G}
\norm{V^{n_k}}_{\Sob^{2,p}(\sB_R\times \mathbb{S})}\leq \hat{\kappa}
\end{equation*}
Hence, by the Banach–Alaoglu theorem and standard diagonalization argument (see \cref{ETC1.3BC}), there exists $\tilde{V}^* \in \Sobl^{2,p}(\Rd \times \mathbb{S})$ such that, along a subsequence,
\begin{equation*}\label{ErgodLyapRobu1H}
\begin{cases}
V^{n_k}\to & \tilde{V}^*\quad \text{in}\quad \Sobl^{2,p}(\Rd\times \mathbb{S})\quad\text{(weakly)}\\
V^{n_k}\to & \tilde{V}^*\quad \text{in}\quad \cC^{1, \beta}_{loc}(\Rd\times \mathbb{S}) \quad\text{(strongly)}\,.
\end{cases}       
\end{equation*}
Since $\rho_{n_k} \leq M$, along a further subsequence (denoted by the same index) $\rho_{n_k} \to \tilde{\rho}^*$.
As $v_{n_k}^* \to \tilde{v}^*$ in $\Usm$, multiplying both sides of \cref{ErgodLyapRobu1F} by a test function and letting $k \to \infty$, we find that $(\tilde{V}^*, \tilde{\rho}^*) \in \Sobl^{2,p}(\Rd \times \mathbb{S}) \times \RR$ satisfies
\begin{equation*}\label{ErgodLyapRobu1I}
\tilde{\rho}^* = \left[\sL_{\tilde{v}^*}\tilde{V}^*(x,i) + c(x,i,\tilde{v}^*(x,i))\right]\,.
\end{equation*}
Arguing as in \cref{TH3.3}, we have $\tilde{V}^* \in \sorder{(\Lyap)}$.
Therefore, by the uniqueness of the solution to \cref{ErgodLyapRobu1F} (see \cref{TErgoExisPoiss1}), it follows that $(\tilde{V}^*, \tilde{\rho}^*) \equiv (V^{\tilde{v}^*}, \rho^{\tilde{v}^*})$.
Since both $\rho^{v_{n_k}^*}=\sE_{x,i}(c,v_{n_k}^*)$ and $\rho_{n_k}=\sE^*(c_{n_k},v_{n_k}^*)$ converge to the same limit $\rho^{\tilde{v}^*}$, we conclude that $|\sE_{x,i}(c,v_{n_k}^*) - \sE^*(c_{n_k},v_{n_k}^*)| \to 0$ as $k \to \infty$. This completes the proof of the theorem.
\end{proof}
\section{Finite Horizon Cost}

In this section, we establish the robustness of the finite-horizon optimal control problem.
Unlike the discounted or ergodic cases, the value function here depends explicitly on time and satisfies a parabolic system of coupled HJB equations with a prescribed terminal condition.



The following theorem provides a complete characterization of the finite-horizon optimal control $v\in\Um$ within the class of Markov policies.
\begin{theorem}\label{thm:finite-horizon-TM-verification}
    Under assumptions \hyperlink{A1}{{(A1)}}-\hyperlink{A5}{{(A5)}}, the finite-horizon HJB equation
   \begin{equation}\label{FH-hjb}
    \begin{aligned}
\partial_t\psi(t,x,i) + \inf_{\zeta \in \Act} \left[\sL_\zeta \psi(t,x,i) + c(x,i,\zeta)\right] &= 0, \quad\forall\, (t,x,i)\in (0,T)\times \Rd\times \mathbb{S}\\
\psi(T,x,i) &= c_{_{T}}(x,i),\quad\forall\, (x,i)\in \Rd\times \mathbb{S}
\end{aligned} 
\end{equation}
admits a unique solution $\psi \in \Sobl^{1,2,p}((0,T)\times \Rd\times \mathbb{S})
$, for some $p > d+2$.  Moreover, 
\begin{itemize}
    \item[(1)] \(\psi(0,x,i)=\cJ^*_{T}(x,i,c)\) represents the optimal finite-horizon cost.
    \item[(2)]  a markov policy $v^*\in \Um$ is optimal if and only if it attains the pointwise minimum in \cref{FH-hjb}, i.e.,
  \begin{align*}\label{ftmmin}
&b(x,i,v^*(x,i))\cdot \grad \psi(t,x,i) \,+\, \sum_{j \in \mathbb{S}} m_{ij}(x,v^*(x,i))\psi(t,x,j)+ c(x,i,v^*(x,i))\nonumber\\ &= \min_{\zeta\in \Act}\bigg[ b(x,i,\zeta)\cdot \grad \psi(t,x,i) \,+\, \sum_{j \in \mathbb{S}} m_{ij}(x,\zeta) \psi(t,x,j)+c(x,i,\zeta)\bigg]\, \quad\text{a.e.}\,\,\, (x,i)\in\Rd\times \mathbb{S}\,.
\end{align*}
\end{itemize}
\end{theorem}
\begin{proof}
    Assumptions \hyperlink{A1}{{(A1)}}-\hyperlink{A5}{{(A5)}} meet the conditions of \cite{DTSB}*{Theorem 1}. Hence, by \cite{DTSB}*{Theorem 1}(see footnote-8, page 9) for each ball $\cB_R$ of radius $R$, there exists a unique solution  $\psi_R \in \Sob^{1,2,p}((0,T)\times \cB_R\times \mathbb{S})$ of \cref{FH-hjb}. Applying standard diagonalization argument and the Banach–Alaoglu theorem, and letting $R\to\infty$ we obtain a unique global solution $\psi \in \Sobl^{1,2,p}((0,T)\times \Rd\times \mathbb{S})$ of \cref{FH-hjb}. Finally, by It$\hat{\rm o}$–Krylov formula (\cite{ABG-book}*{Lemma~5.1.4}) as in \cite{FH}*{Theorem~3.5} and \cite{FH}*{Proposition~3.6}, one can prove $(1)$ and $(2)$.
\end{proof}

Similarly, for each $n\in\NN$ corresponding to the approximating models, we obtain the following complete characterization of the finite-horizon optimal control.
 
\begin{theorem}\label{thm:finite-horizon-APM-verification}
Under assumption \hyperlink{A6}{{(A6)}}(iii), for each $n\in\NN$, there exists a unique solution  $\psi_n \in \Sobl^{1,2,p}((0,T)\times\Rd\times \mathbb{S})$, $p > d+2$ to the finite-horizon HJB equation associated with the $n$-th approximating model:
\begin{equation}
    \begin{aligned}\label{FHAP-hjb}
\partial_t\psi_n(t,x,i) + \inf_{\zeta \in \Act}\Big[\sL_{\zeta} \psi_n(t,x,i) + c_n(x,i,\zeta)\Big] &= 0,\quad\forall\, (t,x,i)\in (0,T)\times \Rd\times \mathbb{S}\\
\psi_n(T,x,i) &= c_{_{T}}(x,i), \quad\forall\, (x,i)\in \Rd\times \mathbb{S}
\end{aligned} 
\end{equation}
Moreover,
\begin{itemize}
    \item[(1)] \(\psi_n(0,x,i)=\cJ^*_{T,n}(x,i,c_n)\) represents the optimal cost for the $n$-th approximating model.
    \item[(2)]  a markov policy $v^*_n\in \Um$ is optimal control if and only if it minimizes the Hamiltonian in \cref{FHAP-hjb} pointwise, that is,
  \begin{align*}
&b_n(x,i,v^*_n(x,i))\cdot \grad \psi_n(t,x,i) \,+\, \sum_{j \in \mathbb{S}} m_{ij}^n(x,v^*_n(x,i))\psi_n(t,x,j)\,+ c_n(x,i,v^*_n(x,i))\nonumber\\ &= \min_{\zeta\in \Act}\Big[ b_n(x,i,\zeta)\cdot \grad \psi_n(t,x,i) \,+\, \sum_{j \in \mathbb{S}} m_{ij}^n(x,\zeta) \psi_n(t,x,j)+ c_n(x,i,\zeta)\Big]\, \quad\text{a.e.}\,\,\, (x,i)\in\Rd\times \mathbb{S}\,.
\end{align*}
\end{itemize}
\end{theorem}
Next, we establish the continuity result, showing that as the approximating models converge to the true model, the corresponding optimal finite-horizon values also converge to the optimal value of the true model. 
\begin{theorem}[Convergence of Optimal Value]\label{thm:finite-horizon-conv}
Suppose assumptions  \hyperlink{A1}{{(A1)}}-\hyperlink{A6}{{(A6)}} hold. Then
\[
\lim_{n\to\infty} \cJ_{T,n}^*(x,i,c_n) = \cJ_T^*(x,i,c).
\]
\end{theorem}

\begin{proof}
Let $v_n^*$ be any minimizing selector of the Hamiltonian in \eqref{FHAP-hjb}, we have
\begin{equation}\label{FHAP-hjb1}
\begin{aligned}
\partial_t\psi_n(t,x,i) + \sL_{v_n^*}^n \psi_n + c_n(x,i,v_n^*(t,x,i)) &= 0,\quad\forall\, (t,x,i)\in (0,T)\times \Rd\times \mathbb{S}\\
\psi_n(T,x,i) &= c_{_{T}}(x,i),\quad\forall\, (x,i)\in  \Rd\times \mathbb{S}
\end{aligned}
\end{equation}
Thus, by \cref{thm:finite-horizon-APM-verification} and the parabolic PDE estimates \cite{DTSB}*{Theorem 1}, for any $p>d+2$ and $R>0$, the solution of \cref{FHAP-hjb1} satisfies
\begin{equation}\label{FHestm}
\|\psi_n\|_{\Sob^{1,2,p}((0,T)\times\sB_R\times\mathbb{S})} \le \tilde{\kappa}_1 (1+\tilde{\kappa}+\|c_n\|_{\infty}|\sB_{2R}|^{\frac{1}{p}}+\|c_{_{T}}\|_\infty). 
\end{equation}
for some $\tilde{\kappa}_1,\, \tilde{\kappa}>0$.
Thus, from \eqref{FHestm}, we obtain
\begin{equation}\label{FHapestm}
\|\psi_n\|_{\Sob^{1,2,p}((0,T)\times\sB_R\times\mathbb{S})} \le \tilde{\kappa}_3
\end{equation}
for some positive constant $\tilde{\kappa}_3$ independent of $n$. Since $\Sob^{1,2,p}((0,T)\times\sB_R)$ is a reflexive Banach space, it follows that $\Sob^{1,2,p}((0,T)\times\sB_R\times\mathbb{S})$ is also a reflexive Banach space, in view of \cref{FHapestm}, compact embedding $\Sob^{1,2,p} \hookrightarrow \Sob^{0,1,p}$ and by the arguments as in \cref{{ETC1.3B}}-\cref{ETC1.3BC} there exists $\bar{\psi} \in \Sobl^{1,2,p}((0,T)\times\Rd\times \mathbb{S})$ such that (along a subsequence, denoted by the same sequence)
\begin{equation}\label{FHcgs}
\begin{aligned}
\psi_n &\to \bar{\psi} \quad \text{in } \Sobl^{1,2,p}((0,T)\times\Rd\times \mathbb{S})\quad\text{(weakly)} \\
\psi_n &\to \bar{\psi} \quad \text{in } \Sobl^{0,1,p}((0,T)\times\Rd\times \mathbb{S})\quad\text{(strongly)}.
\end{aligned}
\end{equation}
Multiplying both sides of \eqref{FHAP-hjb} by a test function $\varphi \in \cC_c^\infty((0,T)\times \Rd\times \mathbb{S})$ and integrating, we obtain
\begin{equation}\label{FH3.5}
\int_0^T \int_{\Rd} \partial_t\psi_n(t,x,i)\,\varphi(t,x,i)\,dxdt + 
\int_0^T \int_{\Rd} \inf_{\zeta\in \Act} \big[\sL_\zeta^n \psi_n + c_n(x,i,\zeta)\big]\varphi(t,x,i)\,dxdt = 0.
\end{equation}
Thus, in view of  \cref{FHcgs} and by letting $n\to\infty$,
from \cref{FH3.5}, it follows 
(by arguments analogous to those in \cref{ETC1.3BC}--\cref{ETC1.3E}) that 
$\bar{\psi}$ satisfies the limiting HJB equation

\begin{equation*}\label{abc}
\begin{aligned}
\partial_t \bar{\psi}(t,x,i) + \inf_{\zeta \in \Act}\big[\sL_\zeta \bar{\psi}(t,x,i) + c(x,i,\zeta)\big] &= 0,\quad\forall\, (t,x,i)\in (0,T)\times \Rd\times \mathbb{S} \\
\bar{\psi}(T,x,i) &= c_{_{T}}(x,i),\quad\forall\, (x,i)\in \Rd\times \mathbb{S} 
\end{aligned}
\end{equation*}
Since the finite-horizon HJB equation admits a unique solution, we conclude that 
$\bar{\psi}(t,x,i)=\psi(t,x,i)$ for all $(t,x,i)\in [0,T]\times\Rd\times\mathbb{S}$, 
and in particular,
\[
\bar{\psi}(0,x,i)=\psi(0,x,i)=\cJ^*_T(x,i,c).
\]
\end{proof}
Next, we prove the robustness result, i.e., if we use the optimal control policy of the approximating model in the true model, as the approximating model approaches the true model, the associated finite horizon cost converges to the optimal value of the true model.
\begin{theorem}[Robustness for Finite Horizon]\label{thm:finite-horizon-robust}
Suppose Assumptions \hyperlink{A1}{{(A1)}}-\hyperlink{A6}{{(A6)}}  hold. Then for any optimal control $v_n^*$ of the approximating models, we have
\[
\lim_{n\to\infty} \cJ_T^{v_n^*}(x,i,c) = \cJ_T^*(x,i,c).
\]
\end{theorem}

\begin{proof}
By the triangle inequality, we have
\begin{align*}
\big| \cJ^{v_n^*}_T(x,i,c) - \cJ^*_T(x,i,c) \big|
&\leq \big| \cJ^{v_n^*}_T(x,i,c) - \cJ^{*}_{T,n}(x,i,c_n) \big|
     + \big| \cJ^{*}_{T,n}(x,i,c_n) - \cJ^*_T(x,i,c) \big| .
\end{align*}

From the continuity result \cref{thm:finite-horizon-conv}, it is known that $
\big| \cJ^{*}_{T,n}(x,i,c_n) - \cJ^*_T(x,i,c) \big| \to 0$ as $n \to \infty.$
Next, we show that
\[
\big| \cJ^{v_n^*}_T(x,i,c) - \cJ^{v_n^*}_{T,n}(x,i,c_n) \big| \to 0 
\quad \text{as } n \to \infty.
\]
Along a subsequence $v_n^* \to \bar{v}$, follows from the compactness of $\Um$. In view of \cref{thm:finite-horizon-TM-verification}, for each $n \in \NN$ there exists a unique solution 
$\bar{\psi}_n \in \Sobl^{1,2,p}((0,T)\times \Rd\times \mathbb{S})$, 
$p > d+2$, to the following Poisson equation:
\begin{equation}\label{FH5.13}
\begin{aligned}
\partial_t \bar{\psi}_n(t,x,i) 
+ \sL_{v_n^*} \bar{\psi}_n(t,x,i) + c(x,i,v_n^*(t,x,i)) &= 0, \quad\forall\, (t,x,i)\in (0,T)\times \Rd\times \mathbb{S}\\
\bar{\psi}_n(T,x,i) &= c_{_{T}}(x,i),\quad\forall\, (x,i)\in \Rd\times \mathbb{S}
\end{aligned}
\end{equation}
This gives us for $R>0$,
\begin{equation*}\label{FHestm2}
\|\bar{\psi}_n\|_{\Sob^{1,2,p}((0,T)\times\sB_R\times\mathbb{S})} \le \tilde{\kappa}_4 (1+\tilde{\kappa}_5+\|c\|_{\infty}|\sB_{2R}|^{\frac{1}{p}}+\|c_{_{T}}\|_\infty).
\end{equation*}

Arguing as in \cref{thm:finite-horizon-conv}, letting $n \to \infty$ from \cref{FH5.13}, 
we deduce that there exists $\hat{\psi} \in \Sobl^{1,2,p}((0,T)\times \Rd\times \mathbb{S})$ with 
$p > d+2$, satisfying
\begin{equation}\label{abcd}
\begin{aligned}
\partial_t \hat{\psi}(t,x,i) +\sL_{\bar{v}} \hat{\psi}(t,x,i) + c(x,i,\bar{v}(t,x,i))\big] &= 0,\quad\forall\, (t,x,i)\in (0,T)\times \Rd\times \mathbb{S} \\
\hat{\psi}(T,x,i) &= c_{_{T}}(x,i),\quad\forall\, (x,i)\in \Rd\times \mathbb{S} 
\end{aligned}
\end{equation}

Now using \cref{abcd}, and applying the It$\hat{\rm o}$–Krylov formula (\cite{ABG-book}*{Lemma~5.1.4}), we deduce that
\begin{align}\label{FH5.17}
\hat{\psi}(t,x,i) 
= \Exp_{x,i}^{\bar{v}}\!\left[ \int_t^T c(X_s,S_s, \bar{v}(s,X_s,S_s))\,ds + c_{_{T}}(X_T,S_T) \right]. 
\end{align}

Moreover, we have
\begin{equation*}
\begin{aligned}
\partial_t\psi_n(t,x,i) 
+ \sL_{v_n^*}^n \psi_n(t,x,i) + c_n(x,i,v_n^*(t,x,i)) &= 0, \quad\forall\, (t,x,i)\in (0,T)\times \Rd\times \mathbb{S}\\
\psi_n(T,x,i) &= c_{_{T}}(x,i),\quad\forall\, (x,i)\in \Rd\times \mathbb{S}
\end{aligned}
\end{equation*}
Letting $n \to \infty$, as in \cref{thm:finite-horizon-conv}, 
we obtain that there exists $\tilde{\psi} \in \Sobl^{1,2,p}((0,T)\times \Rd\times \mathbb{S})$, 
$p > d+2$, satisfying
\begin{equation}\label{abcde}
\begin{aligned}
\partial_t \tilde{\psi}(t,x,i) + \sL_{\bar v}  \tilde{\psi}(t,x,i) + c(x,i,{\bar{v}(t,x,i)})\big] &= 0,\quad\forall\, (t,x,i)\in (0,T)\times \Rd\times \mathbb{S} \\
 \tilde{\psi}(T,x,i) &= c_{_{T}}(x,i),\quad\forall\, (x,i)\in \Rd\times \mathbb{S} 
\end{aligned}
\end{equation}

Now by the It$\hat{\rm o}$–Krylov formula (\cite{ABG-book}*{Lemma~5.1.4}) from \cref{abcde} it follows that
\begin{align}\label{FH5.21}
\tilde{\psi}(t,x,i) 
= \Exp_{x,i}^{\bar{v}}\!\left[ \int_t^T c(X_s,S_s, \bar{v}(s,X_s,S_s))\,ds + c_{_{T}}(X_T,S_T) \right].
\end{align}
Hence, from \cref{FH5.17} and \cref{FH5.21}, we conclude that 
$\cJ_T^{v_n^*}(x,i,c) = \bar{\psi}_n(0,x,i),$ and $\cJ^{v_n^*}_{T,n}(x,i,c_n) = \psi_n(0,x,i)$
converge to the same limit. This completes the proof.
\end{proof}

\section{Control up to an Exit Time}
In this section, we consider an optimal control problem for the exit time cost. We will assume that $\beta \in \cC(\bar{\mathcal{O}}\times\mathbb{S}\times\Act)$ and $h \in \Sob^{2,p}(\mathcal{O}\times\mathbb{S})$. Following the steps as in \cite{VSB2005}*{p. 228-229}, one can derive the following HJB equation in this setting:
\[
\min_{\zeta \in \Act} \left[\sL_\zeta \varphi(x,i) - \beta(x,i,\zeta)\,\varphi(x,i) + c(x,i,\zeta)\right] = 0, 
\quad (x,i) \in \mathcal{O}\times\mathbb{S}, 
\qquad \varphi = h \text{ on } \partial \mathcal{O}\times\mathbb{S}.
\]
In the next theorem, we provide the complete characterization of the optimal policies for the exit time control problem. 
\begin{theorem}\label{TH_ettm}
    Suppose that Assumptions \hyperlink{A1}{{(A1)}}-\hyperlink{A5}{{(A5)}} hold. Then the HJB equation associated with the exit-time control problem,
    \begin{equation}\label{exit-time-hjb}
     \begin{aligned}
    \min_{\zeta \in \Act} \left[\sL_\zeta \varphi(x,i) - \beta(x,i,\zeta)\,\varphi(x,i) + c(x,i,\zeta)\right] &= 0, 
\quad (x,i) \in \mathcal{O}\times\mathbb{S}, \\
 \varphi &= h \quad\text{ on } \partial \mathcal{O}\times\mathbb{S}.
      \end{aligned}
 \end{equation}
admits a unique solution $\hat{\cJ}^*_{e} \in\Sob^{2,p}(\mathcal{O}\times\mathbb{S})\cap\mathcal{C}(\bar{\mathcal{O}}\times\mathbb{S}),\, p\geq d$. Moreover, 
\begin{itemize}
    \item[(1)] \(\hat{\cJ}^*_{e}\) is the optimal exit time cost.
    \item[(2)]  a stationary markov policy $v^*\in \Usm$ is exit time optimal control if and only if it is a pointwise minimizer in \cref{exit-time-hjb}, i.e.,
  \begin{align}\label{ettmmin}
&b(x,i,v^*(x,i))\cdot \grad \hat{\cJ}^*_{e}(x,i) \,+\, \sum_{j \in \mathbb{S}} m_{ij}(x,v^*(x,i))\hat{\cJ}^*_{e}(x,j)- \beta(x,i,v^*(x,i))\,\hat{\cJ}^*_{e}(x,i)+ c(x,i,v^*(x,i))\nonumber\\ &= \min_{\zeta\in \Act}\Big[ b(x,i,\zeta)\cdot \grad \hat{\cJ}^*_{e}(x,i) \,+\, \sum_{j \in \mathbb{S}} m_{ij}(x,\zeta) \hat{\cJ}^*_{e}(x,j)- \beta(x,i,\zeta)\,\hat{\cJ}^*_{e}(x,i)\nonumber\\&\qquad\qquad\qquad\qquad\qquad\qquad\qquad\quad+ c(x,i,\zeta)\Big]\, \quad\text{a.e.}\,\,\, (x,i)\in\mathcal{O}\times \mathbb{S}\,.
\end{align}

\end{itemize}
 \end{theorem}
\begin{proof} Let $v\in \Usm$ such that \cref{ettmmin} holds. For $h \in \Sob^{2,p}(\mathcal{O}\times\mathbb{S})$, consider the Dirichlet problem
\begin{equation}\label{etdp}
    \begin{aligned}
        \min_{\zeta \in \Act} \left[\sL_\zeta \psi(x,i) - \beta(x,i,\zeta)\,\psi(x,i) +f(x,i)\right] 
&= 0 , \quad (x,i) \in\mathcal{O}\times\mathbb{S}, \\
 \psi
&= 0 \quad\text{ on } \partial \mathcal{O}\times\mathbb{S}.
    \end{aligned}
\end{equation}
where $f(x,i)= \sL_\zeta h(x,i) - \beta(x,i,\zeta)\,h(x,i) + c(x,i,\zeta)$

 By \cite{AS}*{Theorem 2.1}, \cref{etdp} has a unique solution $\psi\in \Sob^{2,p}(\mathcal{O}\times\mathbb{S})\cap\mathcal{C}_0(\bar{\mathcal{O}}\times\mathbb{S})$, by taking $\psi=\varphi-h$ in \cref{etdp}, it ensures that \cref{exit-time-hjb} has a unique solution $\varphi\in \Sob^{2,p}(\mathcal{O}\times\mathbb{S})\cap\mathcal{C}(\bar{\mathcal{O}}\times\mathbb{S})$ and using It$\hat{\rm o}$–Krylov formula (\cite{ABG-book}*{Lemma~5.1.4}) (as in  \cref{ETC1.3E}-\cref{ETC1.3FC}) we have 
 \begin{align}\label{etdps}
 \varphi(x,i)= \Exp_{x,i}^v\bigg[\int_0^{\uptau(\mathcal{O})} &e^{-\int_0^t \beta(X_s,S_s,v(X_s,S_s))\,ds}\, c(X_t,S_t,v(X_t,S_t))\,dt\nonumber\\
        &+ e^{-\int_0^{\uptau(\mathcal{O})} \beta(X_s,S_s,v(X_s,S_s))\,ds}\, h(X_{\uptau(\mathcal{O})},S_{\uptau(\mathcal{O})})
    \bigg]
 \end{align}
 Again for any $U\in \Uadm$, from \cref{exit-time-hjb} using It$\hat{\rm o}$–Krylov formula (\cite{ABG-book}*{Lemma~5.1.4}) we get,
 \begin{equation}\label{5.4}
      \varphi(x,i)\le \Exp_{x,i}^U\!\left[
        \int_0^{\uptau(\mathcal{O})} e^{-\int_0^t \beta(X_s,S_s,U_s)\,ds}\, c(X_t,S_t,U_t)\,dt
        + e^{-\int_0^{\uptau(\mathcal{O})} \beta(X_s,S_s,U_s)\,ds}\, h(X_{\uptau(\mathcal{O})},S_{\uptau(\mathcal{O})})
    \right]
 \end{equation}
Now from the definition of the optimal value, \cref{etdps} and taking infimum over all $U\in \Uadm$ in \cref{5.4} we obtain, $\varphi=\hat{\cJ}^*_{e}$. This completes the proof of  \((1)\) and the sufficiency part of \((2)\)\\
Next, let  $v^*\in \Usm$ be an optimal policy. Then for $(x,i)\in \Rd\times\mathbb{S}$,
\begin{align*}
    \hat{\cJ}^{*}_e(x,i) = \Exp_{x,i}^{v^*}\!\bigg[
    \int_0^{\uptau(\mathcal{O})} &e^{-\int_0^t \beta(X_s,S_s,v^*(X_s,S_s))\,ds}\, c(X_t,S_t,v^*(X_t,S_t))\,dt
        \\&+ e^{-\int_0^{\uptau(\mathcal{O})} \beta(X_s,S_s,v^*(X_s,S_s))\,ds}\, h(X_{\uptau(\mathcal{O})},S_{\uptau(\mathcal{O})})\bigg]
\end{align*}
from \((1)
\,\,\hat{\cJ}^{*}_e\) is the unique solution in \(\Sob^{2,p}(\mathcal{O}\times\mathbb{S})\cap\mathcal{C}(\bar{\mathcal{O}}\times\mathbb{S})\) for any \(p\geq d\)  of the following Dirichlet problem
\begin{equation}\label{et5.6}
\begin{aligned}
     \sL_{v^*} \varphi(x,i) - \beta(x,i,v^*(x,i))\,\varphi(x,i) + c(x,i,v^*(x,i)) &= 0, 
\quad (x,i) \in \mathcal{O}\times\mathbb{S}
 ,\\ \varphi &= h \quad\text{ on } \partial \mathcal{O}\times\mathbb{S}.
 \end{aligned}
\end{equation}
By comparing \cref{exit-time-hjb} and \cref{et5.6}, we obtain
\begin{align*}
0=&\sL_{v^*} \hat{\cJ}^{*}_e(x,i) - \beta(x,i,v^*(x,i))\,\hat{\cJ}^{*}_e(x,i) + c(x,i,v^*(x,i))\\
&\geq \min_{\zeta \in \Act} \left[\sL_\zeta \hat{\cJ}^{*}_e(x,i) - \beta(x,i,\zeta)\,\hat{\cJ}^{*}_e(x,i) + c(x,i,\zeta)\right]= 0 
\end{align*} 
Hence, $v^*$ is a pointwise minimizer in \cref{exit-time-hjb}. 
\end{proof}
As in \cref{TH_ettm}, following above, for each approximating model we have the following complete characterization of optimal policies in $\Usm$.

\begin{theorem}\label{etapm}
     Suppose assumption \hyperlink{A6}{{(A6)}}(iii) hold. Then for each $n\in \NN$ the HJB equation
\begin{equation}\label{exit-time-hjb-apm}
\begin{aligned}
    \min_{\zeta \in \Act} \left[\sL_\zeta ^n\varphi_n(x,i) - \beta(x,i,\zeta)\,\varphi_n(x,i) + c_n(x,i,\zeta)\right] &= 0, 
\quad (x,i) \in \mathcal{O}\times\mathbb{S}, \\
 \varphi_n &= h \quad\text{ on } \partial \mathcal{O}\times\mathbb{S}.
\end{aligned}
\end{equation}
admits a unique solution $\hat{\cJ}^*_{e,n} \in\Sob^{2,p}(\mathcal{O}\times\mathbb{S})\cap\mathcal{C}(\bar{\mathcal{O}}\times\mathbb{S})$. Moreover, \(\hat{\cJ}^*_{e,n}\) is the optimal exit time cost, and $ v^*_n\in \Usm$ is exit time optimal control if and only if it is a pointwise minimizer in \cref{exit-time-hjb-apm}, i.e.,
\begin{align*}\label{etapmin}
&b_n(x,i,v^*_n(x,i))\cdot \grad \hat{\cJ}^*_{e,n}(x,i) \,+\, \sum_{j \in \mathbb{S}} m_{ij}^n(x,v^*_n(x,i))\hat{\cJ}^*_{e,n}(x,j)\nonumber\\
&\qquad- \beta(x,i,v^*_n(x,i))\,\hat{\cJ}^*_{e,n}(x,i)+ c_n(x,i,v^*_n(x,i))\nonumber\\ &= \min_{\zeta\in \Act}\bigg[ b_n(x,i, \zeta)\cdot \grad \hat{\cJ}^*_{e,n}(x,i)+ \sum_{j \in \mathbb{S}} m_{ij}^n(x,\zeta) \hat{\cJ}^*_{e,n}(x,j)- \beta(x,i,\zeta)\,\hat{\cJ}^*_{e,n}(x,i)+ c_n(x,i,\zeta)\bigg]\, 
\end{align*}
$\text{a.e.}\,\,\, (x,i)\in\mathcal{O}\times \mathbb{S}\,.$
\end{theorem}
By closely following the proof of \cref{TH2.3}, one can obtain the following continuity result. 

\begin{theorem}\label{thm:exit-continuity}
Suppose Assumptions \hyperlink{A1}{{(A1)}}-\hyperlink{A6}{{(A6)}} hold. Then
\[
\lim_{n\to\infty} \hat{\cJ}^*_{e,n}(x,i) = \hat{\cJ}^*_e(x,i), \quad \forall (x,i) \in \bar{\mathcal{O}}\times\mathbb{S}.
\]
\end{theorem}
Similarly for each $n\in\NN$, let ${v}^*_{n} \in \Usm$ and 
${v}^* \in \Usm$ denote the optimal stationary Markov controls corresponding to the approximating and true models, respectively. Then, in view of \cref{thm:exit-continuity}, 
and by following the argument of \cref{TH2.4}, we obtain the following robustness result.

\begin{theorem}\label{thm:exit-robustness}
Suppose Assumptions \hyperlink{A1}{{(A1)}}-\hyperlink{A6}{{(A6)}} hold. Then
\[
\lim_{n\to\infty} \hat{\cJ}^{{v}^*_{n}}_e(x,i) = \hat{\cJ}^{{v}^*}_e(x,i), \quad \forall (x,i) \in \bar{\mathcal{O}}\times\mathbb{S}.
\]
\end{theorem}
\begin{proof}
The proof follows the same framework as that of the discounted case.  
By the triangle inequality, we have
\[
\big|\hat{\cJ}^{v_n^*}_e(x,i) - \hat{\cJ}^{v^*}_e(x,i)\big|
\le \big|\hat{\cJ}^{v_n^*}_e(x,i) - \hat{\cJ}^{v^*}_{e,n}(x,i)\big|
   + \big|\hat{\cJ}^{v^*}_{e,n}(x,i) - \hat{\cJ}^{v^*}_e(x,i)\big|.
\]
The second term \(\big|\hat{\cJ}^{v^*}_{e,n}(x,i) - \hat{\cJ}^{v^*}_e(x,i)\big|\to 0\) as \(n\to\infty\) by the continuity result (\cref{thm:exit-continuity}).  
For the first term, note that for each \(n\in\NN\), the function \(\hat{\cJ}^{v_n^*}_e\) satisfies the fixed-policy elliptic equation on \(\mathcal{O}\) with boundary data \(h\). Standard elliptic estimates imply uniform boundedness of \(\|\hat{\cJ}^{v_n^*}_e\|_{\Sob^{2,p}(\mathcal{O}\times\mathbb{S})}\), independent of \(n\).  
By the Banach–Alaoglu theorem and standard diagonalization argument, there exists a subsequence $\hat{\cJ}^{v_{n_{k}}^*}_e$ and a limit function \(\bar{\cJ}_e\) such that
\[
\hat{\cJ}^{v_{n_{k}}^*}_e \to \bar{\cJ}_e \text{ in } \Sob^{2,p} (\mathcal{O}\times\mathbb{S})\quad\text{weakly}, \qquad
\hat{\cJ}^{v_{n_{k}}^*}_e \to \bar{\cJ}_e \text{ in } \cC^{1,\beta}(\mathcal{O}\times\mathbb{S})\quad \text{strongly}
\]
Since the space of stationary Markov controls \(\Usm\) is compact, we may also assume \(v_{n_{k}}^*\to \hat{v}^*\) in \(\Usm\).  

Next, multiply the Dirichlet equation satisfied by \(\hat{\cJ}^{v_n^*}_e\) by a test function 
\(\phi\in \cC_c^\infty(\mathcal{O}\times\mathbb{S})\) and integrate over \(\mathcal{O}\times\mathbb{S}\).  
Passing to the limit as \(n\to\infty\), using the 
the strong convergence  of \(\hat{\cJ}^{v_n^*}_e\) in \(\cC^{1,\beta}\), and the convergence \(v_n^*\to\hat{v}^*\) in \(\Usm\), we obtain
 \(\bar{\cJ}_e\) satisfies
\[\sL_{\hat{v}^*} \bar{\cJ}_e(x,i)
   - \beta(x,i,\hat{v}^*(x,i))\bar{\cJ}_e(x,i)
   + c(x,i,\hat{v}^*(x,i)) = 0\quad in\;\; \mathcal{O}\times\mathbb{S},
   \qquad \bar{\cJ}_e = h \text{ on } \partial\mathcal{O}\times\mathbb{S}.
\]
Let $\cJ_{e,n_k}$ denote the HJB (Dirichlet) solution corresponding to the selector $v_{n_k}^*$ for the $n_k$-th approximating model. 
By the same uniform elliptic estimates, Banach--Alaoglu compactness and a diagonal argument used above, there exists a further subsequence (still indexed by $n_k$) and a limit
\[
\tilde{\cJ}_e \in \Sob^{2,p}(\mathcal{O}\times\mathbb{S})\cap\cC(\overline{\mathcal{O}}\times\mathbb{S})
\]
such that
\[
\cJ_{e,n_k}\to \tilde{\cJ}_e \ \text{in }\Sob^{2,p}
(\mathcal{O}\times\mathbb{S})\quad
    \text{weakly}, \qquad 
\cJ_{e,n_k}\to \tilde{\cJ}_e \ \text{in }\cC^{1,\beta}_{\mathrm{loc}}(\mathcal{O}\times\mathbb{S})\quad
    \text{strongly}
\]

As above, multiplying the HJB equation satisfied by \(\cJ_{e,n_k}\) by a test function 
\(\phi \in \cC_c^\infty(\mathcal{O}\times\mathbb{S})\) and integrating over \(\mathcal{O}\times\mathbb{S}\), 
then passing to the limit as \(k \to \infty\) using 
the strong convergence in \(\cC^{1,\beta}\), and the convergence \(v_{n_k}^* \to \hat{v}^*\) in \(\Usm\), 
it follows that \(\tilde{\cJ}_e\) satisfies the limiting HJB equation with boundary data \(h\):
\[
\sL_{\hat{v}^*}\tilde{\cJ}_e(x,i)
   - \beta(x,i,\hat{v}^*(x,i))\,\tilde{\cJ}_e(x,i)
   + c(x,i,\hat{v}^*(x,i)) = 0\quad in\;\; \mathcal{O}\times\mathbb{S},
   \qquad \tilde{\cJ}_e = h \text{ on } \partial\mathcal{O}\times\mathbb{S}.
\]  
Applying the It$\hat{\rm o}$–Krylov formula (\cite{ABG-book}*{Lemma~5.1.4}) to both the fixed–policy limit $\bar{\cJ}_e$ (obtained from $\hat{\cJ}_e^{v_{n_k}^*}$) and to $\tilde{\cJ}_e$ yields the same stochastic representation under the same limiting control $\hat v^*$. Hence
\[
\bar{\cJ}_e=\tilde{\cJ}_e,
\]
so the two limits coincide and the subsequential limits are unique.

\end{proof}
\begin{corollary}
    Let $v_n^\epsilon$ be an $\epsilon$-optimal control for the $n$th approximating model  for any of the four costs discussed in \cref{Cost} and \(\cJ_n(x,i,c_n)\) be the associated cost functional, then there exists \(N\in\NN\) such that \(\forall\;n\ge N,\;v_n^\epsilon\) is an $3\epsilon$-optimal for the true model. i.e.,
    \[|\cJ^{v_n^\epsilon}(x,i,c)-\cJ^*(x,i,c)|\le 3\epsilon,\qquad\forall n\ge N\]
\end{corollary}
\begin{proof}
    Given \(\epsilon>0\), there exist \(N_1,N_2,N_3\in\NN\) such that \begin{itemize}
        \item[(i)]  \(|\cJ^{v_n^\epsilon}(x,i,c)-\cJ_n^{v_n^\epsilon}(x,i,c_n)|\le \epsilon\qquad\forall\; n\ge N_1\), (by fixing $v_n^\epsilon$ and arguing as in the proof of the continuity results, e.g., \cref{TH2.4})
   \item[(ii)] \(|\cJ_n^{v_n^\epsilon}(x,i,c_n)-\cJ_n^{*}(x,i,c_n)|\le \epsilon\qquad\forall\; n\ge N_2,\) (follows from $\epsilon$-optimality of $v_n^\epsilon$) 
   \item[(iii)]\(|\cJ_n^{*}(x,i,c_n)-\cJ^*(x,i,c)|\le \epsilon\qquad\forall\; n\ge N_3\) (follows from the continuity results) 
    \end{itemize}
    From the triangle inequality and  combining (i), (ii) and (iii), we get
    \[|\cJ^{v_n^\epsilon}(x,i,c)-\cJ^*(x,i,c)|\le 3\epsilon,\qquad\forall n\ge N\]
    where \(N=\max\{N_1,N_2,N_3\}\).
\end{proof}

\begin{corollary}[Robustness under Noise Approximation]\label{cor:noise-robustness}
Consider a modified approximating model of \cref{ASE1.1} in which the driving noise is not ideal (Brownian).  
The state dynamics are given by
\begin{equation}\label{approx-noise}
\begin{aligned}
    dX_t^n&=b(X_t^n,S_t^n,U_t)dt+\upsigma(X_t^n,S_t^n)dY_t^n\\
    \mathbb{P}( S_{t+\Delta}^n &=j\;|\;S_t^n=i,\; X_s^n, U_s,\;\;s\le t) = m_{ij}\Delta t\;+\;o(\Delta t),\qquad j\ne i
\end{aligned}
\end{equation}
where the noise process \(\{Y_t^n\}\) is approximated by an Itô process of the form
\begin{align}\label{eq:noise}
    dY_t^n = \hat b_n(X_t^n,S_t^n,U_t)\,dt + \hat \upsigma_n(X_t^n,S_t^n)\,dW_t,
\end{align}
with coefficients satisfying
\(
\hat b_n \to 0,\; 
\text{and}\;\; 
\hat \upsigma_n \to I,\; 
\text{as } n \to \infty.
\)
Let \(v_n^*\) denote an optimal control for the approximating model \cref{approx-noise}-\cref{eq:noise}, and let \(\cJ^{U}(x,i,c),\;\cJ^{*}(x,i,c)\) denote, respectively the associated cost functional and the value function  for any of the four cost criteria discussed in \cref{Cost}.  
Then,
\[
\lim_{n\to\infty} \cJ^{v_n^*}(x,i,c) = \cJ^*(x,i,c).
\]
In particular, the robustness result remains valid when the system noise is approximated by an Itô process.
\end{corollary}

\section{Example}
We provide an example in which the robustness result holds under our assumptions \hyperlink{A1}{(A1)}-\hyperlink{A4}{(A4)} and \hyperlink{A6}{(A6)}.
\begin{example}
Consider the linear–quadratic (LQ) control problem associated with the systems \cref{E1.1} and \cref{ASE1.1},
 where 
\[
b(x,i,\zeta) = A(i)x + B(i)\zeta, \qquad 
\upsigma(x,i) = C(i)x,
\]
and 
\[
b_n(x,i,\zeta) = A_n(i)x + B_n(i)\zeta, \qquad 
\upsigma_n(x,i) = C_n(i)x,
\]
for \((x,i,\zeta)\in\Rd\times\mathbb{S}\times\RR^{l}\). Here \(A_n(i),C_n(i)\in \RR^{d\times d}\) and \(B_n(i)\in\RR^{d\times l}\).  The finite-horizon cost is defined through the running costs
\[
c(x,i,\zeta)=x^{\top}Q(i)x + \zeta^{\top}R(i)\zeta, \qquad 
c_n(x,i,\zeta)=x^{\top}Q_n(i)x + \zeta^{\top}R_n(i)\zeta,
\]
and the terminal cost \(c_{_T}(x,i)=x^{\top}P(i)x\),
where \(\{Q_n(i)\in\RR^{d\times d}\}_{i\in\mathbb{S}}\) are positive semidefinite symmetric matrices (\(Q_n(i)\succeq0\)), and \(\{R_n(i)\in\RR^{l\times l}\}_{i\in\mathbb{S}}\)   and \(\{P(i)\in\RR^{d\times d}\}_{i\in\mathbb{S}}\) are positive definite symmetric matrices (\(R_n(i), P(i)\succ 0\)). Assume that \(W_t\) is a one-dimensional standard Wiener process.
\subsection*{True model} The dynamics are given by
\begin{equation}\label{eq:model}
\begin{aligned}
&dX_t = (A(S_t)X_t+B(S_t)U_t)\,dt + C(S_t)X_t\,dW_t,\quad X_s=x\\
&\mathbb{P}( S_{t+\Delta t}=j\;|\;S_t=i,\; X_s, U_s,\;\;s\le t) = m_{ij}\Delta t\;+\;o(\Delta t),\qquad j\ne i
\end{aligned}
\end{equation}
where \( (s,x)\in[0,T)\times\Rd\) \text{are the initial time and state}, and \(S_t\) is a Markov chain on $\mathbb{S}$ with generator \((m_{ij})_{i,j\in\mathbb{S}}\)
 satisfying \(m_{ij}\ge0\) for \(j\ne i\) and \(\sum_{j\in\mathbb{S}}m_{ij}=0\). The performance criterion is 
\[\cJ_{T}^U(x,i,c) = \Exp_{x,i}^{U}\left[\int_0^{T} \left\{X_s^{\top}Q(S_s)X_s + U_s^{\top}R(S_s)U_s\right\} \D{s} + X_T^{\top}P(S_T)X_T\right],\,\]
and the corresponding HJB equation is given by
\begin{equation}\label{eq:hjb}
\begin{aligned}
\partial_t \psi(t,x,i) 
&+ \inf_{\zeta\in \Act}\Big\{\trace\!\big(\frac{1}{2}x^{\top}C(i)^{\top} \nabla^2 \psi(t,x,i)C(i)x\big)
+ (A(i)x+B(i)\zeta)\cdot \nabla \psi(t,x,i)
\\ &\hspace{6.5em}
+ \sum_{j\in \mathbb{S}} m_{ij}\psi(t,x,j)
+ x^{\top}Q(i)x + \zeta^{\top}R(i)\zeta
\Big\} = 0,
\\[2mm]
\psi(T,x,i)&=x^{\top}P(i)x,\qquad (x,i)\in \Rd\times\mathbb{S}.
\end{aligned}
\end{equation}
The solution of the HJB equation in this LQ problem is assumed to have a quadratic form. Let us assume the solution be \begin{equation}\label{eq:valuefn}
\psi(t,x,i)=x^{\top}K(t,i)x
\end{equation}
for some matrix \(K\in\RR^{d\times d}\). Without loss of generality assume $K$ to be a symmetric matrix. Substituting \cref{eq:valuefn} in \cref{eq:hjb} and by comparing the coefficients of $x$ we get the following system of Riccati equations for $K(t,i)$: 
\begin{align}
-\dot{K}(t,i)
&=C(i)^{\top}  K(t,i)C(i)+ A(i)^\top K(t,i) + K(t,i)A(i) 
- K(t,i)B(i) R(i)^{-1} B(i)^\top K(t,i)
\nonumber\\
&\qquad+ Q(i) 
+ \sum_{j\in\mathbb{S}} m_{ij}K(t,j),\nonumber\\
K(T,i)&=P(i)
\label{eq:riccati}
\end{align}
It can be shown (see, \cite{LiandZhou}*{Lemma 4.1}) that these equations admit a unique solution and the optimal control is given by \(v^*(t,x,i)=-R(i)^{-1}B(i)^{\top}K(t,i)x\) (see, \cite{LiandZhou}*{Corollary 3.1})
\subsection*{Approximating model} The approximating dynamics are

\begin{equation}
\begin{aligned}
&dX_t^n = (A_n(S_t^n)X_t^n+B_n(S_t^n)U_t)\,dt + C_n(S_t^n)X_t^n\,dW_t,\quad X^n_s=x\\
&\mathbb{P}( S_{t+\Delta}^n =j\;|\;S_t^n=i,\; X_s^n, U_s,\;\;s\le t) = m_{ij}^n\Delta t\;+\;o(\Delta t),\qquad j\ne i
\end{aligned}
\end{equation}
with the cost \(\cJ_{T,n}^U(x,i,c_n) = \Exp_{x,i}^{U}\left[\int_0^{T} \left\{(X_{s}^n)^{\top}Q_n(S_s^n)X_s^n + U_s^{\top}R_n(S_s^n)U_s\right\} \D{s} + (X_T^n)^{\top}P(S_T^n)X_T^n\right]\)
and HJB equation
\begin{equation}\label{eq:aphjb}
\begin{aligned}
\partial_t \psi_n(t,x,i) 
&+ \inf_{\zeta\in \Act}\Big\{\trace\!\big(\frac{1}{2}x^{\top}C_n(i)^{\top} \nabla^2 \psi_n(t,x,i)C_n(i)x\big)
+ (A_n(i)x+B_n(i)\zeta)\cdot \nabla \psi_n(t,x,i)
\\ &\hspace{6.5em}
+ \sum_{j\in \mathbb{S}} m_{ij}^n\psi_n(t,x,j)
+ x^{\top}Q_n(i)x + \zeta^{\top}R_n(i)\zeta
\Big\} = 0,
\\[2mm]
\psi_n(T,x,i)&=x^{\top}P(i)x,\qquad (x,i)\in \Rd\times\mathbb{S}.
\end{aligned}
\end{equation}
with the solution 
\begin{equation}\label{eq:apvaluefn}
\psi_n(t,x,i)=x^{\top}K_n(t,i)x
\end{equation}
and the optimal control is \(v^*_n(t,x,i)=-R_n(i)^{-1}B_n(i)^{\top}K_n(t,i)x\) (see, \cite{LiandZhou}*{Corollary 3.1}). The
Riccati equations for $K_n(t,i)$: 
\begin{align}
-\dot{K}_n(t,i)
&=C_n(i)^{\top}  K_n(t,i)C_n(i)+ A_n(i)^\top K_n(t,i) + K_n(t,i)A_n(i) 
\nonumber\\
&\qquad- K_n(t,i)B_n(i) R_n(i)^{-1} B_n(i)^\top K_n(t,i)+ Q_n(i) 
+ \sum_{j\in\mathbb{S}} m_{ij}^nK_n(t,j),\nonumber\\
K_n(T,i)&=P(i)
\label{eq:apriccati}
\end{align}
\subsubsection*{\textbf{Continuity}}
Since 
\[
A_n(i)\to A(i),\quad B_n(i)\to B(i),\quad C_n(i)\to C(i),
\quad Q_n(i)\to Q(i),\quad R_n(i)\to R(i),\quad m_{ij}^n\to m_{ij},
\]
as \(n\to\infty\) for every \(i,j\in\mathbb{S}\),
and in view of \cref{Kbd} and \cref{Klips} stated below, the family $\{K_n(t,i)\}$ is equicontinuous and uniformly bounded. Hence, by the  Arzel\`{a}-Ascoli theorem, there exists a subsequence $K_{n_k}$ and a function \(\hat K\in \cC([0,T]\times\mathbb{S})\) such that \(K_{n_k}\to \hat K\)  uniformly on $[0,T]$.
From \cref{eq:apriccati}, for every \(t\in[0,T]\) and \(i\in\mathbb{S}\),
\begin{align*}
K_{n_k}(t,i)
&= \int_{t}^{T}\bigg[-C_{n_k}(i)^{\top}  K_{n_k}(s,i)C_{n_k}(i)-A_{n_k}(i)^{\top} K_{n_k}(s,i) - K_{n_k}(s,i)A_{n_k}(i) 
\nonumber\\
&\qquad+ K_{n_k}(s,i)B_{n_k}(i) R_{n_k}(i)^{-1} B_{n_k}(i)^{\top} K_{n_k}(s,i)- Q_{n_k}(i)
-\sum_{j\in\mathbb{S}}m_{ij}^{n_k}K_{n_k}(s,j)\bigg] ds \;+\;P(i).
\end{align*}
Letting $k\to \infty$ gives
\begin{align*}
\hat K(t,i)
&= \int_{t}^{T}\bigg[-C(i)^{\top} \hat {K}(s,i)C(i)-A(i)^{\top} \hat K(s,i) -\hat K(s,i)A(i) 
+\hat K(s,i)B(i) R(i)^{-1} B(i)^{\top}\hat K(s,i)\nonumber\\
&\qquad\qquad\qquad- Q(i)
-\sum_{j\in\mathbb{S}}m_{ij}\hat K(s,j)\bigg] ds\;\;+\;\;P(i)
\end{align*}
Equivalently,
\begin{align*}
\dot {\hat {K}}(t,i)
&= -C(i)^{\top}  \hat {K}(t,i)C(i)-A(i)^{\top} \hat K(t,i) -\hat K(t,i)A(i) 
+\hat K(t,i)B(i) R(i)^{-1} B(i)^{\top}\hat K(t,i)\nonumber\\
&\qquad- Q(i)
-\sum_{j\in\mathbb{S}}m_{ij}\hat K(t,j)
\end{align*}
By the uniqueness of the solution of Riccati equations, we conclude $\hat K=K$. Therefore \(K_n\to K\) uniformly on \([0,T]\).
\subsubsection*{\textbf{ Robustness}}
Since \(\|R_n(i)-R(i)\|,\|B_n(i)-B(i)\|\; \text{and}\;\|K_n(\cdot,i)-K(\cdot,i)\|\to 0\) as \(n\to \infty\), it follows that \(|v_n^*(t,x,i)-v^*(t,x,i)|=|R_n(i)^{-1}B_n(i)^{\top}K_n(t,i)x-R(i)^{-1}B(i)^{\top}K(t,i)x|\to 0\). Thus, by the dominated convergence theorem \[|\cJ_{T}^{v_n^*}(x,i,c)-\cJ_{T}^{*}(x,i,c)|\to 0\]


\end{example}

\begin{lemma}[Positive definiteness of the Riccati solution]\label{lem:posdef}
Let $\{K_n(t,i)\}_{i\in\mathbb{S}}$ solve the coupled Riccati systems \cref{eq:apriccati} associated with the HJB equations \eqref{eq:aphjb}  on $[0,T]$, with $Q_n(i)\succeq0$ and $R_n(i),\,P(i)\succ0$. Then for every $i\in\mathbb{S}$ and $t\in[0,T)$, $K_n(t,i)$ is positive definite.
\end{lemma}

\begin{proof}
Fix $t\in[0,T)$ and $i\in\mathbb{S}$. Let $(X^n_{(\cdot)},S_{(\cdot)}^n)$ be the optimal state trajectory and $v^*_n(\cdot)$ an optimal control
\(
v_n^*(t)= -R_n(i)^{-1} B_n(i)^\top K_n(t,S_t)\,X_t^n
\). Define the value function $\psi_n(t,x,i)=x^\top K_n(t,i)x$. By the Ito formula, we have the representation
\begin{equation}\label{eq:value-rep}
\psi_n(t,x,i)
=\Exp_{x,i}^{v^*}\!\left[\int_t^T \bigl((X_s^n)^\top Q_n(S_s^n)\,X_s+v_n^*(s)^\top R_n(S_s^n)\,v_n^*(s)\bigr)\,ds
\;+\; (X_T^n)^\top P(S_t^n)\,X_T^n \right].
\end{equation}
Because $R_n{(\cdot)},P{(\cdot)}\succ0$ and $Q_n{(\cdot)}\succeq0$, the integrand is non-negative and vanishes only if $X\equiv0$ and $v^*_n\equiv0$. Hence for any $x\neq0$,
\[
x^\top K_n(t,i) x \;=\; \psi_n(t,x,i) \;>\; 0.
\]
Therefore $K_n(t,i)\succ0$ for every $t\in[0,T)$ and $i\in\mathbb{S}$.  \qedhere
\end{proof}

\begin{lemma}\label{Kbd}(A priori bounds on $\psi_n$ and $K_n$)\label{lem:apriori}
There exist constants $C_0,k_0>0$ (independent of $t,i$) such that for all $(t,x,i)\in[0,T]\times\Rd\times\mathbb{S}$,
\begin{equation}\label{eq:V-upper}
0 \;\le\; \psi_n(t,x,i)\;\le\; C_0\,e^{k_0 T} |x|^2 (T+1),
\end{equation}
and, in particular,
\begin{equation}\label{eq:K-bound}
\|K_n(t,i)\| \;\le\; C_0\,e^{k_0 T}(T+1)\,.
\end{equation}
\end{lemma}

\begin{proof}
\emph{Step 1 (upper bound for $\psi_n$).}
Consider the admissible control $U\equiv0$. The corresponding state satisfies the linear SDE
\[
dX_t^n=A_n(S_t^n)\,X_t^n\,ds + C_n(S_t^n)X_t^n\,dW_t,
\qquad X_0^n=x.
\]
By It$\hat{\rm o}$ formula and the Grönwall's inequality, there exists $k_0>0$ (depending on $\sup_{n}\sup_i\|A_n(i)\|$ and the diffusion bound) such that
\begin{equation}\label{eq:state-moment}
\mathbb E|X_t^n|^2 \;\le\;|x|^2 e^{k_0T}.
\end{equation}
Plugging $U\equiv0$ into the cost, using $Q_n{(\cdot)},P{(\cdot)}\succeq0$ and \eqref{eq:state-moment}, we obtain \eqref{eq:V-upper} with some $C_0>0$ (depending on $\sup_{n}\sup_i\|Q_n(i)\|$ and $\sup_{n}\sup_i\|P_n(i)\|$).

\emph{Step 2 (bound on $K_n$).} Since $\psi_n(t,x,i)=x^\top K_n(t,i)x$ and $K_n(t,i)=K_n(t,i)^\top\succ0$ by \cref{lem:posdef}, the norm of $K_n(t,i)$ equals:
\[
\|K_n(t,i)\|=\sup_{|y|=1} y^\top K_n(t,i)\,y \,.
\]
Fix any unit vector $y$ and set $x=a\,y$. Then
\[
a^2\,y^\top K(t,i) y \;=\; \psi_n(t,a y,i)\;\le\; C_0 e^{k_0T}|a|^2(T+1).
\]
Dividing by $a^2$, 
we get $y^\top K_n(t,i) y \le C_0 e^{k_0T}(T+1)$. Taking the supremum over unit $y$, we obtain \eqref{eq:K-bound}.  \qedhere
\end{proof}

\begin{lemma}\label{Klips}(Uniform Lipschitz continuity in time)\label{lem:lipschitz}
For each $i\in\mathbb{S}$, the maps $t\mapsto K_n(t,i)$ is Lipschitz continuous (uniformly in $i$ and $n$) on $[0,T]$.
\end{lemma}

\begin{proof}
\emph{Step 1 (Lipschitz for $\psi_n$).}
Fix $(t,x,i)$ and $\delta>0$ with $t,t+\delta\in[0,T]$. Write $\psi_n(t,x,i)$ as
\[
\psi_n(t,x,i)=\Exp_{x,i}^{v^*_n}\!\left[\int_t^{T}
\bigl((X_{s}^n)^\top Q_n(S_s^n) X_s^n + v^*_n(s){}^\top R_n(S_s^n) v^*_n(s)\bigr)\,ds
+ (X_T^n)^\top P(S_s^n)\,X_T^n\right],
\]

where $v^*_n(t)=v^*_n(t,X_t^n,S_t^n)$ is the optimal control.
By the change of variable \(s\to s+\delta\) in the above equation, we have 
\begin{align*}
\psi_n(t,x,i)=\mathbb{E}_{x,i}^{v^*_n}\bigg[\int_{t+\delta}^{T+\delta} \Big[(X_{s-\delta}^n)^{\top}Q_n({S_{s-\delta}^n})X_{s-\delta}^n &+ v_n^*(s-\delta)^{\top}R_n({S_{s-\delta}^n})v_n^*(s-\delta)\Big] ds \\
&+ (X_T^n)^{\top}P(S_T^n)X_T^n\bigg] 
\end{align*}

For \(s \in [t+\delta,T+\delta]\), let \(\tilde{X}^{n}_s = X_{s-\delta}^n\), \(\tilde{S}^{n}_s = S_{s-\delta}^n\), and \(\tilde{U}^n_s = v_n^*({s-\delta})\). Clearly, \(\tilde{S}^{n}_{t+\delta}= S_t^n = i\) and \(\tilde{X}^{n}_{t+\delta} = X_t^n = x\).
Moreover, 
\begin{align*}
\psi_n(t,x,i)&=\mathbb{E}_{x,i}^{\tilde{U}^n}\left[\int_{t+\delta}^{T+\delta} \left[(\tilde{X}^{n}_s)^{\top}Q_n({\tilde{S}^{n}_n)}\tilde{X}^{n}_s + (\tilde{U}^n_s)^{\top}R_n({\tilde{S}^{n}_s})\tilde{U}^n_s\right] ds + (\tilde{X}^{n}_{T+\delta})^{\top}P(\tilde{S}^{n}_{T+\delta})\tilde{X}^{n}_{T+\delta}\right] \\
&\geq \mathbb{E}_{x,i}^{\tilde{U}^n}\left[\int_{t+\delta}^{T} \left[(\tilde{X}^{n}_s)^{\top}Q_n({\tilde{S}^{n}_n)}\tilde{X}^{n}_s + (\tilde{U}^n_s)^{\top}R_n({\tilde{S}^{n}_s})\tilde{U}^n_s\right] ds + (\tilde{X}^{n}_{T+\delta})^{\top}P({\tilde{S}^{n}_{T+\delta}})\tilde{X}^{n}_{T+\delta}\right]
\end{align*}
By the Itô formula, we have
\[
\psi_n(t+\delta,x,i)\leq \Exp_{x,i}^{\tilde{U}^n}\left[\int_{t+\delta}^T \left[(\tilde{X}^{n}_s)^{\top}Q_n({\tilde{S}^{n}_n)}\tilde{X}^{n}_s + (\tilde{U}^n_s)^{\top}R_n({\tilde{S}^{n}_s})\tilde{U}^n_s\right] ds + (\tilde{X}^{n}_{T})^{\top}P({\tilde{S}^{n}_T})\tilde{X}^{n}_{T}\right] 
\]
Also, by applying It$\hat{\rm o}$'s formula to \((\tilde{X}_t^n)^{\top}P(\tilde{S}^{n}_t)\tilde{X}_t^n\) twice  and using \cref{eq:state-moment}, we have
\[ \Exp_{x,i}^{\tilde{U}^n}[(\tilde{X}^{n}_T)^{\top}P({\tilde{S}^{n}_T})\tilde{X}^{n}_T - (\tilde{X}^{n}_{T+\delta})^{\top}P(\tilde{S}^{n}_{T+\delta})\tilde{X}^{n}_{T+\delta}] \leq k_1|x|^2\delta. \]
for some \(k_1>0\), it follows that
\begin{equation}\label{eq:leftbd}
 \psi_n(t+\delta,x,i) - \psi_n(t,x,i) \leq k_1|x|^2\delta. \end{equation}
To show the reverse inequality. Let \(\Exp_{x,i}\) denote the conditional expectation given \((X_{s+\delta},S_{s+\delta}) = (x,i)\). Then, we have, under the optimal control \(v^*_n\) and the corresponding trajectory \((X_t^n,S_t^n)\)
\[
\psi_n(t+\delta,x,i) = \Exp_{x,i}^{v_n^*}\left[\int_{t+\delta}^T \left[(X_s^n)^{\top}Q_n(S_s^n)X_s^n + v_n^*(s)^{\top}R_n(S_s^n)v_n^*(s)\right] ds + (X_T^n)^{\top}P(S_T^n)X_T^n\right]
\]
By the change of variable \(s \rightarrow s-\delta\) in the above equation, we have
\begin{align*}
\psi_n(t+\delta,x,i) &= \Exp_{x,i}^{v_n^*}\bigg[\int_{t}^{T-\delta} \left[(X_{s+\delta}^n)^{\top}Q_n(S_{s+\delta}^n)X_{s+\delta}^n + v_n^*(s+\delta)^{\top}R_n(S_{s+\delta}^n)v_n^*(s+\delta)\right] ds \\
& \qquad+ (X_T^n)^{\top}P(S_T^n)X_T^n\bigg] 
\end{align*}
Let \(\hat{S}_s^n = S_{s+\delta}^n\), \(\hat{U}^n_s = v_n^*({s+\delta})\), and \(\hat{X}_s^n = X_{s+\delta}^n\). Note, \(\hat{S}^n(\cdot)\) is a Markov chain with \(\hat{S}_t^n = S_{t+\delta}^n = i\). We can extend the definition of \(\hat{U}^n_{(\cdot)}\) for \(s \in [T-\delta,T]\). If we define \(\hat{U}^n_{t} = 0\) for \(t \in (T-\delta,T]\), we may extend the definition of \(\hat{X}_{t}^n\), and can be defined on \((T-\delta,T]\) as well by solving \cref{eq:model}.
Rewriting the above equation, gives
\begin{equation}\label{eq7.13}
\psi_n(t+\delta,x,i) = \Exp_{x,i}^{\hat{U}^n}\left[\int_{t}^{T-\delta} \left[(\hat{X}^n_s)^{\top}Q_n(\hat{S}^n_s)\hat{X}^n_s + (\hat{U}^n_s)^{\top}R_n(\hat{S}^n_s)\hat{U}^n_s\right] ds + (\hat{X}^n_{T-\delta})^{\top}P(\hat{S}^n_{T-\delta})\hat{X}^n_{T-\delta}\right] 
\end{equation}
Also, by the It$\hat{\rm o}$ formula, we have 
\begin{equation}\label{eq7.14}
\psi_n(t,x,i) \le \Exp_{x,i}^{\hat{U}^n}\left[\int_{t}^{T} \left[(\hat{X}^n_s)^{\top}Q(\hat{S}^n_s)\hat{X}^n_s + (\hat{U}^n_s)^{\top}R(\hat{S}^n_s)\hat{U}^n_s\right] ds + (\hat{X}^n_{T})^{\top}P(\hat{S}^n_{T})\hat{X}^n_{T}\right] 
\end{equation}
Subtracting \cref{eq7.14} from \cref{eq7.13} and using \cref{eq:state-moment} gives
\begin{align} \label{eq:rightbd}
\psi_n(t+\delta,x,i)-\psi_n(t,x,i)&\ge -\Exp_{x,i}^{\hat{U}^n}\int_{T-\delta}^{T} \left[(\hat{X}^n_s)^{\top}Q_n(\hat{S}^n_s)\hat{X}^n_s + (\hat{U}^n_s)^{\top}R_n(\hat{S}^n_s)\hat{U}^n_s\right] ds \nonumber\\&\qquad+ \Exp_{x,i}^{\hat{U}^n}\left[(\hat{X}^n_{T-\delta})^{\top}P(\hat{S}^n_{T-\delta})\hat{X}^n_{T-\delta}-(\hat{X}^n_{T})^{\top}P(\hat{S}^n_{T})\hat{X}^n_{T}\right]\nonumber\\
&\ge-k_1|x|^2\delta
\end{align}
From \cref{eq:leftbd} and \cref{eq:rightbd}, we get
\begin{equation}\label{eq:V-lip}
\bigl|\psi_n(t+\delta,x,i)-\psi_n(t,x,i)\bigr| \;\le\; k_1\,|x|^2\,\delta .
\end{equation}

\emph{Step 2 (Lipschitz for $K$).}
Set $x=a\,y$ with $|y|=1$. Then
\[
a^2\, y^\top\!\bigl(K_n(t+\delta,i)-K_n(t,i)\bigr) y
= \psi_n(t+\delta,a y,i)-\psi_n(t,a y,i).
\]
Use \eqref{eq:V-lip}, divide by $a^2$, 
to obtain
\[
\bigl| y^\top\!\bigl(K_n(t+\delta,i)-K_n(t,i)\bigr) y \bigr| \;\le\; k_1\,\delta .
\]
Taking the supremum over unit $y$ yields $\|K_n(t+\delta,i)-K_n(t,i)\|\le k_1\,\delta$. Hence $K_n(\cdot,i)$ is Lipschitz on $[0,T]$. \qedhere
\end{proof}
\section{Conclusion}
In this paper, we developed a comprehensive robustness theory for controlled regime-switching diffusions under multiple cost criteria. By combining analytical and probabilistic methods, we proved the continuity of value functions and the robustness of optimal controls as the system parameters of the approximating models converge to those of the true model. The results generalize the robustness framework for diffusion processes to hybrid stochastic systems involving both continuous and discrete dynamics.

Our analysis required refined Sobolev estimates and weak compactness arguments for coupled HJB systems, together with stochastic representations based on the It$\hat{\rm o}$–Krylov formula. The framework accommodates a broad class of models and cost structures, including finite-horizon, discounted, ergodic, and exit-time formulations.

Possible directions for future work include extending the present results to non-Markovian switching, partial observation settings, and mean-field or risk-sensitive formulations. Such extensions would further strengthen the understanding of the robust control for hybrid stochastic systems in uncertain and data-driven environments.
\section*{Acknowledgment}
This research of the first author was partially supported by a Start-up Grant IISERB/ R\&D/2024-25/154 and Prime Minister Early Career Research Grant ANRF/ECRG/2024/001658/ PMS. The research of the second author was partially supported by the UGC Junior Research Fellowship (UGC-JRF), Ministry of Education, Government of India.
\appendix
\section{Verification Theorem}

For completeness, we provide the Verification Theorem for the ergodic cost. Prior to that, we need the following lemmas.

\begin{lemma}\label{3.3.4}
Under \hyperlink{A7}{{(A7)}}, all stationary Markov policies are stable, i.e., 
$\Usm=\Uadm_{\mathsf{ssm}}$.

\end{lemma} 
\begin{proof}
Following the proof technique as in \cref{TH3.3}, we obtain the estimate \cref{3.13}, which implies the stability.
\end{proof}

\begin{lemma}\label{infmeasure} 
    Let $v\in \Uadm_{\mathsf{ssm}}$ and $\eta_v$ be the corresponding invariant probability measure, then for any ball $\sB_R$ of radius $R>0$ in $\Rd$
    \[ \inf_{v\in\Uadm_{\mathsf{ssm}}}\eta_v(\sB_R\times\mathbb{S})>0\] 
\end{lemma}
\begin{proof}
    In view of  \cite{ABG-book}*{Theorem~5.3.4} \cite{FP}*{p. 560} \cite{sirakov}*{Theorem 2}, and by following the arguments of \cite{ABG-book}*{Lemma 3.2.4}, one can obtain $\inf_{v\in\Uadm_{\mathsf{ssm}}}\eta_v(\sB_R\times \mathbb{S}) > 0$ for any $R>0$.
\end{proof}



\begin{lemma}\label{3.7.2}
    \textit{Suppose Assumption \hyperlink{A5}{{(A5)}} holds and there exists a constant $k_0 > 0$, and a pair of nonnegative, inf-compact functions $(\Lyap, h) \in \cC^2(\Rd\times \mathbb{S}) \times \mathfrak{C}$ such that $1 + c \in \sorder(h)$ and
\begin{align}\label{3.7.3}
\sL_{\zeta} 
\Lyap(x,i) \leq k_0 - h(x, i, \zeta), \quad \forall\,\; (x, i, \zeta) \in \Rd \times \mathbb{S} \times  \Act
\end{align}
Then the following properties are satisfied:
\begin{enumerate}
\item for any $r > 0$,
\begin{align}\label{3.7.4}
x \mapsto \sup_{v \in \Uadm_{\mathsf{ssm}}} \mathbb{E}^v_{x,i} \left[ \int_0^{\uuptau_r} (1 + c_v(X_t,S_t)) \, dt \right] \in \sorder(\Lyap)
\end{align}
\item if $\varphi \in \sorder(\Lyap)$, then for any $(x,i) \in \Rd\times\mathbb{S}$,
\begin{align}\label{3.7.5}
\lim_{t \to \infty} \sup_{v \in \Uadm_{\mathsf{ssm}}} \frac{1}{t} \mathbb{E}^v_{x,i} [\varphi(X_t,S_t)] = 0,
\end{align}
and for all $v\in\Uadm_{\mathsf{ssm}}$ and $t\ge0$,
\begin{align}\label{3.7.6}
\lim_{R \to \infty} \mathbb{E}^v_{x,i} [\varphi(X_{t \wedge \uptau_R},S_{t \wedge \uptau_R})] = \mathbb{E}^v_{x,i} [\varphi(X_t,S_t)]
\end{align}
\end{enumerate}}
\end{lemma}
\textit{Proof.} The proof is identical to that of \cite{ABG-book}*{lemma 3.7.2}, and is therefore omitted.
\hfill $\Box$
\begin{corollary}\label{co3.7.3}
Let $v \in \Uadm_{\mathsf{ssm}}$ with $\rho_v < \infty$.  
For any fixed $r > 0$, define
\[
\varphi(x) := \Exp_{x,i}^v \!\left[ \int_0^{\uuptau_r} \bigl( 1 + c_v(X_t,S_t) \bigr)\, dt \right], 
\qquad (x,i) \in \Rd\times \mathbb{S} .
\]

Then
\[
\lim_{t \to \infty} \frac{1}{t}\, \mathbb{E}_{x,i}^v \!\big[ \varphi(X_t,S_t) \big] = 0,
\qquad \forall\, (x,i) \in \Rd\times \mathbb{S},
\]
and for any $t \geq 0$,
\[
\lim_{R \to \infty} \mathbb{E}_{x,i}^v \!\big[ \varphi(X_{t \wedge \uptau_R},S_{t \wedge \uptau_R}) \big]
= \mathbb{E}_{x,i}^v \!\big[ \varphi(X_t,S_t) \big], 
\qquad \forall\, v \in \Uadm_{\mathsf{ssm}}, \ \forall\, (x,i) \in \Rd\times \mathbb{S}.
\]
\end{corollary}

\vspace{5mm}
The following result is adapted from \cite{ABG-book}*{Theorem 5.5.2}, and the proof follows the general methodology of \cite{ABG-book}*{Lemma~3.6.3}.
\begin{theorem}\label{5.5.2}
Suppose Assumption \hyperlink{A5}{{(A5)}} holds. For each radius \( R > 0 \), there exists a constant   \( \widehat{C}_1(R) > 0 \) and a constant \( \widehat{C}_2 > 0 \) (independent of \(R\)) such that, for all \( \alpha \in (0,1) \), we have
\begin{align}
\|\cJ_\alpha^v - \cJ_\alpha^v(0,1)\|_{\Sob^{2,p}(\cB_R\times \mathbb{S})} 
&\le \widehat{C}_1, \label{A.2}\\
\sup_{(\cB_R\times \mathbb{S})} \alpha \cJ_\alpha^v 
&\le \widehat{C}_2 \label{A.3}
\end{align}
Moreover, \cref{A.2}-\cref{A.3} still holds if $\cJ_\alpha^v$ is replaced by $V_\alpha$.
\end{theorem}

\begin{proof}
Let $v\in \Uadm_{\mathsf{ssm}}$. Since $\cJ_{\alpha}^{v}$ satisfies
\[
\sL_v\psi-c_v=\alpha\psi
 \]
and the running cost $c$ is bounded, it follows that \(\alpha \cJ_{\alpha}^{v}(x,i)\leq \norm{c}_{L^{\infty}(\Rd\times\mathbb{S}\times\Act)}\). This provides a uniform upper bound on \(\alpha \cJ_{\alpha}^{v} \) throughout \(\Rd\times\mathbb{S}\), and thus inequality \cref{A.3} follows directly.

To establish the \(\Sob^{2,p}\) estimate \cref{A.2}, 
consider the Poisson equation 
\[
\sL_v\psi-c_v=\alpha\psi
 \quad \text{in}\;\;\cB_{3R}\times\mathbb{S}\]
By It$\hat{\rm o}$–Krylov formula, we have $\psi=\cJ^v_\alpha$.
Next, define the auxiliary Dirichlet problem
\begin{align*}
    \sL_v\varphi-\alpha\varphi=0\qquad\text{in}\;\cB_{3R}\times\mathbb{S},\qquad \varphi=\cJ^v_\alpha\qquad\text{on}\;\partial\cB_{3R}\times\mathbb{S}
\end{align*}
Then \(\varphi-\psi\) satisfies
\begin{equation}\label{A.7}
\sL_v(\varphi-\psi)-\alpha(\varphi-\psi)=-c_v
 \quad \text{in}\;\;\cB_{3R}\times\mathbb{S}\end{equation}
and \(\psi-\varphi\) satisfy
\begin{equation}\label{A.8}
\sL_v(\psi-\varphi)-\alpha(\psi-\varphi)=c_v\ge 0
 \quad \text{in}\;\;\cB_{3R}\times\mathbb{S}\end{equation}
Applying the Aleksandrov–Bakelman–Pucci (ABP) estimate \cite{AS}*{Theorem A.1} in \cref{A.7} and \cref{A.8} yields
\begin{equation}\label{a1}
\sup_{\cB_{3R}\times\mathbb{S}}|\varphi-\psi|\leq C|\cB_{3R}|^{\frac{1}{d}}\norm{c}_{L^\infty(\cB_{3R}\times\mathbb{S}\times\Act)}\end{equation}
for some $C>0$. 
Therefore
\begin{align}\label{osc}
\osc_{\cB_{2R}\times\mathbb{S}}\psi(x,i)&\leq \sup_{\cB_{2R}\times\mathbb{S}}(\psi(x,i)-\varphi(x,i))+\sup_{\cB_{2R}\times\mathbb{S}}(\varphi(x,i)-\psi(x,i))\nonumber\\
&\leq 2C|\cB_{2R}|^{\frac{1}{d}}\norm{c}_{L^\infty(\cB_{2R}\times\mathbb{S}\times\Act)}
\end{align}
Define \(\bar\psi:=\cJ^v_\alpha-\cJ^v_\alpha(0,1)\). Then
\(\bar\psi\) satisfies
\[
\sL_v\bar\psi-\alpha\bar\psi=-c_v+\alpha\cJ^v_\alpha(0,1)
 \quad \text{in}\;\;\cB_{2R}\times\mathbb{S}\]
 Equivalently,
\begin{align*}
\trace\!\big(a(x,i)\nabla^2 \bar{\psi}(x,i)\big)
+ b(x,i,v(x,i))\cdot\nabla\bar{\psi}(x,i)
+ (m_{ii}(x,v(x,i))-\alpha)\bar{\psi}(x,i)
= H(x,i),
\end{align*}
where
\[
H(x,i)
= -\Big[
c(x,i,v(x,i))
+ \sum_{j\ne i} m_{ij}(x,v(x,i))\bar{\psi}(x,j)
- \alpha \cJ^v_\alpha(0,1)
\Big].
\]
Using the elliptic $\Sob^{2,p}$ estimates (see \cite{GilTru}*{Theorem~9.11}) together with Harnack’s inequality \cite{sirakov}*{Theorem~2}, we obtain
\begin{align*}
\|\bar\psi\|_{\Sob^{2,p}(\cB_R\times\mathbb{S})}
&\le
C_1\bigl(
\|\bar\psi\|_{L^p(\cB_{2R}\times\mathbb{S})}
+ \|H\|_{L^p(\cB_{2R}\times\mathbb{S})}
\bigr)\\
&\le
C_2|\cB_{2R}|^{\frac{1}{p}}
\left(
\osc_{\cB_{2R}\times\mathbb{S}}\cJ^v_\alpha
+ \|c\|_{L^{\infty}(\cB_{2R}\times\mathbb{S}\times\Act)}
+|\mathbb{S}|M\osc_{\cB_{2R}\times\mathbb{S}}\cJ^v_\alpha+ \sup_{\cB_{2R}\times\mathbb{S}}\alpha \cJ^v_\alpha
\right)
\end{align*}
Finally, combining this with \cref{osc} and the uniform bound \cref{A.3},
we obtain the desired Sobolev estimate \cref{A.2}.
\end{proof}


\begin{definition}\label{A.D.1}
Let $\tilde{\Uadm}_{ssm} := \{v \in \Uadm_{\mathsf{ssm}} : \rho_v < \infty\}$. Define
\begin{align*}
\Psi^v((x,i); \rho) &:= \lim_{r \downarrow 0} \Exp^v_{x,i} \left[ \int_0^{\uuptau_r} (c_v(X_t,S_t) - \rho) dt \right], \quad v \in \tilde{\Uadm}_{ssm}, \\
\Psi^*((x,i); \rho) &:= \liminf_{r \downarrow 0} \inf_{v \in \tilde{\Uadm}_{ssm}} \Exp^v_{x,i} \left[ \int_0^{\uuptau_r} (c_v(X_t,S_t) - \rho) dt \right],
\end{align*}
provided the limits exist and are finite.
\end{definition}


The next theorem extends the classical vanishing-discount analysis for diffusions (see \cite{ABG-book}) to the regime-switching diffusion framework, where the HJB equation becomes a weakly coupled system of elliptic PDEs.
\begin{theorem}[Vanishing Discount and Ergodic HJB]\label{vd&hjb}
Assume \hyperlink{A7}{(A7)}(i) holds.  

\begin{enumerate}
\item \textbf{Vanishing Discount Approximation.}  
Suppose $\rho^* < \infty$. For any sequence $\alpha_n \downarrow 0$, there exists a subsequence (also denoted as $\alpha_n$), a function 
$
V \in \Sobl^{2,p}(\Rd \times \mathbb{S}),\;\; \rho \in \RR,$
such that, as $n \to \infty$,
\[
\alpha_n V_{\alpha_n}(0,1) \;\to\; \rho, 
\qquad 
\bar V_{\alpha_n} := V_{\alpha_n} - V_{\alpha_n}(0,1) \;\to\; V,
\]
uniformly on compact subsets of $\Rd\times\mathbb{S
}$.  
The pair $(V,\rho)$ satisfies the ergodic HJB equation
\begin{equation}\label{vd-hjb}
\min_{\zeta \in \Act} \{ \sL_\zeta V(x,i) + c(x,i,\zeta) \} = \rho, \qquad x \in \Rd,\, i\in \mathbb{S}.
\end{equation}
Moreover,
\begin{align*}
V(x,i) \;\leq\; \Psi^*((x,i); \rho),\qquad \rho \leq \rho^*
\end{align*}
and for any $r>0$,
\begin{equation}\label{3.7.42}
    V(x,i)\geq -\rho^* \limsup_{\alpha\to 0}\Exp_{x,i}^{{v}_{\alpha}
    }[\uuptau_r]-\sup_{\sB_r\times\mathbb{S}}V \quad \forall\, (x,i)\in \sB_r^{c}\times \mathbb{S}
\end{equation}
where $\{v_{\alpha}:0<\alpha<1\}\subset \Usm$ is a collection of $\alpha$-discounted optimal control.
\item \textbf{Existence of Smooth Solutions for Fixed Controls.}  
If $\hat v \in \Uadm_{\mathsf{ssm}}$ and $\rho_{\hat v} < \infty$, then there exist
$
\hat V \in \Sobl^{2,p}(\Rd\times\mathbb{S}), \;\;  p > 1,\; 
\text{and}\;\; \hat \rho \in \RR,
$
such that
\[
\sL_{\hat v} \hat V + c_{\hat v} = \hat \rho \quad \text{in } \Rd\times\mathbb{S},
\]
with
\[
\alpha \cJ^{\hat v}_\alpha(0,1) \;\to\; \hat \rho, 
\qquad 
\cJ^{\hat v}_\alpha - \cJ^{\hat v}_\alpha(0,1) \;\to\; \hat V 
\qquad \text{as}\;\;\alpha \to 0,
\]
uniformly on compact subsets, and
\(
\hat V(x,i) = \Psi^{\hat v}((x,i); \hat{\rho}),
\qquad \hat \rho = \rho_{\hat v}.
\)
\vspace{2mm}

\item \textbf{Optimality and Characterization of the Ergodic Value.}  
The HJB equation
\begin{align}\label{hjb}
\min_{\zeta \in \Act} \{ \sL_\zeta V(x,i) + c(x,i,\zeta) \} = \rho\,, 
\qquad (x,i) \in \Rd\times \mathbb{S},
\end{align}
admits a solution
\[
V^* \in \Sobl^{2,p}(\Rd\times\mathbb{S}) \cap \sorder(\Lyap), 
\qquad V^*(0,1) = 0, 
\qquad \rho = \rho^*.
\]
Furthermore, if $v^* \in \Usm$ satisfies
\begin{align}\label{3.7.52}
&b(x,i,v^*(x,i))\cdot \grad V^*(x,i)\,+\, \sum_{j \in \mathbb{S}} m_{ij}(x,v^*(x,i)) V^*(x,j) + c(x,i, v^*(x,i))\nonumber\\
&= \min_{\zeta \in \Act} \{ b(x,i,\zeta)\cdot \grad V^*(x,i)\,+\, \sum_{j \in \mathbb{S}} m_{ij}(x,\zeta) V^*(x,j) + c(x,i, \zeta) \}, 
\quad \text{a.e.}\, (x,i)\in \Rd\times \mathbb{S}
\end{align}
then
\begin{equation}\label{3.7.53}
\rho_{v^*} = \rho^* 
= \inf_{U \in \Uadm} 
\liminf_{T \to \infty} \frac{1}{T} 
\mathbb{E}^U_{x,i}\!\left[ \int_0^T c(X_t,S_t,U_t)\,dt \right].
\end{equation}
\end{enumerate}
\end{theorem}

\begin{proof}(Sketch)
\begin{enumerate}
    \item  By \cref{5.5.2}, the quantities $\{\alpha V_\alpha(0,1)\}_{\alpha\in(0,1)}$ are uniformly bounded, and $\bar{V}_\alpha := V_\alpha - V_\alpha(0,1)$ are uniformly bounded in  $\Sob^{2,p}(\cB_R\times\mathbb{S})$, for each $R>0$, 
$p>1$, independently of~$\alpha$. Starting from~\cref{2.1}, and by standard diagonalization arguments and Banach-Alaoglu theorem as in the proof of~\cref{TH2.3}, 
we deduce that, along some sequence $\alpha_n \downarrow 0$,
\(
\bar{V}_{\alpha_n} \to V \;\;\text{uniformly on compact subsets of } \Rd\times\mathbb{S},
\)
and $\alpha_n V_{\alpha_n}(0,1) \to \rho$ for some $\rho\in\RR$. 
Passing to the limit in the weak form of the discounted HJB yields that $(V,\rho)$ satisfies~\cref{vd-hjb}.
To establish $\rho \leq \rho^*$ and the bounds involving $\Psi^*$ and~\cref{3.7.42}, 
we follow the argument of~\cite{ABG-book}*{Lemma~3.7.8}, 
adapting it to the regime-switching setting.  
Specifically, for each $v\in\tilde{\Uadm}_{ssm}$ we consider a concatenated control that coincides with~$v$ until the stopping time $\uuptau_r\wedge \uptau_R$, for $r<R$, and then switches to an $\alpha$-discounted optimal control~$v_\alpha$ thereafter.  Using the strong Markov property, It$\hat{\rm o}$–Krylov formula, and letting $R\to \infty$ one can obtain $\bar{V}_\alpha $ then arguing as in ~\cite{ABG-book}*{Lemma~3.7.8} and letting $\alpha\to0$, $r\downarrow0$, 
we obtain
\[
V(x,i)\;\leq\;
\liminf_{r\downarrow0}\inf_{v\in\tilde{\Uadm}_{ssm}}
\Exp^v_{x,i}\!\left[\int_0^{\uuptau_r}(c_v(X_t,S_t)-\rho)\,dt\right],
\]
and the lower bound~\cref{3.7.42} follows similarly by taking $v=v_\alpha$. 
This completes the proof of part~(1).
\item 
Let $\hat{v}\in\Uadm_{\mathsf{ssm}}$ be such that $\rho_{\hat{v}}<\infty$.  
Using the uniform Sobolev bounds (see~\cref{A.2}), standard diagonalization arguments and Banach-Alaoglu theorem as in~\cref{TH2.3}, 
there exist subsequences $\alpha_n\to0$ and limits $\hat{V}$ and $\hat{\rho}$ such that
\[
\alpha_n \cJ^{\hat v}_{\alpha_n}(0,1)\to\hat{\rho}, 
\qquad 
\cJ^{\hat v}_{\alpha_n}-\cJ^{\hat v}_{\alpha_n}(0,1)\to\hat{V}\]
and $(\hat{V},\hat{\rho})$ satisfies 
$\sL_{\hat{v}}\hat{V}+c_{\hat{v}}=\hat{\rho}$ on $\Rd\times\mathbb{S}$.  
Following the proof of~\cite{ABG-book}*{Lemma~3.7.8(ii)}, we obtain the stochastic representation
\begin{equation}\label{3.7.50-rev}
\hat{V}(x,i) = 
\Exp^{\hat{v}}_{x,i}
\!\left[\int_0^{\uuptau_r}\!\{c_{\hat{v}}(X_t,S_t)-\hat{\rho}\}\,dt
+ \hat{V}(X_{\uuptau_r},S_{\uuptau_r})\right].
\end{equation}
Letting $r\downarrow0$ gives
\(
\hat{V}(x,i) = \Psi^{\hat v}((x,i); \hat{\rho})\)

By \cref{3.7.50-rev} and \cref{co3.7.3}, we have \(\lim_{t \to \infty} \frac{1}{t} \Exp^{\hat{v}}_{x,i}[\hat{V}(X_t,S_t)] \to 0.\) Applying It$\hat{\rm o}$-Krylov formula to $\sL_{\hat{v}}\hat{V}+c_{\hat{v}}=\hat{\rho}$, dividing by $t$, 
and letting $t\to\infty$ yields $\hat{\rho}=\rho_{\hat{v}}$.  
This completes the proof of part~(2).



\item Part~(1) guarantees the existence of a solution $(V,\rho)$ to~\cref{hjb} with $\rho\le\rho^*$. Using again the argument of~\cite{ABG-book}*{Lemma~3.7.8}, we have for all $(x,i)\in\Rd\times\mathbb{S}$,
\begin{align}\label{A.18}
    V(x,i) \leq \Exp^v_{x,i} \left[ \int_0^{\uuptau_r} c_v(X_t,S_t) dt + V(X_{\uuptau_r},S_{\uuptau_r}) \right]
\end{align}
Combining \cref{A.18} \& \cref{3.7.42} yields the growth estimate,
\begin{equation*}
|V(x,i)| \leq \sup_{v \in \Uadm_{\mathsf{ssm}}} \Exp^v_{x,i} \left[ \int_0^{\uuptau_r} \{c_v(X_t,S_t) + \rho^*\} dt \right] + \sup_{\cB_r\times \mathbb{S}} V \quad \forall (x,i) \in \cB_r^c\times \mathbb{S}.
\end{equation*}
so $V\in\sorder(\Lyap)$ by~\cref{3.7.4}.  
Now let $v^*\in\Usm$ be a minimizing selector satisfying~\cref{3.7.52}.  
Applying the It$\hat{\rm o}$-Krylov formula (\cite{ABG-book}*{Lemma~5.1.4}) to $V$ under~$v^*$ gives
\begin{align}\label{3.7.54}
\Exp^{v^*}_{x,i}[V(X_{t \wedge \uptau_R},S_{t \wedge \uptau_R})] - V(x,i) &= \Exp^{v^*}_{x,i} \left[ \int_0^{t \wedge \uptau_R} \sL_{v^*} V(X_s,S_s) ds \right] \nonumber \\
&= \Exp^{v^*}_{x,i} \left[ \int_0^{t \wedge \uptau_R} \{\rho - c_{v^*}(X_s,S_s)\} ds \right].\nonumber\\
&=\rho \Exp^{v^*}_{x,i}[t \wedge \uptau_R] - \Exp^{v^*}_{x,i} \left[ \int_0^{t \wedge \uptau_R} c_{v^*}(X_s,S_s) ds \right]
\end{align}
Letting $R\to\infty$ by applying \cref{3.7.6} and employing monotone convergence, we obtain\begin{equation}\label{3.7.55}\Exp^{v^*}_{x,i}[V(X_t,S_t)] - V(x,i) = \Exp^{v^*}_{x,i} \left[ \int_0^t \{\rho - c_{v^*}(X_s,S_s)\} ds \right] \end{equation}
Dividing by $t$, and passing to the limit as $t\to\infty$ 
(using~\cref{3.7.5}) yields $\rho_{v^*}=\rho$. Since $\rho\le\rho^*$, it follows that $\rho_{v^*}=\rho^*$.  
Finally, applying It$\hat{\rm o}$-Krylov again to~\cref{hjb} for any $U\in\Uadm$ gives~\cref{3.7.53}, 
establishing optimality of~$v^*$.
\end{enumerate}
\end{proof}

We now state a result addressing the uniqueness of solutions to the HJB equation.

\begin{theorem}\label{3.7.12}
Let $V^*$ denote the solution of \cref{hjb} obtained  as  the vanishing discount limit in \cref{vd&hjb}. Then the following statements hold:
\begin{enumerate}
\item $V^*(x,i) = \Psi^*((x,i); \rho^*) = \Psi^{\hat{v}}((x,i); \rho^*)$, for any average-cost optimal $\hat{v} \in \Uadm_{\mathsf{ssm}}$;
\item a control $\hat{v} \in \Uadm_{\mathsf{ssm}}$ is average-cost optimal if and only if it is a measurable selector from the minimizer $\min_{\zeta \in \Act}\{\sL_\zeta V^*(x,i) + c(x,i,\zeta)\}$;
\item if $(\tilde{V}, \tilde{\rho}) \in (\Sobl^{2,p}(\Rd\times\mathbb{S}) \cap \sorder(\Lyap)) \times \RR$ satisfies \cref{hjb} with $\tilde{V}(0,1)=0$, then $(\tilde{V}, \tilde{\rho}) = (V^*, \rho^*)$.
\end{enumerate}
\end{theorem}

\begin{proof}
\begin{enumerate}
    \item[(1)-(2)]
From \cref{vd&hjb}(1), $V^*$ arises as the limit of $\bar{V}_{\alpha_n}$ as $\alpha_n \to 0$, hence $V^* \le \Psi^*((x,i);\rho)$, where by \cref{vd&hjb}(3), $\rho = \rho^*$.
Let $\hat{v} \in \Uadm_{\mathsf{ssm}}$ be any average-cost optimal control, and let $\hat{V} \in \Sobl^{2,p}(\Rd\times\mathbb{S}),$ $p>1$, be the solution to the Poisson equation
\(
\sL_{\hat{v}}\hat{V} + c_{\hat{v}} = \rho_{\hat{v}}.
\)
By optimality, $\rho_{\hat{v}} = \rho^*$.  
Subtracting, we have $\sL_{\hat{v}}(V^*-\hat{V}) \ge 0$ and
\[
V^*(x,i) - \hat{V}(x,i)
\le
\Psi^*((x,i); \rho^*) - \Psi^{\hat{v}}((x,i); \rho^*) \le 0.
\]
Since $V^*(0,1)=\hat{V}(0,1)$, the strong maximum principle (\cite{GilTru}*{Theorem~9.6}) 
 implies $V^*=\hat{V}$ (see, \cref{TErgoExisPoiss1}).  
This establishes both (1) and (2).
\item[(3)]
Let $(\tilde{V}, \tilde{\rho}) \in (\Sobl^{2,p}(\Rd\times\mathbb{S}) \cap \sorder(\Lyap)) \times \RR$ satisfy \cref{hjb}, and let $\tilde{v} \in \Uadm_{\mathsf{ssm}}$ be a measurable selector attaining the minimum.  
Applying the Itô--Krylov formula (\cite{ABG-book}*{Lemma~5.1.4}) and using \cref{3.7.6}, we obtain the analogue of \cref{3.7.55}:
\[
\Exp^{\tilde{v}}_{x,i}[\tilde{V}(X_t,S_t)] - \tilde{V}(x,i)
= 
\Exp^{\tilde{v}}_{x,i}\!\left[\int_0^t (\tilde{\rho} - c_{\tilde{v}}(X_s,S_s))\,ds\right].
\]
Dividing by $t$ and letting $t\to\infty$, together with \cref{3.7.5}, gives $\rho_{\tilde{v}}=\tilde{\rho}$, hence $\rho^* \le \tilde{\rho}$.  
Conversely, applying the same formula to \cref{hjb} under an average-cost optimal control $v^*$ yields
\begin{equation}\label{3.7.56}
\Exp^{v^*}_{x,i}[\tilde{V}(X_t,S_t)] - \tilde{V}(x,i) \geq \Exp^{v^*}_{x,i} \left[ \int_0^t \{\tilde{\rho} - c_{v^*}(X_s,S_s)\} ds \right]. 
\end{equation}
Dividing by $t$, taking $t\to\infty$, and using again \cref{3.7.5}, we deduce $\tilde{\rho}\le\rho^*$.  
Thus, $\tilde{\rho} = \rho^*$.
Next, to compare $\tilde{V}$ and $V^*$, let $r<R$ be such that $r<|x|<R$.  
By applying It$\hat{\rm o}$-Krylov formula (\cite{ABG-book}*{Lemma 5.1.4}) to \cref{hjb} under $\tilde{v}$, we obtain
\begin{align}\label{3.7.57}
\tilde{V}(x,i) &= \Exp^{\tilde{v}}_{x,i} \left[ \int_0^{\uuptau_r \wedge \uptau_R} \{c_{\tilde{v}}(X_t,S_t) - \rho^*\} dt + \tilde{V}(X_{\uuptau_r},S_{\uuptau_r}) \Ind_{\{\uuptau_r < \uptau_R\}} \right. \nonumber \\
&\quad \left. + \tilde{V}(X_{\uptau_R},S_{\uptau_R}) \Ind_{\{\uuptau_r \geq \uptau_R\}} \right].
\end{align}
Using the Lyapunov inequality \cref{3.7.3} and the It$\hat{\rm o}$-Krylov formula (\cite{ABG-book}*{Lemma 5.1.4}), we obtain
\begin{equation}\label{3.7.58}
\Exp^v_{x,i}\left[\Lyap(X_{\uptau_R},S_{\uptau_R}) \Ind_{\{\uptau_R \leq \uuptau_r\}}\right] \leq k_0 \Exp^v_{x,i}[\uuptau_r] + \Lyap(x,i) \quad \forall v \in \Uadm_{\mathsf{ssm}}.
\end{equation}
Since $\tilde{V}\in\sorder(\Lyap)$, the right-hand side is finite, implying
\begin{equation*}
\sup_{v \in \Uadm_{\mathsf{ssm}}} \Exp^v_{x,i}[\tilde{V}(X_{\uptau_R},S_{\uptau_R}) \Ind_{\{\uptau_R \leq \uuptau_r\}}] \xrightarrow{R \to \infty} 0.
\end{equation*}
Letting $R\to\infty$ in \cref{3.7.57} and applying Fatou’s lemma gives
\[
\tilde{V}(x,i)\ge
\inf_{v\in\Uadm_{\mathsf{ssm}}}
\Exp^v_{x,i}\!\left[\int_0^{\uuptau_r}\!(c_v(X_t,S_t)-\rho^*)\,dt\right]
+\inf_{\cB_r\times\mathbb{S}}\tilde{V}.
\]

Sending $r\downarrow0$ and using $\tilde{V}(0,1)=0$ yields
$\tilde{V}(x,i)\ge\Psi^*((x,i);\rho^*)$.
Therefore, $V^*-\tilde{V}\le0$ and $\sL_{\tilde{v}}(V^*-\tilde{V})\ge0$.  
By the strong maximum principle (\cite{GilTru}*{Theorem~9.6}), 
we conclude $\tilde{V}=V^*$ (see, \cref{TErgoExisPoiss1}).  
This completes the proof.

\end{enumerate}
\end{proof}



\bibliography{Dinesh_Robustness}

@Article{Escobedo,
AUTHOR = {Escobedo-Trujillo, Beatris Adriana and Garrido-Meléndez, Javier and Alcalá, Gerardo and Revuelta-Acosta, J. D.},
TITLE = {Optimal Control with Partially Observed Regime Switching: Discounted and Average Payoffs},
JOURNAL = {Mathematics},
VOLUME = {10},
YEAR = {2022},
NUMBER = {12},
ARTICLE-NUMBER = {2073},
URL = {https://www.mdpi.com/2227-7390/10/12/2073},
ISSN = {2227-7390},
DOI = {10.3390/math10122073}
}

@article{LiandZhou,
author = {Li, Xun and Zhou, Xun},
year = {2003},
month = {01},
pages = {265-282},
title = {Indefinite stochastic LQ controls with Markovian jumps in a finite time horizon},
volume = {2},
journal = {Communications in Information and Systems},
doi = {10.4310/CIS.2002.v2.n3.a4}
}

@article {WLS01,
    AUTHOR = {S. Walczak and U. Ledzewicz and H.Sch\"attler},
     TITLE = {Stability of elliptic optimal control problems},
   JOURNAL = {Computers \& Mathematics with Applications},
    VOLUME = {41},
      YEAR = {2001},
    NUMBER = {10-11},
     PAGES = {1245--1256},
}

@article{SI72,
  title={Continuous dependence on parametem and boundary data for nonlinear two-point boundary value problems},
  author={S.K. Ingram},
  journal={Pacific J. Math.},
  volume={41},
  pages={395--408},
  year={1972}}

@article{arapostathis2022risk,
  title={Risk-sensitive control for a class of diffusions with jumps},
  author={Arapostathis, Ari and Biswas, Anup},
  journal={The Annals of Applied Probability},
  volume={32},
  number={6},
  pages={4106--4142},
  year={2022},
  publisher={Institute of Mathematical Statistics}
}

@book{yong2012stochastic,
  title={Stochastic Controls: Hamiltonian Systems and HJB Equations},
  author={Yong, J. and Zhou, X.Y.},
  isbn={9781461214663},
  lccn={98055411},
  series={Stochastic Modelling and Applied Probability},
  url={https://books.google.co.in/books?id=LesHCAAAQBAJ},
  year={2012},
  publisher={Springer New York}
}

@book{HP09-book,
    AUTHOR = {Pham, H.},
     TITLE = {Continuous-time stochastic control and applications with financial applications},
    SERIES = {Stochastic Modelling and Applied Probability},
    VOLUME = {61},
 PUBLISHER = {Springer},
      YEAR = {2009}}

@book {Bor-book,
    AUTHOR = {Borkar, Vivek S.},
     TITLE = {Optimal control of diffusion processes},
    SERIES = {Pitman Research Notes in Mathematics Series},
    VOLUME = {203},
 PUBLISHER = {Longman Scientific \& Technical, Harlow; copublished in the
              United States with John Wiley \& Sons, Inc., New York},
      YEAR = {1989},
     PAGES = {vi+196},
      ISBN = {0-582-03540-6},
   MRCLASS = {93E20 (49A60 49C10 60H10 60J60 93-02)},
  MRNUMBER = {1005532},
MRREVIEWER = {Anatoli\u{\i} B. Juditsky},
}

@article{Yin10,
author = {Yin, G. and Mao, Xuerong and Yuan, Chenggui and Cao, Dingzhou},
title = {Approximation Methods for Hybrid Diffusion Systems with State-Dependent Switching Processes: Numerical Algorithms and Existence and Uniqueness of Solutions},
journal = {SIAM Journal on Mathematical Analysis},
volume = {41},
number = {6},
pages = {2335-2352},
year = {2010},
doi = {10.1137/080727191},
URL = {https://doi.org/10.1137/080727191},
eprint = { https://doi.org/10.1137/080727191}}

@article{sirakov,
  TITLE = {{Some estimates and maximum principles for weakly coupled systems of elliptic PDE}},
  AUTHOR = {Sirakov, Boyan},
  URL = {https://hal.science/hal-00351672},
  JOURNAL = {{Nonlinear Analysis: Theory, Methods and Applications}},
  PUBLISHER = {{Elsevier}},
  VOLUME = {70},
  NUMBER = {8},
  PAGES = {3039-3046},
  YEAR = {2009},
  HAL_ID = {hal-00351672},
  HAL_VERSION = {v1},
}

@Article{FP,
journal={Journal of Optimization Theory and Applications},
author={A. M. Il'in and R. Z. Khasminskii and G. Yin},
title={Singularly Perturbed Switching Diffusions: Rapid Switchings and Fast Diffusions},
year={1999},
month={September},
pages={555-591},
volume={102},
number={3},
doi={10.1023/A:1022698023010},
url={https://ideas.repec.org/a/spr/joptap/v102y1999i3d10.1023_a1022698023010.html},
}

@article{FH,
title = {Dynamic programming for a Markov-switching jump–diffusion},
journal = {Journal of Computational and Applied Mathematics},
volume = {267},
pages = {1-19},
year = {2014},
issn = {0377-0427},
doi = {https://doi.org/10.1016/j.cam.2014.01.021},
url = {https://www.sciencedirect.com/science/article/pii/S0377042714000491},
author = {N. Azevedo and D. Pinheiro and G.-W. Weber}}

@article{AS,
title = {Ergodic risk-sensitive control for regime-switching diffusions},
journal = {Systems \& Control Letters},
volume = {170},
pages = {105399},
year = {2022},
issn = {0167-6911},
doi = {https://doi.org/10.1016/j.sysconle.2022.105399},
url = {https://www.sciencedirect.com/science/article/pii/S0167691122001761},
author = {Anup Biswas and Somnath Pradhan},
}

@misc{DTSB,
      title={Existence, uniqueness, and regularity of solutions to nonlinear and non-smooth parabolic obstacle problems}, 
      author={Durandard Théo and Strulovici Bruno},
      year={2024},
      eprint={2404.01498},
      archivePrefix={arXiv},
      primaryClass={math.AP},
      url={https://arxiv.org/abs/2404.01498}, }

@article{AGM93,
author = {Ghosh, Mrinal K. and Arapostathis, Aristotle and Marcus, Steven I.},
title = {Optimal Control of Switching Diffusions with Application to Flexible Manufacturing Systems},
journal = {SIAM Journal on Control and Optimization},
volume = {31},
number = {5},
pages = {1183-1204},
year = {1993},
doi = {10.1137/0331056},
URL = {https://doi.org/10.1137/0331056},
eprint = { https://doi.org/10.1137/0331056}
}

@article{PradhanYuksel2023,
author = {Pradhan, Somnath and Y\"{u}ksel, Serdar},
title = {Robustness of Stochastic Optimal Control to Approximate Diffusion Models Under Several Cost Evaluation Criteria},
journal = {Mathematics of Operations Research},
volume = {49},
number = {4},
pages = {2049-2077},
year = {2024},
doi = {10.1287/moor.2022.0134},
URL = {https://doi.org/10.1287/moor.2022.0134},
eprint = { https://doi.org/10.1287/moor.2022.0134}}

@book{FlemingSoner2006,
title={Controlled Markov Processes and Viscosity Solutions},
author={Fleming, W. H. and Soner, H. M.},
edition={2nd},
year={2006},
publisher={Springer}
}

@article{Ghosh97,
author = {Ghosh, Mrinal K. and Arapostathis, Aristotle and Marcus, Steven I.},
title = {Ergodic Control of Switching Diffusions},
journal = {SIAM Journal on Control and Optimization},
volume = {35},
number = {6},
pages = {1952-1988},
year = {1997},
doi = {10.1137/S0363012996299302},
URL = {https://doi.org/10.1137/S0363012996299302},
eprint = {https://doi.org/10.1137/S0363012996299302}
}

@article{HS08,
 ISSN = {00028282},
 URL = {http://www.jstor.org/stable/2677734},
 author = {Lars Peter Hansen and Thomas J. Sargent},
 journal = {The American Economic Review},
 number = {2},
 pages = {60--66},
 publisher = {American Economic Association},
 title = {Robust Control and Model Uncertainty},
 urldate = {2025-10-09},
 volume = {91},
 year = {2001}
}

@book{YinZhu2010,
  title={Hybrid Switching Diffusions: Properties and Applications},
  author={Yin, G. and Zhu, C.},
  year={2010},
  publisher={Springer}
}

@book{MaoYuan2006,
author = {Mao, Xuerong and Yuan, Chenggui},
title = {Stochastic Differential Equations with Markovian Switching},
publisher = {PUBLISHED BY IMPERIAL COLLEGE PRESS AND DISTRIBUTED BY WORLD SCIENTIFIC PUBLISHING CO.},
year = {2006},
doi = {10.1142/p473},
address = {},
edition   = {},
URL = {https://www.worldscientific.com/doi/abs/10.1142/p473},
eprint = {https://www.worldscientific.com/doi/pdf/10.1142/p473}
}

@book{ABG-book,
    AUTHOR = {Arapostathis, A. and Borkar, V. S. and
              Ghosh, M. K.},
     TITLE = {Ergodic control of diffusion processes},
    SERIES = {Encyclopedia of Mathematics and its Applications},
    VOLUME = {143},
 PUBLISHER = {Cambridge University Press},
   ADDRESS = {Cambridge},
      YEAR = {2012},
      MRNUMBER = {2884272}}

@book {Adams,
    AUTHOR = {Adams, R. A.},
     TITLE = {Sobolev Spaces},
 PUBLISHER = {Academic Press, New York},
      YEAR = {1975},}

@book{HB-book,
    AUTHOR = {Haim Brezis},
     TITLE = {Functional Analysis, Sobolev Spaces and Partial Differential Equations},
    SERIES = {Universitext},
 PUBLISHER = {Springer-Verlag},
   ADDRESS = {New York},
      YEAR = {2010},}

@article{BJP02,
  title={Robustness and risk-sensitive filtering},
  author={R. K. Boel and M. R. James and I. R. Petersen},
  journal={IEEE Trans. Automat. Control},
  volume={47},
  pages={451--461},
  year={2002}
}

@article{VSB2005,
 author={Vivek S. Borkar},
  title={Controlled diffusion processes},
journal={Probab. Surveys},
 volume={2},
  pages={213--244},
   year={2005},
   DOI ={10.1214/154957805100000131}}

@article{GL99,
  author={E. Gordienko and E. Lemus-Rodr\'iguez},
  title={Estimation of robustness for controlled diﬀusion processes},
  journal={Stoch. Anal. and Appl.},
  volume={17},
  number={3},
  pages={421--441},
  year={1999}}

@article{Bor89,
  title={A topology for markov controls},
  author={V. S. Borkar},
  journal={Applied Mathematics and Optimization},
  volume={20},
  pages={55--62},
  year={1989}}

@article{KV16,
  title={On robustness of discrete time optimal filters},
  author={M. L. Kleptsyna and A. Yu. Veretennikov},
  journal={Mathematical Methods of Statistics},
  volume={25},
  number={3},
  pages={207--218},
  year={2016}}

@book{KushnerDupuis2001,
  title={Numerical Methods for Stochastic Control Problems in Continuous Time},
  author={Kushner, H. and Dupuis, P.G.},
  isbn={9781461300076},
  lccn={00061267},
  series={Stochastic Modelling and Applied Probability},
  url={https://books.google.com.vc/books?id=JfbSBwAAQBAJ},
  year={2013},
  publisher={Springer New York}
}

@article{LJE15,
  title={Discounted robust control for Markov diffusion processes},
  author={J.D. L\'opez-Barrientos and H. Jasso-Fuentes and B. A. Escobedo-Trujillo},
  journal={TOP},
  volume={23},
  pages={53--76},
  year={2015}}

@article{NG05,
  title={Robust control of Markov decision processes with uncertain transition matrices},
  author={A. Nilim and L. E. Ghaoui},
  journal={Oper. Res.},
  volume={53},
  pages={780--798},
  year={2005}}

@article{SX15,
  title={Convergence analysis for distributionally robust optimization and equilibrium problems},
  author={H. Sun and H. Xu},
  journal={Math. Oper. Res.},
  volume={41},
  pages={377--401},
  year={2015}}

@article {KY-20,
    AUTHOR = {Kara, A. D. and Y\"uksel, S.},
     TITLE = {Robustness to Incorrect System Models in Stochastic Control},
   JOURNAL = {SIAM J. Control Optim.},
  FJOURNAL = {SIAM Journal on Control and Optimization},
    VOLUME = {58},
      YEAR = {2020},
    NUMBER = {2},
     PAGES = {1144--1182},
   MRCLASS = {93E20, 93E03, 93E11, 62G35},
       DOI = {10.1137/18M1208058},
       URL = {https://doi.org/10.1137/18M1208058}}

@article{KRY-20,
title = {Robustness to incorrect models and data-driven learning in average-cost optimal stochastic control},
journal = {Automatica},
volume = {139},
pages = {110179},
year = {2022},
issn = {0005-1098},
doi = {https://doi.org/10.1016/j.automatica.2022.110179},
url = {https://www.sciencedirect.com/science/article/pii/S0005109822000243},
author = {Ali Devran Kara and Maxim Raginsky and Serdar Yüksel},  
}

@book{GilTru,
	AUTHOR = {Gilbarg, David and Trudinger, Neil S.},
     TITLE = {Elliptic partial differential equations of second order},
    SERIES = {Grundlehren der Mathematischen Wissenschaften},
    VOLUME = {224},
   EDITION = {Second},
 PUBLISHER = {Springer-Verlag, Berlin},
      YEAR = {1983},
  MRNUMBER = {737190},
       DOI = {10.1007/978-3-642-61798-0}}

@article {CL89,
    AUTHOR = {Caffarelli, L.},
     TITLE = {Interior a Priori Estimates for Solutions of Fully Non-Linear Equations},
   JOURNAL = {Annals of Mathematics},
    VOLUME = {130},
      YEAR = {1989},
    NUMBER = {1},
     PAGES = {189--213},
       DOI = {10.2307/1971480}}

@article{SPSY,
author = {Somnath Pradhan and Serdar Y{\"u}ksel},
title = {{Continuity of cost in Borkar control topology and implications on discrete space and time approximations for controlled diffusions under several criteria}},
volume = {29},
journal = {Electronic Journal of Probability},
number = {none},
publisher = {Institute of Mathematical Statistics and Bernoulli Society},
pages = {1 -- 32},
year = {2024},
doi = {10.1214/24-EJP1093},
URL = {https://doi.org/10.1214/24-EJP1093}
}
\bibliographystyle{elsarticle-num} 

\end{document}